\documentclass[11pt]{article}
\pdfoutput=1
\usepackage{iftex}
\ifPDFTeX
	\usepackage{amsmath,mathtools}
	\usepackage{stix2}
	\usepackage[T1]{fontenc}
	\usepackage[utf8]{inputenc}
\input{glyphtounicode}

	\UndeclareTextCommand{\textasteriskcentered}{TS1}
	\DeclareTextSymbolDefault{\textasteriskcentered}{T1}
	\DeclareTextCommand{\textasteriskcentered}{T1}{\ensuremath{*}}

\else
	\usepackage[math-style=TeX]{unicode-math}
	\usepackage{amsmath,mathtools}
\fi

\usepackage{tikz}
\usepackage{accsupp}
\usetikzlibrary{svg.path}
\definecolor{orcidlogocol}{HTML}{A6CE39}
\newcommand{\orcidlogo}{\BeginAccSupp{method=escape,Alt={ORCiD ID}}\begin{tikzpicture}[yscale=-1,transform shape]
\fill[orcidlogocol] svg{M256,128c0,70.7-57.3,128-128,128C57.3,256,0,198.7,0,128C0,57.3,57.3,0,128,0C198.7,0,256,57.3,256,128z};
    \fill[white] svg{M86.3,186.2H70.9V79.1h15.4v48.4V186.2z}
                 svg{M108.9,79.1h41.6c39.6,0,57,28.3,57,53.6c0,27.5-21.5,53.6-56.8,53.6h-41.8V79.1z M124.3,172.4h24.5c34.9,0,42.9-26.5,42.9-39.7c0-21.5-13.7-39.7-43.7-39.7h-23.7V172.4z}
                 svg{M88.7,56.8c0,5.5-4.5,10.1-10.1,10.1c-5.6,0-10.1-4.6-10.1-10.1c0-5.6,4.5-10.1,10.1-10.1C84.2,46.7,88.7,51.3,88.7,56.8z};
\end{tikzpicture}\EndAccSupp{}}
\newcommand\orcidlogosized[1]{\raisebox{-#1/5}{\resizebox{!}{#1}{\orcidlogo}}}
\newcommand\orcid[2][1em]{\mbox{\href{https://orcid.org/#2}{\orcidlogosized{#1}\hspace{\dimexpr #1/2}\nolinkurl{#2}}}}

\usepackage[margin=1.25in]{geometry}
\usepackage{booktabs,array}
\usepackage[draft=false,final]{microtype}
\usepackage{amsthm}
\usepackage[nottoc]{tocbibind}
\usepackage{xparse}
\RequirePackage{xcolor}
\definecolor{DarkGreen}{rgb}{0.15,0.5,0.15}
\definecolor{DarkRed}{rgb}{0.6,0.2,0.2}
\definecolor{DarkBlue}{rgb}{0.15,0.15,0.55}
\definecolor{DarkPurple}{rgb}{0.4,0.2,0.4}
\usepackage{hyperxmp}
\usepackage[draft=false,pagebackref=true]{hyperref}
\hypersetup{
	pdfencoding=unicode,
	linktocpage=true,
	colorlinks=true,                
	linkcolor=DarkBlue,     
	citecolor=DarkBlue, 
	urlcolor=DarkBlue,      
	pdftitle={Optimal Bounds between \unichar{"1D453}-Divergences and Integral Probability Metrics},
	pdfauthor={Rohit Agrawal, Thibaut Horel},
	pdflang={en},
	keeppdfinfo,
}
\usepackage[ocgcolorlinks]{ocgx2}
\usepackage[capitalize]{cleveref}
\usepackage{crossreftools}
\def\titlephi/{\texorpdfstring{$\phi$}{\ifpdfstringunicode{\unichar{"1D719}}{phi}}}

\newenvironment{enum}{\begin{enumerate}

\setlength{\itemsep}{0pt}}{\end{enumerate}}

\usepackage{mathtools}
\DeclarePairedDelimiter{\paren}{(}{)}
\DeclarePairedDelimiterX{\divergenceparen}[2]{(}{)}{#1\;\delimsize\|\;\mathopen{}#2}
\DeclarePairedDelimiter{\norm}{\lVert}{\rVert}

\DeclareMathOperator{\KL}{D}
\newcommand{\kl}[3][\divergenceparen*]{\KL#1{#2}{#3}}
\newcommand{\chisq}[3][\divergenceparen*]{\chi^2#1{#2}{#3}}

\newcommand{\pushforward}[2]{#2_*#1}
\newcommand{\di}[3][\phi]{\KL_{#1}\divergenceparen*{#2}{#3}}
\newcommand{\tv}[2]{{\rm TV}\paren*{#1,#2}}
\newcommand{\ipmsymb}[1]{d_{#1}}
\newcommand{\ipm}[3][\G]{\ipmsymb{#1}\paren*{#2,#3}}
\renewcommand{\L}{\mathcal{L}}
\newcommand{\M}{\mathcal{M}}

\DeclareDocumentCommand\G{}{{\mathcal{G}}}
\newcommand{\A}{\mathcal{F}}
\newcommand{\B}{\mathcal{B}}

\newcommand{\pos}[1]{#1_{\scriptscriptstyle +}}
\newcommand{\con}[1]{#1^\star}
\newcommand{\bicon}[1]{#1^{\star\star}}
\newcommand{\conder}[1]{#1^{\star\prime}}
\newcommand{\ip}[2]{\langle #1, #2\rangle}
\DeclareMathOperator{\dom}{\mathrm{dom}}
\DeclareMathOperator{\inter}{\mathrm{int}}
\DeclareMathOperator{\extreme}{\mathrm{ext}}
\DeclareMathOperator{\convhull}{\mathrm{co}}
\newcommand{\ind}{\mathbf{1}}
\newcommand{\ident}{\mathrm{Id}}
\newcommand{\supp}[1]{h_{#1}}
\newcommand{\lbsymbol}{\mathscr{L}}
\NewDocumentCommand{\lb}{o m m}{\IfValueTF{#1}{\lbsymbol_{#2,#3,#1}}{\lbsymbol_{#2,#3}}}
\newcommand{\lbsing}[3][\perp]{\lbsymbol_{#2,#3,#1}}
\newcommand{\lbipm}[2]{\lbsymbol_{#1,#2,#2}}

\newcommand{\cgf}[3][\phi]{K_{#2,#3}}
\newcommand{\cgfxi}[4][\phi]{K_{#2,#3,#4}}
\newcommand{\optimcgf}[4][\phi]{F_{#2,#3,#4}}

\newcommand{\subexp}[1][\phi]{S^{#1}}
\newcommand{\ssubexp}[1][\phi]{S_{\scriptscriptstyle \heartsuit}^{#1}}

\newcommand{\orhrt}[1]{L_{\scriptscriptstyle \heartsuit}^{#1}}

\makeatletter
\newcommand{\intf}{\@ifstar\tildeintegral\integral}
\newcommand{\dif}{\@ifstar\tildedifunc\difunc}
\newcommand{\integral}[2]{I_{#1,#2}}
\newcommand{\difunc}[2]{\KL_{#1,#2}}
\newcommand{\tildeintegral}[2]{\tilde{I}_{#1,#2}}
\newcommand{\tildedifunc}[2]{\widetilde{\KL}_{#1,#2}}
\makeatother

\newcommand{\sls}[1]{\operatorname{sls}_{#1}}
\newcommand{\height}[1]{\operatorname{hgt}_{#1}}

\makeatletter
\newcommand{\ex}{\@ifstar\exBig\exSmall}
\newcommand{\exBig}[2]{\int #2\,\mathrm{d}#1}
\newcommand{\exSmall}[3][\paren*]{#2#1{#3}}
\makeatother

\DeclarePairedDelimiter{\abs}{\lvert}{\rvert}
\renewcommand{\land}{\,\wedge\,}

\newcommand\cdotarg{\:\cdot\:}

\providecommand\given{}
\newcommand\SetGiven[1][]{%
\nonscript\:#1\vert\allowbreak\nonscript\:\mathopen{}}
\DeclarePairedDelimiterX\set[1]\{\}{%
\renewcommand\given{\SetGiven[\delimsize]}#1}

\newcommand{\R}{\mathbb{R}}
\newcommand{\eR}{\overline\R}
\newcommand{\eps}{\varepsilon}
\providecommand{\eqdef}{\coloneqq}

\DeclareMathOperator{\esssup}{ess\,sup}
\DeclareMathOperator{\essinf}{ess\,inf}
\DeclareMathOperator{\essimop}{ess\,im}
\newcommand{\essim}[2]{\essimop_{#1}(#2)}

\newtheorem{theorem}{Theorem}[subsection]

\newtheorem{lemma}[theorem]{Lemma}

\newtheorem{proposition}[theorem]{Proposition}

\newtheorem{corollary}[theorem]{Corollary}

\theoremstyle{definition}
\newtheorem{definition}[theorem]{Definition}

\theoremstyle{remark}
\newtheorem{remark}[theorem]{Remark}
\newtheorem{example}[theorem]{Example}

\makeatletter
\@ifpackageloaded{natbib}%
	{%
		\NewDocumentCommand\citea{o m m}{\IfValueTF{#1}{#1 }{}\citet{#3}}
		\newcommand\ifnatbib[2]{#1}
	}%
	{%
		\usepackage{letltxmacro}
		\LetLtxMacro\citet\cite
		\LetLtxMacro\citep\cite
		\newcommand\citea[3][]{#2 \citet{#3}}
		\newcommand\ifnatbib[2]{#2}
	}
\makeatother
\let\oldcite\cite
\let\cite\useciteptainstead
 
\title{Optimal Bounds between \texorpdfstring{$f\hspace{-0.3ex}$}{\unichar{"1D453}}-Divergences\\ and Integral Probability Metrics\footnote{An extended abstract of this work appeared as \citet{AH20}.}}
\author{%
Rohit Agrawal%
\thanks{Harvard John A. Paulson School of Engineering and Applied Sciences,
	Cambridge, MA 02138, USA. \orcid{0000-0001-5563-7402}. Supported by the Department
of Defense (DoD) through the National Defense Science \& Engineering Graduate
Fellowship (NDSEG) Program.}\\Harvard University\and
Thibaut Horel%
\thanks{Institute for Data, Systems, and Society, Massachusetts
Institute of Technology, Cambridge, MA 02139, USA. \orcid{0000-0002-4889-6833}.}%
\\MIT}

\begin{document}
\maketitle

\begin{abstract}
The families of $f\hspace{-0.3ex}$-divergences (e.g.\ the Kullback--Leibler
divergence) and Integral Probability Metrics (e.g.\ total variation distance or
maximum mean discrepancies) are widely used to quantify the similarity between
probability distributions. In this work, we systematically study the
relationship between these two families from the perspective of convex duality.
Starting from a tight variational representation of the
$f\hspace{-0.3ex}$-divergence, we derive a generalization of the
moment-generating function, which we show exactly characterizes the best lower
bound of the $f\hspace{-0.3ex}$-divergence as a function of a given IPM. Using
this characterization, we obtain new bounds while also recovering in a unified
manner well-known results, such as Hoeffding's lemma, Pinsker's inequality and
its extension to subgaussian functions, and the Hammersley--Chapman--Robbins
bound. This characterization also allows us to prove new results on
topological properties of the divergence which may be of independent interest.
 \end{abstract}

\section{Introduction}
\label{sec:intro}

Quantifying the extent to which two probability distributions differ from one
another is central in most, if not all, problems and methods in machine
learning and statistics. In a line of research going back at least to the work
of \citea{Kullback}{K59}, information theoretic measures of dissimilarity
between probability distributions have provided a fruitful and unifying
perspective on a wide range of statistical procedures. A prototypical example
of this perspective is the interpretation of maximum likelihood estimation as
minimizing the Kullback--Leibler divergence between the empirical
distribution---or the ground truth distribution in the limit of infinitely
large sample---and a distribution chosen from a parametric family.

A natural generalization of the Kullback--Leibler divergence is provided by the
family of $\phi$-divergences\footnote{Henceforth, we use $\phi$-divergence
instead of $f\hspace{-0.3ex}$-divergence and reserve the letter $f$ for
a generic function.} \citep{C63,C67} also known in statistics as Ali--Silvey
distances \citep{AS66}.\footnote{$\phi$-divergences had previously been considered
\citep{R61,M63}, though not as an independent
object of study.} Informally, a $\phi$-divergence quantifies the divergence
between two distributions $\mu$ and $\nu$ as an average cost of the likelihood
ratio, that is, $\di\mu\nu \eqdef \ex*\nu{\phi(d\mu/d\nu)}$ for a convex cost
function $\phi:\R_{\geq 0}\to\R_{\geq 0}$. Notable examples of
$\phi$-divergences include the Hellinger distance, the $\alpha$-divergences (a
convex transformation of the Rényi divergences), and the $\chi^2$-divergence.

Crucial in applications of $\phi$-divergences are their so-called
\emph{variational representations}. For example, the Donsker--Varadhan
representation \citep[Theorem 5.2]{DV76} expresses the
Kullback--Leibler divergence $\kl\mu\nu$ between probability distributions $\mu$ and $\nu$ as
\begin{equation}
	\label{eq:dv-intro}
	\kl\mu\nu = \sup_{g\in \L^b}\set*{\ex*\mu{\!g}
	-\log\!\!\ex*\nu{\!\!e^g}}
	\,,
\end{equation}
where $\L^b$ is the space of bounded measurable functions.
Similar variational representations were for example used by
\citet{NWJ08, NWJ10, RRGP12, BBROBCH18} to construct estimates of
$\phi$-divergences by restricting the optimization problem \eqref{eq:dv-intro}
to a class of functions $\G\subseteq \L^b$ for which the problem
becomes tractable (for example when $\G$ is a RKHS or representable by a given
neural network architecture). In recent work, \citet{NCT16, NCMQW17}
conceptualized an extension of generative adversarial networks (GANs) in
which the problem of minimizing a $\phi$-divergence is expressed via 
representations such as \eqref{eq:dv-intro} as a two-player game between 
neural networks, one minimizing over probability distributions $\mu$, the other
maximizing over $g$ as in \eqref{eq:dv-intro}.

Another important class of distances between probability distributions is given
by Integral Probability Metrics (IPMs) defined by \citet{M97} and taking the
form
\begin{equation}
	\label{eq:ipm-intro}
	\ipm\mu\nu = \sup_{g\in\G}\set*{\abs*{\ex*\mu g - \ex*\nu g}}
	\,,
\end{equation}
where $\G$ is a class of functions parametrizing the distance. Notable examples
include the total variation distance ($\G$ is the class of all functions taking
value in $[-1,1]$), the Wasserstein metric ($\G$ is a class of Lipschitz
functions) and Maximum Mean Discrepancies ($\G$ is the unit ball of a RKHS).
Being already expressed as a variational problem, IPMs are amenable  to
estimation, as was exploited by \citet{SFGSL12, GBRSS12}. MMDs have also been
used in lieu of $\phi$-divergences to train GANs as was first done by
\citet{DRG15}.

Rewriting the optimization problem \eqref{eq:dv-intro} as
\begin{equation}
	\label{eq:dv-intro-bis}
	\sup_{g\in \L^b}\set*{\ex*\mu{\!g}
		-\ex*\nu{\!g}
	-\log\!\!\ex*\nu{\!\!e^{(g-\ex*\nu{\!g})}}}
\end{equation}
reveals an important connection between $\phi$-divergences and IPMs. Indeed,
\eqref{eq:dv-intro-bis} expresses the divergence as the solution to
a regularized optimization problem in which one attempts to maximize the mean
deviation $\ex*\mu g-\ex*\nu g$, as in \eqref{eq:ipm-intro}, while also
penalizing functions $g$ which are too ``complex'' as measured by
the centered log moment-generating function of $g$. In this
work, we further explore the connection between $\phi$-divergences and IPMs,
guided by the following question:
\begin{center}
	\emph{what is the best lower bound of a given $\phi$-divergence\\
		as a function of a given integral probability metric?}
\end{center}
Some specific instances of this question are already well understood. For
example, the best lower bound of the Kullback--Leibler divergence by
a quadratic function of the total variation distance is known as Pinsker's
inequality. More generally, describing the best lower bound of a
$\phi$-divergence as a function of the total variation distance (without being
restricted to being a quadratic), is known as Vajda's problem, to which an
answer was given by \citet{FHT03} for the Kullback--Leibler divergence and by
\citet{G06} for an arbitrary $\phi$-divergence.

Beyond the total variation distance---in particular, when the class $\G$ in
\eqref{eq:ipm-intro} contains unbounded functions---few results are known.
Using \eqref{eq:dv-intro-bis}, \citet[\S 4.10]{BLM13} shows that Pinsker's
inequality holds as long as the log moment-generating function grows at most
quadratically. Since this is the case for bounded functions (via Hoeffding's
lemma), this recovers Pinsker's inequality and extends it to the class of
so-called \emph{subgaussian} functions. This was recently used by \citet{RZ20}
to control bias in adaptive data analysis.

In this work, we systematize the convex analytic perspective underlying many
of these results, thereby developing the necessary tools to resolve the above
guiding question. As an application, we recover in a unified manner the known
bounds between $\phi$-divergences and IPMs, and extend them along several
dimensions.  Specifically, starting from the observation of \citet{RRGP12} that
the variational representation of $\phi$-divergences commonly used in the
literature is not ``tight'' for probability measures (in a sense which will be
made formal in the paper), we make the following contributions:
\begin{itemize}
	\item we derive a tight representation of $\phi$-divergences for probability
		measures, exactly generalizing the Donsker--Varadhan representation of the
		Kullback--Leibler divergence.
	\item we define a generalization of the log moment-generating function and
		show that it exactly characterizes the best lower bound of
		a $\phi$-divergence by a given IPM. As an application, we show that
		this function grows quadratically if and only if the $\phi$-divergence
		can be lower bounded by a quadratic function  of the given IPM and
		recover in a unified manner the extension of Pinsker's inequality to
		subgaussian functions and the Hammersley--Chapman--Robbins bound.
	\item we characterize the existence of \emph{any} non-trivial lower
		bound on an IPM in terms of the generalized log moment-generating
		function, and give implications for topological properties of the
		divergence, for example regarding compactness of sets of measures with
		bounded $\phi$-divergence and the relationship between convergence in
		$\phi$-divergence and weak convergence.
	\item the answer to Vajda's problem for bounded functions is
		re-derived in a principled manner, providing a new geometric
		interpretation on the optimal lower bound of the $\phi$-divergence
		by the total variation distance. From this, we derive a refinement
		of Hoeffding's lemma and generalizations of Pinsker's inequality
		to a large class of $\phi$-divergences.
\end{itemize}

The rest of this paper is organized as follows: \cref{sec:related}
discusses related work, \cref{sec:prelims} gives a brief overview of
concepts and tools used in this paper, \cref{sec:var} derives the tight
variational representation of the $\phi$-divergence, \cref{sec:singleg}
focuses on the case of an IPM given by a single function $g$ with
respect to a reference measure $\nu$, deriving the optimal bound in this
case and discussing topological applications, and \cref{sec:ipm} extends
this to arbitrary IPMs and sets of measures, with applications to
subgaussian functions and Vajda's problem.
 \section{Related work}
\label{sec:related}

The question studied in the present paper is an instance of the broader problem
of the constrained minimization of a $\phi$-divergence, which has been
extensively studied in works spanning information theory, statistics and convex
analysis.

\paragraph{Kullback--Leibler divergence.} The problem of minimizing the
Kullback--Leibler divergence \citep{KL51} subject to a convex constraint can be
traced back at least to \citet{S57} in the context of large deviation theory and
to \citet{K59} for the purpose of formulating an information theoretic approach
to statistics. In information theory, this problem is known as an
$I$-projection \citep{C75,CM03}. The case where the convex set is defined by
finitely many affine equality constraints, which is closest to our work, was
specifically studied in \citet{BC77,BC79} via a convex duality approach. This
special case is of particular relevance to the field of statistics, since the
exponential family arises as the optimizer of this problem.

\paragraph{Convex integral functionals and general $\phi$.}

With the advent of the theory of convex integral functionals, initiated in
convex analysis by \citet{R66,R68}, the problem is generalized to arbitrary
$\phi$-divergences, sometimes referred to as $\phi$-entropies, especially when
seen as functionals over spaces of functions, and increasingly studied via
a systematic application of convex duality \citep{TV93}. In the case of affine
constraints, the main technical challenge is to identify constraint
qualifications guaranteeing that strong duality holds: \citet{BL91,BL93,BK06}
investigate the notion of quasi-relative interior for this purpose, and
\citet{L01_Inverse,L01} consider integrability conditions on the functions
defining the affine constraints. A comprehensive
account of this case can be found in \citet{CM12}. We also
note the work \citet{AS06}, which shows a duality between \emph{approximate}
divergence minimization---where the affine constraints are only required to
hold up to a certain accuracy---and maximum a posteriori estimation in
statistics.

At a high level, in our work we show in \cref{sec:ipm} that one can
essentially reduce the problem of minimizing the divergence on probability
measures subject to a constraint on an IPM to the problem of minimizing the
divergence on finite measures subject to two affine constraints: the first
restricting to probability measures, and the second constraining the mean
deviation of a single function in the class defining the IPM.  For the
restriction to probability measures, we prove that constraint qualification
always holds, a fact which was not observed in the aforecited works, to the
best of our knowledge. For the second constraint, we show in
\cref{sec:inf-compact} that by focusing on a single function, we can relate
strong duality of the minimization problem to compactness properties of the
divergence. In particular, we obtain strong duality under similar assumptions
as those considered in \citet{L01_Inverse}, even when the usual interiority
conditions for constraint qualification do not hold.

\paragraph{Relationship between $\phi$-divergences.}

A specific case of the minimization question which has seen significant work is
when the feasible set is defined by other $\phi$-divergences, and most notably
is a level set the total variation distance. The best-known result in this
line is Pinsker's inequality, first proved in a weaker form in \citet{P60,P63}
and then strengthened independently in \citet{K67, K69, C67}, which gives the
best possible quadratic lower bound on the Kullback--Leibler divergence by the
total variation distance. More recently, for $\phi$-divergences other than the
Kullback--Leibler divergence, \citet{G10} identified conditions on $\phi$ under
which quadratic ``Pinsker-type'' lower bounds can be obtained.

More generally, the problem of finding the best lower bound of the
Kullback--Leibler divergence as a (possibly non-quadratic) function of the
total variation distance was introduced by Vajda in \citet{V70} and generalized to
arbitrary $\phi$-divergences in \citet{V72}, and is therefore sometimes referred
to as \emph{Vajda's problem}. Approximations of the best lower bound were
obtained in \citet{BH79,V70} for the Kullback--Leibler divergence and in
\citet{V72,G08,G10} for $\phi$-divergences under various assumptions on $\phi$.
The optimal lower bound was derived in \citet{FHT03} for the Kullback--Leibler
divergence and in \citet{G06} for any $\phi$-divergence.  As an example
application of Section~\ref{sec:ipm}, in Section~\ref{sec:bounded} we rederive
the optimal lower bound as well as its quadratic relaxations in a unified
manner.

In \citet{RW09,RW11}, the authors consider the generalization of Vajda's
problem of obtaining a tight lower bound on an arbitrary $\phi$-divergence
given multiple values of \emph{generalized total variation distances}; their
result contains \citet{G06} as a special case. Beyond the total variation
distance, \citet{HV11} introduced the general question of studying the
\emph{joint range} of values taken by an arbitrary pair of $\phi$-divergences,
which has its boundary given by the best lower bounds of one divergence as a
function of the other. \citet{GSS14} generalize this further and consider the
general problem of understanding the joint range of multiple
$\phi$-divergences, i.e.~minimizing a $\phi$-divergence subject to a finite
number of constraints on other $\phi$-divergences. A key conceptual
contribution in this line of work is to show that these optimization problems,
which are defined over (infinitely dimensional) spaces of measures, can be
reduced to finite dimensional optimization problems. A related line of work
\citep{SV16,S18} deriving relations between $\phi$-divergences instead approaches
the problem by defining integral representations of $\phi$-divergences in terms
of simple ones.

Our work differs from results of this type since we are primarily concerned
with IPMs other than the total variation distance, and in particular with
those containing unbounded functions. It was shown in
\citet{KFG06,KFG07,SGFLS09,SFGSL12} that the class of $\phi$-divergences and
the class of pseudometrics (including IPMs) intersect \emph{only} at the total
variation distance. As such, the problem studied in the present paper cannot
be phrased as the one of a joint range between two $\phi$-divergences, and to
the best of our knowledge cannot be handled by the techniques used in studying
the joint range.

\paragraph{Transport inequalities.} Starting with the work of \citea{Marton}{M86},
transportation inequalities upper bounding the Wasserstein distance
by a function of the relative entropy have been instrumental in the study of
the concentration of measure phenomenon (see e.g.\ \citet{GL10} for a survey).
These inequalities are related to the question studied in this work since the
$1$-Wasserstein distance is an IPM when the probability space is a Polish space
and coincides with the total variation distance when the probability space is
discrete and endowed with the discrete metric. In an influential paper, 
\citea{Bobkov and Götze}{BG99} proved that upper bounding the $1$-Wasserstein distance
by a square root of the relative entropy is equivalent to upper bounding the
log moment-generating function of all $1$-Lipschitz functions by a quadratic
function. The extension of Pinsker's inequality in \citet[\S 4.10]{BLM13},
which was inspired by \citet{BG99}, is also based on quadratic upper
bounds of the log moment-generating function and we in turn follow similar
ideas in \cref{sec:var-prob,sec:char} of the present work.
 \section{Preliminaries}
\label{sec:prelims}

\subsection{Measure Theory}
\paragraph{Notation.}
Unless otherwise noted, all the probability measures in this paper are defined
on a common measurable space $(\Omega,\A)$, which we assume is non-trivial in
the sense that $\set{\emptyset,\Omega}\subsetneq \A$, as otherwise all questions
considered in this paper become trivial. We denote by $\M(\Omega, \A)$,
$\M^+(\Omega, \A)$ and $\M^1(\Omega,\A)$ the sets of finite signed measures,
finite non-negative measures, and probability measures respectively.
$\L^0(\Omega,\A)$ denotes the space of all measurable functions from $\Omega$ to
$\R$, and $\L^b(\Omega,\A)\subseteq\L^0(\Omega,\A)$ is the set of all bounded
measurable functions. For $\nu\in\M(\Omega, \A)$, and $1\leq p\leq\infty$,
$\L^p(\nu,\Omega, \A)$ denotes the space of measurable functions with finite
$p$-norm with respect to $\nu$, and $L^p(\nu,\Omega,\A)$ denotes the space
obtained by taking the quotient with respect to the space of functions which are
$0$ $\nu$-almost everywhere. Similarly, $L^0(\nu,\Omega,\A)$ is the space of all
measurable functions $\Omega$ to $\R$ up to equality $\nu$-almost everywhere.
When there is no ambiguity, we drop the indication $(\Omega, \A)$. For a
measurable function $f\in\L^0$ and measure $\nu\in\M$, $\ex\nu f\eqdef \ex*\nu
f$ denotes the integral of $f$ with respect to $\nu$.

For two measures $\mu$ and $\nu$, $\mu\ll\nu$ (resp.\ $\mu\perp\nu$) denotes
that $\mu$ is absolutely continuous (resp.\ singular) with respect to $\nu$
and we define $\M_c(\nu)\eqdef\set{\mu\in\M\given \mu\ll\nu}$ and
$\M_s(\nu)\eqdef\set{\mu\in\M\given \mu\perp \nu}$, so that by the Lebesgue
decomposition theorem we have the direct sum $\M = \M_c(\nu)\oplus\M_s(\nu)$.
For $\mu\in\M_c(\nu)$, $\frac{d\mu}{d\nu}\in L^1(\nu)$ denotes the
Radon--Nikodym derivative of $\mu$ with respect to $\nu$.  For a signed
measure $\nu\in\M$, we write the Hahn--Jordan decomposition $\nu=\nu^+-\nu^-$
where $\nu^+,\nu^-\in\M^+$, and denote by $\abs\nu=\nu^++\nu^-$ the total
variation measure.

More generally, given a $\sigma$-ideal $\Sigma\subseteq\A$ we
write $\mu\ll \Sigma$ to express that $\abs\mu(A)=0$ for all $A\in\Sigma$ and
define $\M_c(\Sigma)\eqdef\set{\mu\in\M\given \mu\ll\Sigma}$. Similarly,
$L^0(\Sigma)$ denotes the quotient of $\L^0$ by the space of functions equal to
$0$ except on an element of $\Sigma$. For a measurable function $f\in\L^0$, and
$\sigma$-ideal $\Sigma$, $\essim\Sigma f \eqdef
\bigcap_{\eps>0}\set*{x\in\R\given f^{-1}\paren[\big]{(x-\eps,x+\eps)}\notin
\Sigma}$ is the essential range of $f$ with respect to $\Sigma$, and
$\esssup_\Sigma f\eqdef\sup\essim\Sigma f$ and $\essinf_\Sigma
f\eqdef\inf\essim\Sigma f$ denote the $\Sigma$-essential supremum and infimum
respectively. Finally $L^\infty(\Sigma)$ denotes the space of a functions whose
$\Sigma$-essential range is bounded, up to equality except on an element of
$\Sigma$. When $\Sigma$ is the $\sigma$-ideal of null sets of a measure $\nu$,
we abuse notations and write $\essim\nu f$ for $\essim\Sigma f$ and similarly
for the essential supremum and infimum.

Finally, for brevity, we define for a subspace $X\subseteq \M$ of finite
signed measures the subsets $X^+\eqdef X\cap\M^+$ and $X^1\eqdef X\cap\M^1$,
and for $\nu\in\M$ we also define $X_c(\nu) \eqdef X\cap\M_c(\nu)$ and
$X_s(\nu)\eqdef X\cap\M_s(\nu)$.

\paragraph{Integral Probability Metrics.}

\begin{definition}
	For a non-empty set of measurable functions $\G\subseteq\L^0$,
	the \emph{integral probability metric} associated with $\G$ is defined by
	\begin{displaymath}
		\ipm\mu\nu\eqdef\sup_{g\in\G}\set*{\abs*{\ex*\mu g-\ex*\nu g}}
		\,,
	\end{displaymath}
	for all pairs of measures $(\mu,\nu)\in\M^2$ such that all functions in
	$\G$ are absolutely $\mu$- and $\nu$-integrable.  We extend this definition
	to all pairs of measures $(\mu,\nu)\in\M^2$ by $\ipm\mu\nu=+\infty$ in cases
	where there exists a function in $\G$ which is not $\mu$- or $\nu$-
	integrable.
\end{definition}
\begin{remark}
	When the class $\G$ is closed under negation, one can drop the absolute value
	in the definition.
\end{remark}
\begin{example}
	The total variation distance $\tv\mu\nu$ is obtained when $\G$ is the class
	of measurable functions taking values in $[-1,1]$.\footnote{Note that total
	variation distance is sometimes defined as half of this quantity,
	corresponding to functions taking values in $[0,1]$.}
\end{example}
\begin{example}
	\label{ex:ipm-rv}
	Note that the integrals $\ex*\mu g$ and $\ex*\nu g$ depend only on the
	pushforward measures $\pushforward\mu g$ and $\pushforward\nu g$ on $\R$.
	Equivalently, when $\mu$ and $\nu$ are the probability distributions of
	random variables $X$  and $Y$ taking values in $\Omega$, we have that
	$\ex*\mu g=\ex*{\pushforward\mu g}{\ident_{\R}} =\mathbb{E}[g(X)]$, the
	expectation of the random variable $g(X)$, and similarly $\ex*\nu
	g = \mathbb{E}[g(Y)]$. The integral probability metric $\ipmsymb\G$ thus
	defines the distance between random variables $X$ and $Y$ as the largest
	difference in expectation achievable by ``observing'' $X$ and $Y$ through
	a function from the class $\G$.
\end{example}

\subsection{Convex analysis}

Most of the convex functions considered in this paper will be defined over
spaces of measures or functions. Consequently, we will apply tools from convex
analysis in its general formulation for locally convex topological vector
spaces. References on this subject include \citet{BCR84} and \citet[II. and
IV.\S1]{B87} for the topological background, and \citet[Part I]{ET99} and
\citet[Chapters 1 \& 2]{Z02} for convex analysis. We now briefly review the main
concepts appearing in the present paper.

\begin{definition}[Dual pair]
	A \emph{dual pair} is a triplet $(X, Y, \ip\cdotarg\cdotarg)$ where $X$ and $Y$
	are real vector spaces, and $\ip\cdotarg\cdotarg:X\times Y\to\R$ is a bilinear
	form satisfying the following properties:
	\begin{enum}
	\item\label{it:dp-x} for every $x\in X\setminus\{0\}$, there exists $y\in
		Y$ such that $\ip x y \neq 0$.
	\item\label{it:dp-y} for every $y\in Y\setminus\{0\}$, there exists $x\in
		X$ such that $\ip x y \neq 0$.
	\end{enum}
	We say that the \emph{pairing} $\ip\cdotarg\cdotarg$ \emph{puts $X$
	and $Y$ in (separating) duality}. Furthermore, a topology $\tau$ on $X$ is
	said to be \emph{compatible} with the pairing if it is locally convex and
	if the topological dual $\con X$ of $X$ with respect to $\tau$ is
	isomorphic to $Y$. Topologies on $Y$ compatible with the pairing are
	defined similarly.
\end{definition}

\begin{example}
	\label{ex:weak}
	For an arbitrary dual pair $(X, Y, \ip\cdotarg\cdotarg)$, the
	\emph{weak topology} $\sigma(X, Y)$ induced by $Y$ on $X$ is defined to be
	the coarsest topology such that for each $y\in Y$, $x\mapsto \ip x y$ is
	a continuous linear form on $X$. It is a locally convex Hausdorff topology
	induced by the family of seminorms $p_y: x\mapsto \abs{\ip x y}$ for $y\in Y$
	and is thereby compatible with the duality between $X$ and $Y$.

	Note that in finite dimension, all Hausdorff vector space topologies
	coincide with the standard topology.
\end{example}

In the remainder of this section, we fix a dual pair $(X, Y, \ip\cdotarg\cdotarg)$
and endow $X$ and $Y$ with topologies compatible with the pairing. As is
customary in convex analysis, convex functions take values in the set of
extended reals $\eR\eqdef\R\cup\set{-\infty,+\infty}$ to which the addition
over $\R$ is extended using the usual conventions, including
$(+\infty)+(-\infty)=+\infty$. In this manner, convex functions can always be
extended to be defined on the entirety of their domain by assuming the value
$+\infty$ when they are not defined. For a convex function $f:X\to\overline\R$,
$\dom f\eqdef\set{x\in X\given f(x)<+\infty}$ is the \emph{effective domain} of
$f$ and $\partial f(x)\eqdef\set{y\in Y\given \forall x'\in X,f(x')\geq
f(x)+\ip{x'-x}y}$ denotes its subdifferential at $x\in X$.

\begin{definition}[Lower semicontinuity, inf-compactness]
	The function $f:X\to\eR$ is \emph{lower semicontinuous (lsc)}
	(resp.\ \emph{inf-compact}) if for every $t\in\R$ the sublevel set
	$f^{-1}(-\infty, t]\eqdef \set{x\in X\given f(x)\leq t}$ is closed (resp.\
	compact).
\end{definition}

\begin{lemma}
\label{lem:optimal-value-convex}
	If $f:X\times C\to\eR$ is a convex function for $C$  a convex subset of some
	linear space, then $g:X\to\eR$ defined as $g(x)\eqdef\inf_{c\in C}f(x,c)$ is
	convex. Furthermore, if for some topology on $C$ the function $f$ is
	inf-compact with respect to the product topology, then $g$ is also
	inf-compact.
\end{lemma}

\begin{definition}[Properness]
	A convex function $f:X\to\eR$ is \emph{proper} if $\dom f\neq\emptyset$ and
	$f(x)>-\infty$ for all $x\in X$.
\end{definition}

\begin{definition}[Convex conjugate]
	The \emph{convex conjugate} (also called Fenchel dual or Fenchel--Legendre
	transform) of $f:X\to\eR$ is the function $\con f:
	Y\to\eR$ defined for $y\in Y$ by
	\begin{displaymath}
		\con f(y) \eqdef \sup_{x\in X} \set[\big]{\ip x y - f(x)}
		\,.
	\end{displaymath}
\end{definition}

For a set $C\subseteq X$, $\delta_C:X\to\eR_{\geq 0}$ denotes the
characteristic function of $C$, that is $\delta_C(x)$ is $0$ if $x\in C$ and
$+\infty$ elsewhere. The support function of $C$ is $\supp C:Y\to
\R\cup\set{+\infty}$ defined by $h_C(y) = \sup_{x\in C}\ip x y$. If $C$ is
closed and convex then  $(\delta_C, \supp C)$ form a pair of convex
conjugate functions.

\begin{proposition}
	\label{prop:conj}
	Let $f:X\to\eR$ be a function. Then:
	\begin{enumerate}
		\setlength{\itemsep}{0pt}
		\item $\con f:Y\to\eR$ is convex and lower semicontinuous.
		\item\label{it:fy} for all $x\in X$ and $y\in Y$, $f(x) + \con f(y)\geq \ip x y$
			with equality iff $y\in\partial f(x)$.
		\item\label{it:fm} $\bicon f\leq f$ with
			equality iff $f$ is proper convex lower semicontinuous, $f\equiv+\infty$ or
			$f\equiv-\infty$.
		\item if $f\leq g$ for some $g:X\to\eR$, then $\con g\geq \con f$.
	\end{enumerate}
\end{proposition}

\begin{remark}
	In \cref{prop:conj}, \cref{it:fy} is known as the
	Fenchel--Young inequality and \cref{it:fm} as the Fenchel--Moreau
	theorem.
\end{remark}

In the special case of $X=\R=Y$ and a proper convex function $f:\R\to\eR$,
we can be more explicit about some properties of $\con f$ and $\bicon f$.

\begin{lemma}\label{lem:biconj-R}
	If $f:\R\to\eR$ is a proper convex function, then $x\in\R$ is such that
	$f(x)\neq\bicon f(x)$ only if $\dom f$ has non-empty interior and $x$ is
	one of the (at most two) points on its boundary, in which case
	$\bicon f(x)$ is the limit of $f(x')$ as $x'\to x$ within $\dom f$.
\end{lemma}

\begin{definition}
\label{def:derinfty}
	For $f:\R\to\eR$ a proper convex function, we define for
	$\ell\in\set{-\infty,+\infty}$ the quantity
	$f'(\ell)\eqdef\lim_{x\to\ell}f(x)/x\in\R\cup\set{+\infty}$.
\end{definition}
\begin{remark}
	The limit is always well-defined in $\R\cup\set{+\infty}$ for proper convex
	functions. The name $f'(\ell)$ is motivated by the fact that when $f$ is
	differentiable, we have $f'(\ell)=\lim_{x\to\ell}f'(x)$.
\end{remark}

\begin{lemma}
\label{lem:domconj}
	If $f:\R\to\eR$ is a proper convex function, then the domain of $\con
f:\R\to\eR$ satisfies $\inter(\dom\con
f)=\paren[\big]{f'(-\infty),f'(+\infty)}$.
\end{lemma}

\begin{lemma}\label{lem:bicon-inf}
	Let $(f_i)_{i\in I}$ be a collection of convex functions from $\R$ to $\eR$
	which are non-decreasing over some convex set $C\subseteq \R$. Then for all $x\in\inter C$
	\begin{displaymath}
		\lim_{x'\to x^{-}}\inf_{i\in I} f_i(x')
		\leq \inf_{i\in I} \bicon{f_i}(x)\leq \inf_{i\in I} f_i(x)
		\,.
	\end{displaymath}
\end{lemma}

\begin{proof}
	For each $i\in I$ we have by \cref{lem:biconj-R} that
	$\bicon f_i(x)\in\set{f_i(x),\lim_{x'\to x^{-}}f_i(x')}$, so since $f_i$
	is non-decreasing over $C$ and $\bicon f_i\leq f_i$ by \cref{prop:conj},
	the result follows by taking the infimum over $i\in I$ as
  $\lim_{x'\to x^{-}} \inf_{i\in I}f_i(x')\leq \inf_{i\in I}\lim_{x'\to x^{-}} f_i(x')$.
\end{proof}

Fenchel duality theorem is arguably the most fundamental result in convex
analysis, and we will use it in this paper to compute the convex conjugate and
minimum of a convex function subject to a linear constraint. The following
proposition summarizes the conclusions obtained by instantiating the duality
theorem to this specific case.

\begin{proposition}
\label{prop:fenchelcor}
	Let $f:X\to(-\infty, +\infty]$ be a convex
	function. For $y\in Y$ and $\eps\in\R$, define
	$f_{y,\eps}:X\to(-\infty,+\infty]$ by
	\begin{displaymath}
		f_{y,\eps}(x)
\eqdef f(x)+\delta_{\set{\eps}}\paren[\big]{\ip xy}
=\begin{cases}f(x)&\text{ if }\ip{x}y=\eps\\+\infty&\text{ otherwise}\end{cases}
	\end{displaymath}
	for all $x\in X$.
	\begin{enumerate}
		\item Assume that $f$ is lower semicontinuous and define
		$\ip{\dom f}y\eqdef\set{\ip xy\given x\in\dom f}$. If
	$\eps\in\inter\paren[\big]{\ip{\dom  f}y}$, then
	$\con{f_{y,\eps}}(\con x) =\inf_{\lambda\in\R} \con f(\con x+\lambda
	y) - \lambda\cdot\eps$ for all $\con x\in Y$, where
			the infimum is reached whenever $\con{f_{y,\eps}}(\con x)$ is finite.
\item Assume that $f$ is non-negative and satisfies $f(0) = 0$. Define
	the \emph{marginal value function}
	\begin{equation}
		\lbsymbol_{y, f}(\eps)
		\eqdef\inf_{x\in X} f_{y,\eps}(x)
		=\inf\set*{ f(x)\given x\in X\land \ip xy=\eps}\,.
		\label{eqn:ipmPhiy}
	\end{equation}
	Then $\lbsymbol_{y, f}$ is a non-negative convex function satisfying
	$\lbsymbol_{y, f}(0) = 0$ and its convex conjugate is given by
	$\con{\lbsymbol_{y, f}}(t) =\con f(ty)$. 
	Furthermore, $\lbsymbol_{y, f}$ is lower semicontinuous at $\eps$, that
	is $\lbsymbol_{y, f}(\eps)=\bicon{\lb y f}(\eps)$, if and only if
	strong duality holds for problem \eqref{eqn:ipmPhiy}, i.e.~if and
	only if
	\[
		\inf\set*{ f(x)\given x\in X\land \ip xy=\eps}
			=
		\sup\set*{t\cdot\eps - \con f(t\cdot y)\given t\in\R}
		\,.
	\]
	\end{enumerate}

\end{proposition}

\begin{proof}
	\begin{enumerate}
		\item This follows from a direct application of Fenchel's duality
			theorem (see e.g.\ \citet[Corollary 2.6.4, Theorem 2.8.1]{Z02}).
		\item Define the \emph{perturbation function} $F:X\times\R \to\eR$ by
			$F(x, \eps)\eqdef f_{y,\eps}(x) = f(x) + \delta_{\set
			0}\paren[\big]{\ip xy-\eps}$ so that $\lbsymbol_{y,
			f}(\eps)=\inf_{x\in X}F(x,\eps)$. Since $F$ is non-negative, jointly
			convex over the convex set $X\times\R$ and $F(0,0) = 0$, we get that
			$\lbsymbol_{y, f}$ is itself convex, non-negative, and satisfies
			$\lbsymbol_{y, f}(0)=0$.  Furthermore, $\con F(\con x, t)=\con f(\con
			x + ty)$ and $\con{\lbsymbol_{y, f}}(t)=\con F(0,t)=\con f(ty)$ by
			e.g.\ \citet[Theorem 2.6.1, Corollary 2.6.4]{Z02}.\qedhere
	\end{enumerate}
\end{proof}

Finally, we will use the following result giving a sufficient condition for
a convex function to be bounded below. Most such results in convex analysis
assume that the function is either lower semicontinuous or bounded above on an
open set. In contrast, the following lemma assumes that the function is upper
bounded on a closed, convex, bounded set of a Banach space, or more generally
on a \emph{cs-compact} subset of a real Hausdorff topological vector space.

\begin{lemma}[{cf.~\citet[Example 1.6(0), Remark 1.9]{K86}}]
\label{lem:convex-bounded-cs-compact}
	Let $C$ be a cs-compact subset of a real Hausdorff topological
	vector space. If $f:C\to\R$ is a convex function such that
	$\sup_{x\in C}f(x)<+\infty$, then $\inf_{x\in C}f(x)>-\infty$.
	In particular, if $f:C\to\R$ is linear, then $\sup_{x\in C}
	f(x)<+\infty$ if and only if $\inf_{x\in C}f(x)>-\infty$.
\end{lemma}

The notion of cs-compactness (called $\sigma$-convexity in \citet{K86}) was
introduced and defined \citea[in]{by Jameson in}{J72}, and Proposition
2 of the same paper states that closed, convex, bounded sets of Banach spaces
  are cs-compact. For completeness, we include a proof of
  \cref{lem:convex-bounded-cs-compact} in \cref{sec:convex-bounded-proof}.

\subsection{Orlicz spaces}
\label{sec:orlicz-prelims}
We will use elementary facts from the theory of Orlicz spaces which we now
briefly review (see for example \citet{Leonard07} for a concise exposition
or \citet{RR91} for a more complete reference).  A function
$\theta:\R\to[0,+\infty]$ is a \emph{Young function} if it is a convex,
lower semicontinuous, and even function with $\theta(0)=0$ and
$0<\theta(s)<+\infty$ for some $s>0$. Then writing $\intf\theta\nu:
f\mapsto\ex*\nu{\theta(f)}$ for $\nu\in\M$,
one defines\footnote{The definition and theory of Orlicz spaces holds more
generally for $\sigma$-finite measures. The case of finite measures already
covers all the applications considered in this paper whose focus is primarily
on probability measures.} two spaces associated with
$\theta$:
\begin{itemize}
	\item the Orlicz space $L^\theta(\nu)\eqdef\set[\big]{f\in L^0(\nu)\given
		\exists\alpha>0, \intf\theta\nu(\alpha f)<\infty}$,
	\item the Orlicz heart \citep{ES89}
		$\orhrt\theta(\nu)\eqdef\set[\big]{f\in L^0(\nu)\given\forall\alpha>0,
		\intf\theta\nu(\alpha f)<\infty}$, also known as the
 Morse--Transue space \citep{MT50},
\end{itemize}
which are both Banach spaces when equipped with the Luxemburg norm $\norm
f_\theta\eqdef \inf\set{t>0\given \intf\theta\nu(f/t)\leq 1}$.  Furthermore,
$\orhrt\theta(\nu)\subseteq L^\theta(\nu)\subseteq L^1(\nu)$ and
$L^\infty(\nu)\subseteq L^\theta(\nu)$ for all $\theta$, and
$L^\infty(\nu)\subseteq \orhrt\theta(\nu)$ when $\dom\theta=\R$.  If
$\con\theta$ is the convex conjugate of $\theta$, we have the following
analogue of Hölder's inequality:
$\ex*\nu{f_1f_2}\leq2\norm{f_1}_\theta\norm{f_2}_{\con \theta}$, for all
$f_1\in L^\theta(\nu)$ and $f_2\in L^{\con\theta}(\nu)$, implying that
$(L^\theta, L^{\con\theta})$ are in dual pairing. Furthermore, if
$\dom\theta=\R$, we have that the dual Banach space
$(\orhrt\theta,\norm{\cdotarg}_\theta)^\star$ is isomorphic to
$(L^{\con\theta},\norm\cdotarg_{\con\theta})$.
 \section{Variational representations of \titlephi/-divergences}
\label{sec:var}

In the rest of this paper, we fix a convex and lower semicontinuous function
$\phi:\R\to\R\cup\set{+\infty}$ such that $\phi(1)=0$.  After defining
$\phi$-divergences in \cref{sec:phi-div}, we start with the usual variational
representation of the $\phi$-divergence in \cref{sec:var-gen}, which we then
strengthen in the case of probability measures in \cref{sec:var-prob}. A
reader interested primarily in optimal bounds between $\phi$-divergences
and IPMs can skip \cref{sec:var-gen,sec:var-prob} at a first reading.

\subsection{Convex integral functionals and \titlephi/-divergences}
\label{sec:phi-div}

The notion of a $\phi$-divergence is closely related to the one of a convex
integral functional that we define first.

\begin{definition}[Integral functional]
	\label{def:intf}
	For $\nu\in\M^+$ and $f:\R\to\R\cup\{\infty\}$ a proper convex function, the
	convex integral functional associated with $f$ and $\nu$ is the function
	$\intf f\nu: L^1(\nu)\to\R\cup\{\infty\}$ defined for $g\in L^1(\nu)$ by
	\begin{displaymath}
		\intf f\nu (g) = \ex*\nu {f\circ g}\,.
	\end{displaymath}
\end{definition}
The systematic study of convex integral functionals from the perspective of
convex analysis was initiated by \citea{Rockafellar in}{R68, R71}, who
considered more generally functionals of the form $g\mapsto \ex*\nu{f(\omega,
g(\omega))}$ for $g:\Omega\to\R^n$ and $f:\Omega\times\R^n\to\R$ such that
$f(\omega,\cdot)$ is convex $\nu$-almost everywhere. A good introduction to the
theory of such functionals can be found in \citet{R76, R98_14}. The specific case
of \cref{def:intf} is known as an \emph{autonomous} integral functional, but we
drop this qualifier since it applies to all functionals studied in this paper.

\begin{definition}[$\phi$-divergence]
	\label{def:div}
	For $\mu\in\M$ and $\nu\in\M^+$, write $\mu
	= \mu_c + \mu_s$ with $\mu_c\ll\nu$ and $\mu_s\perp \nu$, the Lebesgue
	decomposition of $\mu$ with respect to $\nu$, and $\mu_s=\mu_s^+-\mu_s^-$
	with $\mu_s^+,\mu_s^-\in\M^+$, the Hahn--Jordan decomposition of
	$\mu_s$. The $\phi$-\emph{divergence} of $\mu$ with respect to $\nu$ is
	the quantity $\di\mu\nu\in\R\cup\set{\infty}$ defined by
	\begin{displaymath}
		\di\mu\nu
		\eqdef\ex*\nu{\phi\paren*{\frac{d\mu_c}{d\nu}}}
		+ \ex{\mu_s^+}{\Omega}\cdot\phi'(\infty)
		- \ex{\mu_s^-}{\Omega}\cdot\phi'(-\infty)
		\,,
	\end{displaymath}
	with the convention $0\cdot(\pm\infty)=0$.
\end{definition}

\begin{remark}
\label{rmrk:divergence-symmetric-defn}
	An equivalent definition of $\di\mu\nu$ which does not require decomposing
	$\mu$ is obtained by choosing $\lambda\in\M^+$ dominating both $\mu$ and
	$\nu$ (e.g.\ $\lambda=\abs\mu+\nu$) and defining
	\begin{displaymath}
		\di\mu\nu
		= \ex*\lambda{\frac{d\nu}{d\lambda}\cdot\phi\paren*{\frac{d\mu/d\lambda}{d\nu/d\lambda}}},
	\end{displaymath}
	with the conventions coming from continuous extension that $0\cdot\phi(a/0)
	= a\cdot\phi'(\infty)$ if $a\geq 0$ and $0\cdot \phi(a/0)
	= a\cdot\phi'(-\infty)$ if $a\leq 0$ (see \cref{def:derinfty}). It is easy
	to check that this definition does not depend on the choice of $\lambda$
	and coincides with \cref{def:div}.
\end{remark}

\begin{table}[t]
	\centering
	\renewcommand{\arraystretch}{1.2}
		\begin{tabular}{@{}m{8.5em}cccm{9em}} \toprule
			Name & $\phi$ & $\phi'(\infty)<\infty$? & $\phi(0)<\infty$? & Notes\\ \midrule
			$\alpha$-divergences & $\frac{x^\alpha-1}{\alpha(\alpha-1)}$ & when $\alpha<1$
			& when $\alpha>0$ &$\phi_\alpha^\dagger=\phi_{1-\alpha}$\\
			KL & $x\log x$ & No & Yes & Limit of $\alpha\to1^-$\\
			reverse KL & $-\log x$ & Yes & No & Limit of $\alpha\to0^+$\\
			squared Hellinger & $(\sqrt{x}-1)^2$ & Yes & Yes & Scaling of $\alpha=\frac12$\\
			$\chi^2$-divergence & $(x-1)^2$ & No & Yes & Scaling of $\alpha=2$\\
			Jeffreys & $(x-1)\log x$ & No & No & KL $+$ reverse KL\\
			$\chi^\alpha$-divergences & $\abs{x-1}^\alpha$ & when $\alpha=1$ & Yes & For $\alpha\geq 1$\ifnatbib{\newline}{ }\citep{V73}\\[0.1em]
			Total variation & $\abs{x-1}$ & Yes & Yes & $\chi^1$-divergence\\[0.2em]
			Jensen--Shannon &\parbox[m]{7em}{\centering $x\log x-\newline(1+x)\log\paren*{\frac{1+x}2}$} & Yes & Yes & a.k.a.~total divergence\newline to the average\\
			Triangular\newline discrimination & $\frac{(x-1)^2}{x+1}$ & Yes & Yes & a.k.a.~Vincze--Le Cam distance \\ \bottomrule
		\end{tabular}
		\caption{Common $\phi$-divergences (see e.g.~\citet{SV16})}
		\label{tab:phi-divergences}
\end{table}

The notion of $\phi$-divergence between probability measures was introduced by
\citea{Csiszár in}{C63,C67} in information theory and independently by
\citea{Ali and Silvey}{AS66} in statistics. The generalization to finite
signed measures is from \citet{CGG99}. Some useful properties of the
$\phi$-divergence include: it is jointly convex in both its arguments, if
$\mu(\Omega)=\nu(\Omega)$ then $\di\mu\nu\geq 0$, with equality if and only if
$\mu=\nu$ assuming that $\phi$ is strictly convex at $1$.

\begin{remark} \label{rem:phi'-finite-infinite}
If $\mu\ll\nu$, the definition simplifies to $\di\mu\nu
= \ex\nu{\phi\circ\frac{d\mu}{d\nu}}$.  Furthermore, if $\phi'(\pm\infty)
= \pm\infty$, then $\di\mu\nu = +\infty$ whenever $\mu\not\ll\nu$. When either
$\phi'(+\infty)$ or $\phi'(-\infty)$ is finite, some authors implicitly or
explicitly redefine $\di\mu\nu$ to be $+\infty$ whenever $\mu\not\ll\nu$, thus
departing from \cref{def:div}. This effectively defines $\di\cdot\nu$ as the
integral functional $\intf\phi\nu$ and the rich  theory of convex integral
functionals can be readily applied. As we will see in this paper, this change
of definition is unnecessary and the difficulties arising from the case
$\mu\not\ll\nu$ in \cref{def:div} can be addressed by separately treating the
component of $\mu$ singular with respect to $\nu$.

An important reason to prefer the general definition is the equality
$\di\nu\mu = \di[\phi^\dagger]\mu\nu$ where $\phi^\dagger:x\mapsto x\phi(1/x)$
is the Csiszár dual of $\phi$, which identifies the \emph{reverse}
$\phi$-divergence---\allowbreak where the arguments are swapped---with the divergence
associated with $\phi^\dagger$. Consequently, any result obtained for the
partial function $\mu\mapsto\di\mu\nu$ can be translated into results for the
partial function $\nu\mapsto\di\mu\nu$ by swapping the role of $\mu$ and $\nu$
and replacing $\phi$ with $\phi^\dagger$. Note that $(\phi^\dagger)'(\infty)
=\lim_{x\to0^+}\phi(x)$ and $(\phi^\dagger)'(-\infty)=\lim_{x\to0^-}\phi(x)$,
and for many divergences of interest (including the Kullback--Leibler divergence)
at least one of $\phi'(\infty)$ and $\phi(0)$ is finite. See
\cref{tab:phi-divergences} for some examples.
\end{remark}

\subsection{Variational representations: general measures}
\label{sec:var-gen}

In this section, we fix a finite and non-negative measure
$\nu\in\M^+\setminus\set 0$ and study the convex functional
$\dif\phi\nu:\mu\mapsto\di\mu\nu$ over a vector space $X$ of finite measures
containing $\nu$. Our primary goal is to derive a \emph{variational
representation} of $\dif\phi\nu$, expressing it as the solution of an
optimization problem over $Y$, a vector space of functions put in dual pairing
with $X$ via $\ip\mu h =\ex\mu h$, for a measure $\mu\in X$ and a function
$h\in Y$.

Care must be taken in specifying the dual pair $(X, Y)$, since the variational
representation we obtain depends on it, or more precisely, on the null
$\sigma$-ideal $\Xi$ of $(\Omega,\A)$ consisting of the measurable sets that
are null for all measures in $X$. Note that, as discussed in
\cref{rem:phi'-finite-infinite}, this ideal $\Xi$ is irrelevant when
$\phi'(\pm\infty)=\pm\infty$ (e.g. when $\KL_\phi$ is the KL divergence), since
then $\di\mu\nu=+\infty$ whenever $\mu\not\ll\nu$, but when, $\phi'(+\infty)$ or
$\phi'(-\infty)$ is finite, it is possible to have $\di\mu\nu<\infty$ even when
$\mu\not\ll\nu$, and we wish to obtain a variational representation for such
discontinuous measures as well.
Denoting by $N$ the $\sigma$-ideal of $\nu$-null
sets, we always have $\Xi\subseteq N$ since we assume that $\nu\in X$. If
$\Xi=N$, then we have $X\subseteq \M_c(\nu)$, corresponding to case of only absolutely
continuous measures, but if $\Xi\subset N$ is a proper subset of $N$ then
$N\setminus \Xi$ quantifies the ``amount of $\nu$-singularities'' of measures
in $X$. The extreme case where $\Xi=\set{\emptyset}$ allows for arbitrary
singularities, since this implies that for any measurable set $A\in\A$, there
exists a measure in $X$ with positive variation on $A$. Furthermore, for common measurable
spaces $\Omega$, there is usually an ambient measure $\lambda$ for which it is
natural to assume that $X\subseteq \M_c(\lambda)$ (e.g. $\lambda$ could be the
Lebesgue measure on $\R$ or more generally the Haar measure on a locally
compact unimodular group).  Denoting by $L$ the null $\sigma$-ideal of this
ambient measure, we could then take $L\subseteq \Xi$, thus restricting the
singularities of measures in $X$.

Formally, we require that the pair $(X, Y)$ satisfies
a \emph{decomposability} condition that we define next. It is closely related
to Rockafellar's notion of a decomposable space \citep[Section
3]{R76} which plays an important role in the theory of convex integral
  functionals.

\begin{definition}[Decomposability]
	\label{assmp:decomp}
	Let $X\subseteq \M(\Omega, \A)$ be a vector space of finite measures and
	define $\Xi \eqdef \set{A\in\A\given \forall \mu\in X,\; \abs\mu(A)=0}$ the
	$\sigma$-ideal of measurable sets that are null for all measures in $X$.
	Let $Y$ be a vector space of measurable functions and let
	$\nu\in X$ be a finite non-negative measure. We say that the
	pair $(X,Y)$ is $\nu$-decomposable if:
	\begin{enumerate}
		\item the pairing $(\mu, h)\mapsto \ex*\mu h$ puts
			 $X$ and $Y$ in separating duality.
		\item $\set*{\mu\in\M_c(\nu)\given \frac{d\mu}{d\nu}\in
			L^\infty(\nu)}\subseteq X$
			and $L^\infty(\Xi)\subseteq Y\subseteq L^0(\Xi)$.
		\item for all $A\notin\Xi$, there
			exists $\mu\in X^+\setminus\set 0$ such that
			$\ex\mu{\Omega\setminus A}=0$.
	\end{enumerate}
\end{definition}

\begin{remark}\label{rem:decomp}
	Note that items $2$ and $3$ together imply that the duality is
	necessarily separating, so the definition would remain identical by only
	requiring in item $1$ that $\ex\mu h$ be finite for each $\mu\in X$ and
	$h\in Y$.
	Furthermore, if we strengthen condition $2$ by requiring that $\set{\mu\in
	\M_c(\mu')\given \frac{d\mu}{d\mu'}\in L^\infty(\mu')}\subseteq X$ for all
	measures $\mu'\in X$, then $3$ is implied. Thus, starting from an arbitrary
	dual pair $(X, Y)$ in separating duality, one can extend $X$ by taking its
	sum with the space of all measures of bounded derivative with respect to
	measures in $X$ and extend $Y$ by taking its sum with $L^\infty(\Xi)$. The
	resulting pair of extended spaces will then be decomposable with respect to
	any measure in $X$.
\end{remark}

\begin{example}
	If $X\subseteq \M_c(\nu)$, then item $2$ implies item $3$,
	and $\nu$-decomposability then simply expresses that $X$ and $Y$
	form a dual pair of decomposable spaces in the sense of \citet[Section
	3]{R76}, once $\M_c(\nu)$ is identified with $L^1(\nu)$ via the
		Radon--Nikodym theorem. An example of a $\nu$-decomposable pair in this
	case is given by $\paren[\big]{\M_c(\nu), L^\infty(\nu)}$. More generally,
	if $\Xi$ is a proper subset of the $\sigma$-ideal of $\nu$-null sets, then
	item 3 requires $X$ to contain ``sufficiently many'' $\nu$-singular measures.
	An example of a $\nu$-decomposable pair for which $\Xi=\set{\emptyset}$ is
	given by $X=\M$ and $Y=\L^b(\Omega)$. An intermediate example which will be
	useful when considering IPMs can be obtained by constructing the largest dual
	pair $(X, Y)$ such that $Y$ contains a class of functions $\G$ of interest.
	The details of the construction are given in \cref{defn:G-dual-pair} and
	decomposability is stated and proved in \cref{lem:xg-yg-decomposable}.
\end{example}

With \cref{assmp:decomp} at hand, our approach to obtain variational
representations of the divergence is simple. We first compute the convex
conjugate $\con{\dif\phi\nu}$ of $\dif\phi\nu$ defined for $h\in Y$ by
\begin{equation}
	\label{eq:div-dual}
	\con{\dif\phi\nu}(h) = \sup_{\mu\in X}\set*{\ex\mu h - \dif\phi\nu(\mu)}
\end{equation}
and prove that $\dif\phi\nu$ is lower semicontinuous. By the Fenchel--Moreau
theorem, we thus obtain the representation $\dif\phi\nu(\mu)
= \bicon{\dif\phi\nu}(\mu)=\sup_{h\in Y}\set[\big]{\ex\mu
h - \con{\dif\phi\nu}(h)}$.

We start with the simplest case where $X\subseteq \M_c(\nu)$, that is when all
the measures in $X$ are $\nu$-absolutely continuous. Since $\dif\phi\nu$
coincides with the integral functional $\intf\phi\nu$ in this case, this lets
us exploit the well-known fact that under our decomposability condition,
$(\intf\phi\nu, \intf{\con\phi\!}\nu)$ form a pair of convex conjugate
functionals. This fact was first observed in \citet{LZ56} in the context of
Orlicz spaces, and then generalized in \citet{R68,R71}.

\begin{proposition}
	\label{prop:i-dual}
	Let $\nu\in\M^+$ be non-negative and finite, and let $(X, Y)$ be
	$\nu$-decom\-posable with $X\subseteq \M_c(\nu)$. Then the convex conjugate
	$\con{\dif\phi\nu}$ of $\dif\phi\nu$ over $X$ is given for all $h\in Y$ by
	\begin{displaymath}
		\con{\dif\phi\nu} (h)= \intf{\con\phi\!}{\nu}(h)=\ex*\nu{\con\phi\circ
		h}.
	\end{displaymath}
	Furthermore $\dif\phi\nu$ is lower semicontinuous, therefore for all $\mu\in
	X$
	\begin{equation}
		\label{eq:var}
		\di\mu\nu = \sup_{h\in Y}\set*{\ex*\mu h-\ex*\nu{\con\phi\circ h}}.
	\end{equation}
\end{proposition}

\begin{proof}
	Since $\nu\in X$ by assumption, the function $\dif\phi\nu$ is proper and
	convex over $X$. The proposition is then immediate consequence of
	\citet[Theorem 3C]{R76} after identifying $\M_c(\nu)$ with $L^1(\nu)$ by the
	Radon--Nikodym theorem and noting that $X$ and $Y$ are
	decomposable \citep[Section 3]{R76} by \cref{assmp:decomp}.
\end{proof}

\begin{example}
	\label{ex:stupid-dv}
	Consider the case of the Kullback--Leibler divergence, corresponding to the
	function $\phi:x \mapsto x\log x$. A simple computation gives $\con\phi(x)
	= e^{x-1}$ and \eqref{eq:var} yields as a variational representation, for all
	$\mu\in X$
	\begin{equation}
		\label{eq:stupid-dv}
		\kl\mu\nu
		= \sup_{g\in Y} \set*{\ex\mu g - \ex*\nu{e^{g-1}}},
	\end{equation}
	Note that this representation differs from the Donsker--Varadhan
	representation~\eqref{eq:dv-intro} discussed in the introduction. This
	discrepancy will be explained in the next section.
\end{example}

The variational representation of the $\phi$-divergence in
\cref{prop:i-dual} is well-known (see e.g.\ \citet{RRGP12}). However, as already
discussed, the case where $X$ contains $\nu$-singular measures is also of
interest and has been comparatively less studied in the literature. The
following proposition generalizes the expression for $\con{\dif\phi\nu}$
obtained in \cref{prop:i-dual} to the general case of an arbitrary
$\nu$-decomposable pair $(X, Y)$, without requiring that $X\subseteq
\M_c(\nu)$.

\begin{proposition}
\label{prop:finite-dagger-dual}
Let $\nu\in\M^+$ be a non-negative and finite measure and assume that $(X,Y)$
is $\nu$-decomposable. Then, the functional $\dif\phi\nu$ over $X$ has convex
conjugate $\con{\dif\phi\nu}$ given for all $g\in Y$ by
	\begin{equation}
		\con{\dif\phi\nu}(h)
		= \begin{cases}
		\intf{\con\phi\!}\nu (h)&\text{if } \essim\Xi h\subseteq [\phi'(-\infty),\phi'(\infty)]\\
		+\infty &\text{otherwise}
	\end{cases} \,,
	\label{eqn:finite-dagger-dual}
	\end{equation}
	where $\Xi \eqdef \set{A\in\A\given \forall \mu\in X,\; \abs\mu(A)=0}$ is
	the null $\sigma$-ideal of $X$.
\end{proposition}

\begin{proof}
	For $h\in Y$, let $C(h)$ be the right-hand side of \cref{eqn:finite-dagger-dual},
	our claimed expression for $\con{\dif\phi\nu}(h)$.

First, we show that $\sup_{\mu\in X}\set[\big]{\ex\mu h - \dif\phi\nu(\mu)}\leq C(h)$. For this, we
assume that $\essim\Xi h\subseteq [\phi'(-\infty),\phi'(\infty)]$,
as otherwise $C(h)=+\infty$ and there is nothing to prove. For $\mu\in X$, write
$\mu=\mu_c+\mu_s^+-\mu_s^-$ with $\mu_c\in\M_c(\nu)$ and
$\mu_s^+,\mu_s^-\in\M^+_s(\nu)$, so that
\begin{equation}
	\label{eq:foo-dec}
	\ex\mu h - \dif\phi\nu(\mu) = \ex{\mu_c}
	h - \intf\phi\nu\paren*{\frac{d\mu_c}{d\nu}}
	+ \ex{\mu_s^+}h - \mu_s^+(\Omega)\cdot\phi'(\infty)
	- \ex{\mu_s^-}h + \mu_s^-(\Omega)\cdot\phi'(-\infty).
\end{equation}
Observe that $\ex{\mu_c} h - \intf\phi\nu\paren*{\frac{d\mu_c}{d\nu}}
= \ex\nu{\frac{d\mu_c}{d\nu}\cdot h - \phi\circ\frac{d\mu_c}{d\nu}}\leq
\ex\nu{\con\phi\circ h}=\intf{\con\phi\!}\nu(h)$, by the Fenchel--Young
inequality applied to $\phi$ and monotonicity of the integral with respect to
the non-negative measure $\nu$. Since $\mu\ll\Xi$ by definition of $\Xi$ and
thus $\mu_s^+\ll\Xi$, we have $\phi'(\infty)\geq\esssup_\Xi h\geq
\esssup_{\mu_s^+} h$ so that $\ex{\mu_s^+}h -\mu_s^+(\Omega)\cdot\phi'(\infty)
= \ex{\mu_s^+}{h - \phi'(\infty)}\leq 0$.  Similarly
$\mu_s^-(\Omega)\cdot\phi'(-\infty) - \ex{\mu_s^-}h\leq 0$. Using these bounds
in \eqref{eq:foo-dec} yields $\ex\mu h - \dif\phi\nu(\mu)\leq C(h)$ as desired.

Next, we show that $\sup_{\mu\in X}\set[\big]{\ex\mu h - \dif\phi\nu(\mu)}\geq C(h)$.
Observe that
\begin{equation}
	\label{eq:sing-cont-case}
\sup_{\mu\in X}\set[\big]{\ex\mu h - \dif\phi\nu(\mu)}
\geq\sup_{\mu\in X_c(\nu)}\set[\big]{\ex\mu h - \dif\phi\nu(\mu)}
=\intf{\con\phi\!}\nu(h)\,,
\end{equation}
where the equality follows from \cref{prop:i-dual} applied to $X_\nu=X_c(\nu)$
and $Y_\nu=Y/\!\!\sim_\nu$ where $\sim_\nu$ is the equivalence relation of
being equal $\nu$-almost everywhere. If $\essim\Xi
h\subseteq[\phi'(-\infty),\phi'(\infty)]$, then $\intf{\con\phi}\nu(h)=C(h)$
and \eqref{eq:sing-cont-case} gives the desired conclusion.  If $\esssup_\Xi
h>\phi'(\infty)$, let $\alpha\in\R$ such that $\phi'(\infty)<\alpha<\esssup_\Xi
h$. Then $A=\set{\omega\in\Omega\given h(\omega)>\alpha}$ is a measurable set
in $\A\setminus\Xi$.
If $\nu(A)>0$,  then $\intf{\con\phi\!}\nu(h)=\infty=C(h)$, since
$\dom\con\phi\subseteq [\phi'(-\infty),\phi'(\infty)]$ and
\eqref{eq:sing-cont-case} again gives the desired conclusion.
If $\nu(A)>0$, then by \cref{assmp:decomp} there exists $\mu_A\in
X^+\setminus\set{0}$ such that $\mu_A(\Omega\setminus A)=0$. But then
\begin{align*}
\sup_{\mu\in X}\set[\big]{\ex\mu h - \dif\phi\nu(\mu)}
&\geq \sup_{c>0}\set[\big]{\ex{(\nu + c\mu_A)}h - \dif\phi\nu(\nu + c\mu_A)}\\
&= \ex\nu h+\sup_{c>0}\set[\big]{c\ex{\mu_A}h - c\mu_A(\Omega)\cdot\phi'(\infty)}\\
&\geq \ex\nu h
+\sup_{c>0}\set[\big]{c\mu_A(\Omega)\cdot\paren[\big]{\alpha-\phi'(\infty)}}
=+\infty=C(h)\,,
\end{align*}
where the first equality is because
$\intf\phi\nu\paren*{\frac{d\nu}{d\nu}}=\phi(1)=0$ and $\mu_A\in X_s^+(\nu)$,
and the second is because $\ex{\mu_A}\Omega>0$ and $\alpha>\phi'(\infty)$. The
case $\essinf_\Xi h(\Omega)<\phi'(-\infty)$ is analogous.
\end{proof}

\begin{remark}
	\label{rem:par}
	Although the expression of $\con{\dif\phi\nu}$ obtained in
	\cref{prop:finite-dagger-dual} should coincide with the one obtained in
	\cref{prop:i-dual} when $X\subseteq \M_c(\nu)$ (in which case $\Xi$
	coincides with the $\sigma$-ideal of $\nu$-null sets), it appears different
	at first glance because of the explicit constraint on the $\Xi$-essential
	range of $g$ present in \eqref{eqn:finite-dagger-dual}.  However, this
	constraint is also present, though implicit, in \cref{prop:i-dual} since
	$\overline{\dom\con\phi}=[\phi'(-\infty),\phi'(\infty)]$ and thus
	$\intf{\con\phi}\nu(h)=+\infty$ whenever $\essim\nu
	h\not\subseteq[\phi'(-\infty),\phi(\infty)]$. When $X$ is allowed to
	contain measures which are not absolutely continuous with respect to $\nu$,
	this implicit constraint on the $\nu$-essential range is simply
	strengthened to restrict the $\Xi$-essential range instead. In the extreme case
	where $\Xi=\set{\emptyset}$ then the true range of $h$ is constrained.
\end{remark}

Finally, we prove that $\dif\phi\nu$ is lower semicontinuous over $X$,
yielding a variational representation of $\di\mu\nu$ in the general case.

\begin{proposition}
\label{prop:finite-dagger-lsc}
Let $\nu\in\M^+$ be a non-negative and finite measure and assume that $(X,Y)$
is $\nu$-decomposable. Then, $\dif\phi\nu$ is lower semicontinuous over $X$.
Equivalently, we have for all $\mu\in X$ the biconjugate representation
	\begin{displaymath}
		\di\mu\nu
		=\sup\set[\big]{ \ex\mu g-\intf{\con\phi\!}\nu(g)
	\given g\in Y\land \essim\Xi g\subseteq [\phi'(-\infty),\phi'(\infty)]} \,,
	\end{displaymath}
	where $\Xi = \set{A\in\A\given \forall \mu\in X,\; \abs\mu(A)=0}$ is
	the null $\sigma$-ideal of $X$.
\end{proposition}

\begin{proof}
	Since $\dif\phi\nu$ is proper, by the Fenchel--Moreau theorem it suffices to
	show that $\bicon{\dif\phi\nu}\geq\dif\phi\nu$.  For $\mu\in X$, write
	$\mu=\mu_c+\mu_s^+-\mu_s^-$ with $\mu_c\in\M_c(\nu)$, and
	$\mu_s^+,\mu_s^-\in\M_s^+(\nu)$ by the Lebesgue and Hahn--Jordan
	decompositions. Furthermore, let $(C,P,N)\in \A^3$ be a partition of
	$\Omega$ such that $\abs{\mu_c}(\Omega\setminus C)=\nu(\Omega\setminus
	C)=0$, $\mu_s^+(\Omega\setminus P)=0$ and $\mu_s^-(\Omega\setminus N)=0$.
	By \cref{prop:finite-dagger-dual},
  \begin{equation}
		\label{eq:foo-bicon}
    \begin{aligned}
      \bicon{\dif\phi\nu}(\mu)
      = \sup\big\{\ex{\mu_c} g - \intf{\con\phi\!}\nu (g)
      &+\ex{\mu_s^+} g -\ex{\mu_s^-} g\\&\SetGiven[\big] g\in Y\land
        \essim\Xi g\subseteq [\phi'(-\infty),\phi'(\infty)]\big\} \,.
    \end{aligned}
  \end{equation}

	Let $\alpha\in\R$ such that $\alpha
	< \intf\phi\nu\paren*{\frac{d\mu_c}{d\nu}}$. Applying \cref{prop:i-dual}
	with $X_\nu=\M_c(\nu)$ and $Y_\nu=L^\infty(\nu)$, we get the existence of $g_c\in
	L^\infty(\nu)$ such that $\ex{\mu_c} {g_c}
	- \intf{\con\phi\!}\nu(g_c)>\alpha$. Furthermore, since
	$\dom{\con\phi}\subseteq[\phi'(-\infty),\phi'(\infty)]$, we have that
	$g_c\in[\phi'(-\infty), \phi'(\infty)]$ $\nu$-almost everywhere.
	Consequently, there exists a representative $\tilde g_c\in\L^b(\Omega)$
	of $g_c$ such that $\tilde g_c(\Omega)\subseteq
	[\phi'(-\infty),\phi'(\infty)]$.

	For $\beta,\gamma\in\R\cap[\phi'(-\infty),\phi'(\infty)]$ (which is
	nonempty since it contains $\dom\con\phi$ and $\phi$ is convex and proper),
	define $\tilde g:\Omega\to \R$ by
	\begin{displaymath}
		\tilde g(\omega) = \begin{cases}
			\tilde g_c(\omega) &\text{if }\omega\in C\\
			\beta &\text{if }\omega\in P\\
			\gamma &\text{if }\omega\in N
		\end{cases}\,.
	\end{displaymath}
	By construction $\tilde g\in\L^b(\Omega)$, hence its equivalence class $g$
	in $L^\infty(\Xi)$ belongs to $Y$ by \cref{assmp:decomp}.  Furthermore,
	since $\mu\ll\Xi$ we have $\ex{\mu_c} g -\intf{\con\phi\!}\nu (g)=
	\ex{\mu_c} {\tilde g_c}-\intf{\con\phi\!}\nu (\tilde
	g_c)=\ex{\mu_c}{g_c}-\intf{\con\phi\!}\nu(g_c) >\alpha$, $\ex{\mu_s^+}
	g = \mu_s^+(\Omega)\cdot\beta$, and $\ex{\mu_s^-}
	g = \mu_s^-(\Omega)\cdot\gamma$. Since $\tilde
	g(\Omega)\subseteq[\phi'(-\infty),\phi'(\infty)]$ by construction, for this
	choice of $g\in Y$, the optimand in \eqref{eq:foo-bicon} is at least
	$\alpha+\ex{\mu_s^+}\Omega\cdot\beta -\ex{\mu_s^-}\Omega\cdot\gamma$. This
	concludes the proof since $\alpha,\beta,\gamma$ can be made arbitrarily
	close to $\intf\phi\nu\paren*{\frac{d\mu_c}{d\nu}}$, $\phi'(\infty)$, and
	$\phi'(-\infty)$ respectively.
\end{proof}

\subsection{Variational representations: probability measures}
\label{sec:var-prob}

When applied to \emph{probability measures}, which are the main focus of this
paper, the variational representations provided by
\cref{prop:i-dual,prop:finite-dagger-lsc} are loose. This fact was first
explicitly mentioned in \citet{RRGP12}, where the authors also suggested that
tighter representations could be obtained by specializing the derivation to
probability measures.

Specifically, given a dual pair $(X,Y)$ as in \cref{sec:var-gen}, we restrict
$\dif\phi\nu$ to probability measures by defining $\dif*\phi\nu:\mu\mapsto
\dif\phi\nu(\mu) +\delta_{\M^1}(\mu)$ for $\mu\in X$. For $g\in Y$ we get
\begin{equation}
	\label{eq:div-dual-prob}
\con{\dif*\phi\nu}(g)
	= \sup_{\mu\in X}\set[\big]{\ex\mu g-\dif*\phi\nu(\mu)}
	=\sup_{\mu\in X^1}\set[\big]{\ex\mu g - \dif\phi\nu(\mu)}\,.
\end{equation}
Observe that compared to \eqref{eq:div-dual}, the supremum is now taken over
the smaller set $X^1=X\cap\M^1$, and thus $\con{\dif*\phi\nu}\leq
\con{\dif\phi\nu}$. When $\dif*\phi\nu$ is lower semicontinuous we then get for
$\mu\in X^1$
\begin{equation}
	\label{eq:var-prob}
	\di\mu\nu = \dif*\phi\nu(\mu) = \bicon{\dif*\phi\nu}(\mu)
=\sup_{g\in Y}\set[\big]{\ex\mu g - \con{\dif*\phi\nu}(g)}\,.
\end{equation}
This representation should be contrasted with the one obtained in
\cref{sec:var-gen}, $\di\mu\nu = \sup_{g\in Y}\set[\big]{\ex\mu
g - \con{\dif\phi\nu}(g)}$, which holds for any $\mu\in X$ and in which the
optimand is smaller than in \eqref{eq:var-prob} for all $g\in Y$ (see also
\cref{exmp:pos-shift,ex:dv} below for an illustration).

In the rest of this section, we carry out the above program by giving an
explicit expression for $\con{\dif*\phi\nu}$ defined in
\eqref{eq:div-dual-prob} and showing that $\dif*\phi\nu$ is lower
semi-continuous. We will assume in the rest of this paper that $\dom\phi$
contains a neighborhood of $1$, as otherwise the $\phi$-divergence on
probability measures becomes the discrete divergence $\di\mu\nu =
\delta_{\set{\nu}}(\mu)$ which is only finite when $\mu=\nu$ and for which the
questions studied in this work are trivial.  We start with the following lemma
giving a simpler expression for $\dif*\phi\nu$.

\begin{lemma}
	\label{lem:encoding}
	Define $\pos\phi:x\mapsto \phi(x) + \delta_{\R_{\geq 0}}(x)$ for
	$x\in\R$. Then for all $\mu\in\M$
	\begin{displaymath}
		\dif*\phi\nu(\mu) = \dif{\pos\phi}\nu(\mu)
	+ \delta_{\set{1}}\paren[\big]{\mu(\Omega)}\,.
	\end{displaymath}
\end{lemma}

\begin{proof}
	Using the same notations as in \cref{def:div}, and since
	$\pos\phi'(-\infty)=-\infty$, it is easy to see that
	$\dif{\pos\phi}\nu(\mu)$ equals $+\infty$ whenever $\mu_s^-\neq 0$ or
	$\nu\paren[\big]{\set[\big]{\omega\in\Omega\given
	\frac{d\mu_c}{d\nu}(\omega)<0}}\neq 0$  and equals $\dif\phi\nu(\mu)$
	otherwise. In other words, $\dif{\pos\phi}\nu(\mu) = \dif\phi\nu(\mu)
	+ \delta_{\M^+}(\mu)$. This concludes the proof since $\delta_{\M^+}(\mu)
+ \delta_{\set{1}}\paren[\big]{\mu(\Omega)}=\delta_{\M^1}(\mu)$.
\end{proof}

In the expression of $\dif*\phi\nu$ given by \cref{lem:encoding}, the
non-negativity constraint on $\mu$ is ``encoded'' directly in the definition
of $\pos\phi$ (cf.~\citet{BL91}), only leaving the constraint
$\mu(\Omega)=1$ explicit. Since $\mu(\Omega)=\ex*\mu{\ind_\Omega}$, this is an
affine constraint which is well-suited to a convex duality treatment. In
particular, we can use \cref{prop:fenchelcor} to compute $\con{\dif*\phi\nu}$.

\begin{proposition}
	\label{prop:r-dual}
	Assume that $(X,Y)$ is a $\nu$-decomposable dual pair for some
	$\nu\in\M^1$. Then the convex conjugate of $\dif*\phi\nu$ with respect to
	$(X,Y)$ is given for all $g\in  Y$, by
	\begin{equation}
		\label{eq:dual-prob-unif}
		\con{\dif*\phi\nu}(g)
		= \inf\set*{\ex*\nu{\con{\pos\phi}(g+\lambda)}-\lambda
		\given \lambda+\esssup_\Xi g\leq \phi'(\infty)}\,,
	\end{equation}
	where $\pos\phi:x\mapsto\phi(x)+ \delta_{\R_{\geq 0}}(x)$
	and $\Xi = \set{A\in\A\given \forall \mu\in X,\; \abs\mu(A)=0}$.

	In \eqref{eq:dual-prob-unif} the infimum is reached if it is finite, which
	holds in particular whenever $g\in L^\infty(\Xi)$.
\end{proposition}

\begin{proof}
	We use \cref{lem:encoding} and apply
	\cref{prop:fenchelcor} with $f = \dif{\pos\phi}\nu$, $y=\ind_\Omega$ and
	$\eps=1$. We need to verify that
	$1\in\inter\paren[\big]{\set{\ex\mu{\ind_\Omega}\given \mu\in
	\dom\dif{\pos\phi}\nu}}$, but this is immediate since $(1\pm\alpha)\nu
	\in\dom\dif{\pos\phi}\nu$ for sufficiently small $\alpha$
	by the assumption that $1\in\inter\dom\phi$.

	Thus, by \cref{prop:fenchelcor}, for all $g\in Y$
	\begin{displaymath}
		\con{\dif*\phi\nu}(g)
		= \inf_{\lambda\in\R}\set*{\con{\dif{\pos\phi}\nu}(g+\lambda)
		- \lambda}
		\,,
	\end{displaymath}
	where the infimum is reached whenever it is finite.
	Equation~\eqref{eq:dual-prob-unif} follows by using
	\cref{prop:finite-dagger-dual} and observing that
	$\pos\phi'(\infty) = \phi'(\infty)$ and $\pos\phi'(-\infty)
	= -\infty$.

	It remains to verify the claims about finiteness of
$\con{\dif*\phi\nu}(g)$. For $g\in L^\infty(\Xi)$, write $M\eqdef \esssup_{\Xi}
g$.  Since $\inter(\dom\con{\pos\phi})=\paren[\big]{-\infty,\phi'(\infty)}$,
for any $A<\phi'(\infty)$, the choice of $\lambda = A-M$ makes the optimand in
\eqref{eq:dual-prob-unif} finite.
\end{proof}

\begin{remark}
	As in \cref{rem:par} above, when $X\subseteq\M_c(\nu)$ the constraint on
	$\lambda$ in \eqref{eq:dual-prob-unif} can be dropped, leading to a simpler
	expression for $\con{\dif*\phi\nu}(g)$ in this case. Indeed,
	$\overline{\dom\con{\pos\phi}} =\bigl(-\infty,\phi'(\infty)\bigr]$ and thus
	the optimand in \eqref{eq:dual-prob-unif} equals $+\infty$ whenever
	$\esssup_\Xi g = \esssup_\nu g>\phi'(\infty)-\lambda$.
\end{remark}
\begin{example}
\label{exmp:pos-shift}
	The effect of the restriction to probability measures is particularly
	pronounced for the total variation distance, which is the $\phi$-divergence
	for $\phi(x)=\abs{x-1}$. In the unrestricted case, a simple calculation
	shows $\phi$ has convex conjugate $\con\phi(x)=x + \delta_{[-1,1]}(x)$, so
	that the conjugate of the unrestricted divergence $\con{\dif\phi\nu}(g)$ is
	$+\infty$ unless $\essim\Xi g\subseteq[-1,1]$.  In the case of probability
	measures, the restriction $\pos\phi$ of $\phi$ to the non-negative reals
	has conjugate $\con{\pos\phi}(x)=x$ when $\abs x\leq 1$,
	$\con{\pos\phi}(x)=+\infty$ when $x>1$, but $\con{\pos\phi}(x)=-1$ when
	$x<-1$. Thus, $\con{\dif{\pos\phi}\nu}(g)<+\infty$ whenever $\essim\Xi
	g\subseteq (-\infty,1]$.  Furthermore, because of the additive $\lambda$
	shift in \cref{eq:dual-prob-unif}, we have $\con{\dif*\phi\nu}(g)<+\infty$
	whenever $\esssup_\Xi g<+\infty$, in particular whenever $g\in
	L^\infty(\Xi)$.
\end{example}

As a corollary, we obtain a different variational representation of
the $\phi$-divergence, valid for probability measures and containing as
a special case the Donsker--Varadhan representation of the Kullback--Leibler
divergence. 

\begin{corollary}
	\label{cor:dv}
	Assume that $(X,Y)$ is $\nu$-decomposable for some $\nu\in \M^1$. Then,
	$\dif*\phi\nu$ is lower semicontinuous over $X$. In particular for all
	probability measures $\mu\in X^1=X\cap\M^1$
			\begin{displaymath}
		\di\mu\nu
		= \sup_{g\in Y}\set*{
		\ex\mu g- \inf\set[\big]{\intf{\con{\pos\phi}}\nu(g+\lambda)-\lambda
	\given \lambda + \esssup_\Xi g \leq \phi'(\infty)}}
		,
	\end{displaymath}
	where $\pos\phi:x\mapsto\phi(x)+ \delta_{\R_{\geq 0}}(x)$
	and $\Xi = \set{A\in\A\given \forall \mu\in X,\; \abs\mu(A)=0}$.
\end{corollary}

\begin{proof}
Since $\ind_\Omega\in Y$ the linear form
$\mu\mapsto \ex\mu{\ind_\Omega}$ is continuous for any topology compatible
with the dual pair $(X, Y)$.  Consequently, the function
$\mu\mapsto\delta_{\set{1}}\paren[\big]{\mu(\Omega)}$ is lower
semicontinuous as the composition of the lower semicontinuous function
$\delta_{\set{1}}$ with a continuous function. Finally, $\dif{\pos\phi}\nu$
is lower semicontinuous by \cref{prop:i-dual,prop:finite-dagger-lsc}. Hence
$\dif*\phi\nu$ is lower semicontinuous as the sum of two lower
semicontinuous functions, by using the expression in \cref{lem:encoding}.
The variational representation immediately follows by expressing
$\dif*\phi\nu$ as its biconjugate.
\end{proof}

\begin{example}
	\label{ex:dv}
	As in \cref{ex:stupid-dv}, we consider the case of the Kullback--Leibler
	divergence, given by $\phi(x)=\pos\phi(x)=x\log x$. For a $\nu$-decomposable
	dual pair $(X,Y)$, since $\con\phi(x) = e^{x-1}$ \cref{prop:r-dual}
	implies for $\nu\in\M^1$ and $g\in Y$ that 
	\begin{displaymath}
		\con{\dif*\phi\nu}(g)
	= \inf_{\lambda\in\R}\ex*\nu{e^{g+\lambda-1}} - \lambda
	=\log\ex*\nu {e^g}\,,
	\end{displaymath}
	where the last equality comes from the optimal choice of
	$\lambda=-\log\ex*\nu{e^{g-1}}$. Using \cref{cor:dv} we obtain for all
	probability measure $\mu\in X^1$
	\begin{align*}
		\kl\mu\nu
		= \sup_{g\in Y} \set*{\ex\mu g - \log\ex*\nu{e^g}}
		= \sup_{g\in Y} \set*{\ex\mu g - \ex\nu
			g - \log\ex*\nu{e^{\paren[\big]{g-\ex\nu g}}}}
			\,,
	\end{align*}
	which is the Donsker--Varadhan representation of the Kullback--Leibler
	divergence \citep{DV76}.  For $\mu\in X^1$, the variational representation
	obtained in \eqref{eq:stupid-dv} can be equivalently written
	\begin{displaymath}
		\kl\mu\nu
		= \sup_{g\in Y} \set*{1+\ex\mu g - \ex*\nu{e^g}}\,.
	\end{displaymath}
	Using the inequality $\log(x)\leq x-1$ for $x>0$, we see that the optimand
	in the previous supremum is smaller than the optimand in the
	Donsker--Varadhan representation for all $g\in Y$. We thus obtained
	a ``tighter'' representation by restricting the divergence to probability
	measures.
\end{example}

\begin{example}
	Consider the family of divergences $\phi(x)=\abs{x-1}^\alpha/\alpha$ for
	$\alpha\geq 1$. A simple computation gives $\con\phi(y)=y + \abs
	y^\beta/\beta$ where $\beta\geq 1$ is such that
	$\frac1\alpha+\frac1\beta=1$. \citea{The paper}{JHW17} uses the variational representation
	given by \cref{prop:i-dual}, that is
	$\di\mu\nu=\sup_g \ex\mu g-\ex\nu{\con\phi(g)}$.
	However, \cref{cor:dv} shows that the tight representation uses
	$\con{\pos\phi}(y)$ which has the piecewise definition $y+\abs y^\beta/\beta$
	when $y\geq -1$ and the constant $-1/\alpha$ when $y\leq -1$,
	and writes 
	$\di\mu\nu=\sup_g \ex\mu g-\inf_\lambda\ex\nu{\con{\pos\phi}(g+\lambda)}$.
	Note that the additive $\lambda$ shift, in e.g.\ the case $\alpha=2$, reduces the second
	term from the raw second moment $\ex\nu{g^2}$ to something
	no larger than the variance $\ex\nu{(g-\ex\nu g)^2}$, which is
	potentially much smaller.
\end{example}
 \section{Optimal bounds for a single function and reference measure}
\label{sec:singleg}

As a first step to understand the relationship between a $\phi$-divergence and
an IPM, we consider the case of a single fixed probability measure $\nu\in\M^1$
and measurable function $g\in\L^0$, and study the optimal lower bound of
$\di\mu\nu$ as a function of the \emph{mean deviation} $\mu(g)-\nu(g)$. We
characterize this optimal lower bound and its convex conjugate in
\cref{sec:char} and then present implications for topological question
regarding the divergence itself in subsequent sections.

In the remainder of this work, since we are interested in probability measures,
which are in particular non-negative, we assume without loss of generality that
$\phi$ is infinite on the negative reals, that is $\phi(x) = \pos\phi(x)
= \phi(x) + \delta_{\R\geq 0}(x)$. As seen in \cref{sec:var-prob}
(in particular \cref{lem:encoding}), this does not change the value of the
divergence on non-negative measures, that is $\di[\phi]\mu\nu
=\di[\pos\phi]\mu\nu$ for $\mu\in\M^+$, but yields a tighter variational
representation since $\con{\pos\phi}\leq \con\phi$.

Furthermore, since for probability measures $\di\mu\nu$ is invariant to affine
shifts of the form $\widetilde\phi(x)=\phi(x)+c\cdot(x-1)$ for $c\in\R$, it
will be convenient to assume that $0\in\partial\phi(1)$ (e.g.~$\phi'(1)=0$),
equivalently that $\phi$ is non-negative and has global minimum at $\phi(1)=0$.
This can always be achieved by an appropriate choice of $c$ and is therefore
without loss of generality.  As an example, we now write for the
Kullback--Leibler divergence $\phi(x)=x\log x-x+1$ which is
non-negative with $\phi'(1)=0$, and equivalent to the standard definition
$\phi(x)=x\log x$ for probability measures.

\subsection{Derivation of the bound}
\label{sec:char}

We first define the optimal lower bound function, which comes in two flavors
depending on whether the mean deviation or the absolute mean deviation is
considered.

\begin{definition}
\label{defn:lb-function-single-g}
	For a probability measure $\nu\in\M^1$, a function $g\in\L^1(\nu)$, and set
	of probability measures $M$ integrating $g$, the \emph{optimal lower bound
	on $\di\mu\nu$ in terms of the mean deviation} is the function $\lb[M]
	g\nu$ defined for $\eps\in\R$ by:
	\begin{align}
		\lb[M]g\nu\paren*{\eps}
		&\eqdef
		\inf\set[\Big]{\di\mu\nu \given \mu\in M\land
	\ex\mu g-\ex\nu g=\eps}\nonumber\\
		&=\inf_{\mu\in M}\set[\big]{\di\mu\nu+\delta_{\set 0}(\ex\mu g
		-\ex\nu g-\eps)}
		\label{eqn:lb-optimization}\\
		\lb[M]{\set{\pm g}}\nu\paren*{\eps}
		&\eqdef
		\inf\set[\Big]{\di\mu\nu \given \mu\in M\land
		\abs{\ex\mu g-\ex\nu g}=\eps}\nonumber\\
		&=\min\set{\lb[M]g\nu(\eps),\lb[M]g\nu(-\eps)}
		\label{eqn:abslb}
	\end{align}
	where we follow the standard convention that the infimum of the empty set
	is $+\infty$.
\end{definition}

\begin{lemma}
\label{lem:lb-properties}
	For every $\nu\in\M^1$, $g\in\L^1(\nu)$, and convex set $M$ of probability
	measures integrating $g$, the function $\lb[M]g\nu$ is convex and
	non-negative. Furthermore, $\lb[M]g\nu(0)=0$ whenever $\nu\in M$,
	and if $\phi'(\infty)=\infty$ then $\lb[M]g\nu=\lb[M\cap \M_c(\nu)]g\nu$.
\end{lemma}
\begin{proof}
	Convexity is immediate from \cref{lem:optimal-value-convex}
	applied to \cref{eqn:lb-optimization}, non-negativity follows from non-negativity
	of $\di\cdot\nu$, the choice $\mu=\nu$ implies $\lb[M]g\nu(0)=0$ when
	$\nu\in M$, and if $\phi'(\infty)=\infty$ then $\di\mu\nu=+\infty$ when
	$\mu\in M\setminus\M_c(\nu)$.
\end{proof}

We compute the convex conjugate of $\lb g\nu$ by applying Fenchel duality to 
\cref{eqn:lb-optimization}.

\begin{proposition}
\label{prop:lb-conjugate}
Let $(X,Y)$ be a $\nu$-decomposable pair for some probability measure
$\nu\in\M^1$ and let $\Xi = \set{A\in\A\given \forall \mu\in X,\;
\abs\mu(A)=0}$.  Then for all $g\in Y$ and $t\in\R$,
	\begin{equation}
		\con{\lb[\,X^1] g\nu}(t)
		=\inf\set*{\ex*\nu{\con\phi(tg+\lambda)}
		-t\cdot \ex\nu g-\lambda
	\given \lambda+\esssup_\Xi(t\cdot g)\leq\phi'(\infty)}\,.
		\label{eqn:lbdual}
	\end{equation}
Furthermore, $\lb[\,X^1] g\nu (\eps)=\bicon{\lb[\,X^1] g\nu}(\eps)$ if and only
if strong duality holds in \cref{eqn:lb-optimization}.
\end{proposition}
\begin{proof}
	Define $\Phi:X\to\eR$ by $\Phi(x)=\dif*\phi\nu(x+\nu)$ so that $\Phi$ is
	convex, lsc, non-negative, and $0$ at $0$. Furthermore,
	$\con\Phi(h)=\con{\dif*\phi\nu}(h)-\ex\nu h$ for $h\in Y$, and  $\lb[X^1]
	g\nu(\eps)=\inf\set{\Phi(x)\given x\in X\wedge \ip xg=\eps}$. The result
	then follows by \cref{prop:fenchelcor,prop:r-dual}.
\end{proof}
\begin{remark}
	Since $\dom\con\phi\subseteq[-\infty,\phi'(\infty)]$, $\lambda$ is always
	implicitly restricted in \cref{eqn:lbdual} to satisfy $\lambda
	+ \esssup_\nu tg \leq \phi'(\infty)$. When $\Xi$ is a proper subset of the
	null $\sigma$-ideal of $\nu$, the constraint in \cref{eqn:lbdual} is
	stronger to account for measures in $X$ which are not continuous with
	respect to $\nu$.

	If $\phi'(\infty)=\infty$, then the infimum in \cref{eqn:lbdual} is taken
	over all $\lambda\in\R$ and in particular, does not depend on $\Xi$. This
	is consistent with the fact that, in this case, $\dif\phi\nu$ is infinite
	on singular measures, hence $\lb[\,X^1]g\nu = \lb[\,X^1_c(\nu)]g\nu$ where
	$X_c(\nu)= X\cap\M_c(\nu)$.
\end{remark}
\begin{remark}
	Unlike in \cref{prop:r-dual}, it is not always true that the
	interiority constraint qualification conditions hold, and indeed strong
	duality does not always hold for the optimization problem
	\eqref{eqn:lb-optimization}.  For example, for $\Omega=(-1/2,1/2)$, $\nu$
	the Lebesgue measure, $g$ the canonical injection into $\R$, and
	$\phi:x\mapsto \abs{x-1}$ corresponding to the total variation distance, we
	have $\lb[\M^1]g\nu(\pm 1/2)=\infty$ but $\lb[\M^1] g\nu(x)\leq 2$ for
	$\abs x<1/2$.  However, as noted in \cref{thm:equiv} below, this generally
	does not matter since it only affects the boundary of the domain of $\lb
	g\nu$, which contains at most two points. Furthermore, we will show in
	\cref{cor:L-lsc} via a compactness argument that when
$\phi'(\infty)=\infty$ and $\dom{\con{\lb g\nu}}=\R$---e.g.\ when $g\in
L^\infty(\nu)$---strong duality holds in \eqref{eqn:lb-optimization}.
\end{remark}

We can simplify the expressions in \cref{prop:lb-conjugate} by introducing the
function $\psi:x\mapsto\phi(x+1)$. We state some useful properties of its
conjugate $\con\psi$ below.
\begin{lemma}
The function $\con\psi:x\mapsto\con\phi(x)-x$ is non-negative, convex, and inf-compact.
Furthermore, it satisfies $\con\psi(0)=0$, $\con\psi(x)\leq -x$ when $x\leq 0$, and
$\inter(\dom\con\psi)=\paren[\big]{-\infty,\phi'(\infty)}$.
\label{lem:con-psi-properties}
\end{lemma}

Recall that at the beginning of \cref{sec:singleg} we assumed, without loss of
generality, that $0\in\partial\phi(1)$ and $\dom\phi\subseteq\R_{\geq
0}$, which is necessary for \cref{lem:con-psi-properties} to hold. The proof
  follows immediately from basic results in convex analysis on $\R$;  for
  completeness, a proof is included in \cref{sec:cgf-deferred}.

The right-hand side of  \cref{eqn:lbdual}, expressed in terms of
$\con\psi$, will be central to our theory, so we give it a name in the
following definition. 

\begin{definition}[Cumulant generating function]
	\label{def:cumulant}
	For a $\sigma$-ideal $\Xi$ and probability measure $\nu\in\M^1_c(\Xi)$, the
	\emph{$(\phi,\nu,\Xi)$-cumulant generating function} $\cgfxi
	g\nu\Xi:\R\to\overline\R$ of a function $g\in L^0(\Xi)$ is defined for all
	$t\in\R$ by
	\begin{equation}
		\label{eq:cumul}
		\cgfxi g\nu\Xi(t)
		\eqdef \inf\set*{\ex*\nu{\con\psi(tg+\lambda)}
		\given
		\lambda+\esssup_\Xi\paren[\big]{t\cdot g} \leq \phi'(\infty) }
		\,.
	\end{equation}
	Note that since $\nu\in\M_c(\Xi)$, we always have $\Xi\subseteq
	N\eqdef\set{A\in\A\given \nu(A)=0}$, hence $\cgfxi g\nu\Xi\geq \cgfxi g\nu
	N$. In the common case where $\Xi=N$ we abbreviate $\cgf g\nu\eqdef
	\cgfxi g\nu N$.
	
Note also that $\esssup_\Xi\paren[\big]{t\cdot g}$ is
the piecewise-linear function
	\begin{displaymath}
		\esssup_\Xi\paren[\big]{t\cdot g}
		=
		\begin{cases}
			t\cdot \esssup_\Xi g&t\geq 0\\
			t\cdot \essinf_\Xi g&t\leq 0
		\end{cases}
	\,.
	\end{displaymath}
\end{definition}

\begin{example}
	\label{ex:kl-cum}
	For the Kullback--Leibler divergence, $\cgf
	g\nu(t)=\log\ex\nu{e^{t(g-\ex\nu g)}}$ by \cref{ex:dv}, which is the
	standard (centered) cumulant generating function, thereby justifying the
	name.
\end{example}

Note that the $(\phi, \nu)$-cumulant generating function $\cgf g\nu$ 
depends only on the pushforward measure $\pushforward\nu g$ of $\nu$ through $g$.
In particular, when $\nu$ is the probability distribution of a random
variable $X$, as in \cref{ex:ipm-rv}, $\cgf g\nu(t)$ can be equivalently
written as
\begin{equation}
	\label{eq:cumul-expectation}
	\cgf g\nu(t)= \inf_{\lambda\in\R}\mathbb{E}[\con\psi(t\cdot
		g(X)+\lambda)]\,,
\end{equation}
highlighting the fact that $\cgf g\nu$ only depends on $g(X)$. This contrasts with
$\cgfxi g\nu\Xi$, for an arbitrary $\Xi\gg\nu$, for which the constraint on
$\lambda$ depends on the $\Xi$-essential range of $g$, which is not solely
a property of the random variable $g(X)$ since it can depend on the value of $g$
on $\nu$-null sets.

Furthermore, since for $t\in\R$, the function $\lambda\mapsto
\intf{\con\psi}\nu(tg+\lambda)$ is convex in $\lambda$, the
$(\phi,\nu)$-cumulant generating function is defined by a single-dimensional
convex optimization problem whose objective function is expressed as an
integral with respect to a probability measure
(\ref{eq:cumul},~\ref{eq:cumul-expectation}).  Hence, the rich spectrum of
stochastic approximation methods, such as stochastic gradient descent, can be
readily applied, leading to efficient numerical procedures to evaluate $\cgf
g\nu (t)$, as long as the pushforward measure $\pushforward\nu g$ is
efficiently samplable.

\begin{remark}
Since the mean deviation, and thus the optimal bound $\lb g\nu$ is invariant to
shifting $g$ by a constant, we are in fact implicitly working in the quotient
space $L^1(\nu)/\R\ind_\Omega$. As such,
$g\mapsto\inf_{\lambda\in\R}\intf{\con\psi}\nu(g+\lambda)$ can be interpreted
as the integral functional induced by $\intf{\con\psi}\nu$ on this quotient
space, by considering its infimum over all representatives of a given
equivalence class. This is analogous to the definition of a norm on a quotient
space.
\end{remark}

The following proposition states some basic properties of the cumulant
generating function. As with \cref{lem:con-psi-properties}, they follow from
basic results in convex analysis, and we defer the proof to
\cref{sec:cgf-deferred}.

\begin{proposition}
\label{prop:cgf-properties}
For every $\sigma$-ideal $\Xi$, probability measure $\nu\in\M^1_c(\Xi)$, and
$g\in L^0(\Xi)$, $\cgfxi g\nu\Xi:\R\to\overline\R$ is non-negative, convex,
lower semicontinuous, and satisfies $\cgfxi g\nu\Xi(0)=0$.

	Furthermore, if $g$ is not $\nu$-essentially constant then $\cgfxi g\nu\Xi$
	is inf-compact. If there exists $c\in\R$ such that $g=c$ $\nu$-almost
	surely, then there exists $t>0$ (resp.\ $t<0$) such that $\cgfxi
	g\nu\Xi(t)>0$ if and only if $\phi'(\infty)<\infty$ and $\esssup_\Xi g>c$
	(resp.\ $\essinf_\Xi g<c$).
\end{proposition}

With these definitions, we can state the main result of this
section giving an expression for the optimal lower bound function.

\begin{theorem}
\label{thm:equiv}
Let $(X,Y)$ be a $\nu$-decomposable pair for some probability measure
$\nu\in\M^1$ and let $\Xi = \set{A\in\A\given \forall \mu\in X,\;
\abs\mu(A)=0}$.  Then for all $g\in Y$ and $\eps\in\inter(\dom\lb[\,X^1] g\nu)$,
	\begin{equation}
		\lb[\,X^1] g\nu(\eps)=\con{\cgfxi g\nu\Xi}(\eps)\,.
		\label{eqn:lbdual-psi-cont}
	\end{equation}

	Furthermore, if $\lb[\,X^1] g\nu$ is lower semi-continuous, equivalently if
	strong duality holds in \eqref{eqn:lb-optimization}, then
	\eqref{eqn:lbdual-psi-cont} holds for all $\eps\in\R$.
\end{theorem}

\begin{proof}
	\Cref{lem:lb-properties} implies that $\lb[\,X^1]g\nu$ is proper and
	convex, thus, by the Fenchel--Moreau theorem, we have
	$\lb[\,X^1]g\nu=\bicon{\lb[\,X^1]g\nu}$ except possibly at the boundary of
	its domain, so this is simply a restatement of \cref{prop:lb-conjugate}
	using the terminology from \cref{def:cumulant}.
\end{proof}

\Cref{prop:lb-conjugate} and \cref{thm:equiv} show that the
conjugate of the optimal lower bound only depends on the space of measures $X$,
through the $\sigma$-ideal $\Xi$, as long as $X$ forms a decomposable dual pair
with a space $Y$ of functions containing $g$.  Hence, starting from
a $\sigma$-ideal $\Xi$ and a function $g$—or more generally a class of
functions $\G$—a natural dual pair to consider is the space $X\subseteq
\M_c(\Xi)$ of all measures integrating functions in $\G$, put in dual pairing
with the subspace of $L^0(\Xi)$ of all functions integrable by measures in $X$.
Formally, we have the following definition.

\begin{definition}
\label{defn:G-dual-pair}
Let $\G$ be a subset of $L^0(\Sigma)$ for some
$\sigma$-ideal $\Sigma$. We define
	\begin{align*}
		X_\G&\eqdef \set{\mu\in\M_c(\Sigma)
			\given \forall g\in\G, \ex{\abs\mu}{\abs g}<\infty}\\
	\text{and } Y_\G&\eqdef \set{h\in L^0(\Xi)
		\given \forall \mu\in X_\G, \ex{\abs\mu}{\abs h}<\infty}\,,
	\end{align*}
	where $\Xi \eqdef \set{A\in\A\given \forall\mu\in X_\G, \abs\mu(A)=0}$.

	For brevity, if $\G=\set g$ is a singleton, we write $X_g$
	for $X_{\set g}$ and $Y_g$ for $Y_{\set g}$.
\end{definition}

\begin{remark}
	We would like to use $\Sigma$ rather than $\Xi$ in the definition of $Y_\G$,
	but need to be careful since if $\Sigma\subsetneq\Xi$ then using $\Sigma$ would
	prevent $(X_\G, Y_\G)$ from being in separating duality. Unfortunately, there
	exist pathological $\sigma$-ideals for which $\Sigma\subsetneq \Xi$ (\citet[2.\
	Corollaire]{S34}), but since for non-pathological choices of $\Sigma$ (e.g.~when
	it is the null ideal of a $\sigma$-finite, semifinite, or $s$-finite measure) we
	indeed have $\Sigma = \Xi$, we do not dwell on this distinction.
\end{remark}

\begin{lemma}
	\label{lem:xg-yg-decomposable}
	Consider a subset $\G\subseteq L^0(\Sigma)$ for some $\sigma$-ideal $\Sigma$.
	Then for every $\nu\in X_\G^+$, the pair $\paren[\big]{X_\G, Y_\G}$ is
	$\nu$-decomposable.
\end{lemma}
\begin{proof}
	That $\mu(h)<\infty$ for all $\mu\in X_\G$ and $h\in Y_\G$ is by definition.
	As discussed in \cref{rem:decomp}, it suffices to verify that item $2$ in
	\cref{assmp:decomp} is true for all $\nu\in X_\G$, and indeed just for
	all $\nu\in X_\G^+$ since $\nu\in X_\G$ implies $\abs\nu\in X_\G^+$. Item $3$ and the
	separability of the duality between $(X_\G, Y_\G)$ then follow immediately.

	For item 2, consider $\nu\in X_\G^+$ and $\mu\in\M_c(\nu)$ such that
	$\frac{d\mu}{d\nu}\in L^\infty(\nu)$. Then for all $g\in\G$ we have
	$\ex{\abs\mu}{\abs g} =\ex{\nu}{\abs*{\frac{d\mu}{d\nu}}\cdot\abs g}<\infty$
	by Hölder's inequality, hence $\mu\in X_\G$. That $L^\infty(\Xi)\subseteq Y_\G$
	holds is immediate since every $\mu\in\M_c(\Sigma)$ integrates every $h\in
	L^\infty(\Xi)$.
\end{proof}

The following easy corollary is an ``operational'' restatement of
\cref{thm:equiv}, specialized to the dual pair of \cref{defn:G-dual-pair}, and
highlighting the duality between upper bounding the cumulant generating
function and lower bounding the $\phi$-divergence by a convex lower
semicontinuous function of the mean deviation.

\begin{corollary}
	\label{cor:equiv-operational}
	Consider a measurable function $g\in L^0(\Sigma)$ for some $\sigma$-ideal
	$\Sigma$ and let $\Xi = \set{A\in\A\given \forall\mu\in X_g,
	\abs\mu(A)=0}$.
	Then for every probability measure $\nu\in X_g^1\eqdef X_g\cap\M^1$ and
	convex lower semicontinuous function $L:\R\to\eR_{\geq 0}$, the following
	are equivalent:
	\begin{enum}
	\item $\di\mu\nu\geq L(\ex\mu g-\ex\nu g)$ for every $\mu\in X^1_g$.
	\item $\cgfxi g\nu\Xi\leq \con L$.
	\end{enum}
\end{corollary}

\begin{example}
	\label{exmp:hcr}
	The Hammersley--Chapman--Robbins bound in statistics is an
	immediate corollary of \cref{cor:equiv-operational} applied to the
	$\chi^2$-divergence given by $\phi(x)=(x-1)^2 +\delta_{\R_\geq 0}(x)$: The
	convex conjugate of $\psi(x)=x^2 + \delta_{[-1,\infty)}(x)$ is
	\[\con\psi(x)=\begin{cases}x^2/4&x\geq -2\\-1-x&x<-2\end{cases}\] and
	satisfies in particular $\con\psi(x)\leq x^2/4$, so that $\cgf
	g\nu(t)\leq\inf_\lambda \ex*\nu{(tg+\lambda)^2/4}
	=t^2\operatorname{Var}_\nu(g)/4$. Since the convex conjugate of $t\mapsto
	t^2\operatorname{Var}_\nu(g)/4$ is $t\mapsto
	t^2/\operatorname{Var}_\nu(g)$, we obtain for all $\mu,\nu\in\M^1$ and
	$g\in L^1(\nu)$ that $\chisq\mu\nu\geq (\ex\mu g-\ex\nu
	g)^2/\operatorname{Var}_\nu(g)$.
\end{example}

\Cref{thm:equiv} also gives a useful characterization of the existence of
a non-trivial lower bound by the \emph{absolute} mean deviation.

\begin{corollary}
	\label{cor:subexp-iff-lb}
	Consider a measurable function $g\in L^0(\Sigma)$ for some $\sigma$-ideal
	$\Sigma$ and let $\Xi = \set{A\in\A\given \forall\mu\in X_g,
	\abs\mu(A)=0}$.
	Then for every $\nu\in X_g^1$, the optimal lower bound
	$\lb[\,X^1_g]{\set{\pm g}}\nu$ is non-zero if and only if
	$\,0\in\inter(\dom\cgfxi g\nu\Xi)$. In other words, the following are
	equivalent
	\begin{enum}
	\item there exists a non-zero function $L:\R_{\geq 0}\to\overline\R_{\geq
	0}$ such that $\di\mu\nu\geq L\paren[\big]{\abs{\ex\mu g-\ex\nu g}}$ for
		every $\mu\in X^1_g$.
	\item the function $\cgfxi g\nu\Xi$ is finite on an open interval around $0$.
	\end{enum}
\end{corollary}
\begin{proof}
	Writing $M=\,X^1_g$, we have by \cref{eqn:abslb} that the function
	$\lb[M]{\set{\pm g}}\nu$ is non-zero if and only if there exists $\eps>0$
	such that $\lb[M]g\nu(\eps)\neq 0\neq \lb[M]g\nu(-\eps)$.  Since
	$\lb[M]g\nu$ is convex, non-negative, and $0$ at $0$ by
\cref{lem:lb-properties}, such an $\eps$ exists if and only if $0$ is contained
in the interval $\paren[\big]{\lb[M]g\nu'(-\infty),\lb[M] g\nu'(\infty)}$, the
interior of the domain of $\con{\lb[M]g\nu}$.
\end{proof}

\begin{remark}
	Throughout this section, we have seen the $\sigma$-ideal $\Xi$ appear in
	our results, in particular via $\cgfxi g\nu\Xi$ in
	\cref{cor:subexp-iff-lb}.  We will see in \cref{thm:ipm-lb-smoothed}
	however that when we consider a true IPM where we require the bound $L$ to
	hold jointly for all measures $\nu$ and $\mu$, we can ignore the
	$\sigma$-ideal $\Xi$ and consider only $\cgf g\nu$.
\end{remark}

\subsection{Subexponential functions and connections to Orlicz spaces}
\label{sec:subexponential}

In \cref{sec:subexponential,sec:inf-compact,sec:continuity}, we explore
properties of the set of functions satisfying the conditions of
\cref{cor:subexp-iff-lb}, i.e.\ for which there is a non-trivial lower bound of
the $\phi$-divergence in terms of the absolute mean deviation, and show its
relation to topological properties of the divergence. A reader primarily
interested in quantitative bounds for IPMs can skip to \cref{sec:ipm}.

In light of \cref{cor:subexp-iff-lb}, we need to consider the set of functions
$g$ such that $\dom \cgfxi g\nu\Xi$ contains a neighborhood
of zero. The following lemma shows that this is the case for bounded functions,
and that furthermore, when $\phi'(\infty)<\infty$, boundedness is necessary. In
other words, when $\phi'(\infty)<\infty$, the $\phi$-divergence cannot upper
bound the absolute mean deviation of an unbounded function. This is in sharp
contrast with the KL divergence (satisfying $\phi'(\infty)=\infty$), for which
such upper bounds exist as long as the function satisfies Gaussian-type tail
bounds \citep[\S 4.10]{BLM13}.

\begin{lemma}
	\label{lem:lb-finitedagger-bounded}
	\label{lem:linfty-strong-subexponential}
	Let $\Xi$ be a $\sigma$-ideal and $\nu\in\M^1_c(\Xi)$.
	If $g\in L^\infty(\Xi)$ then $\dom\cgfxi g\nu\Xi$ is
	all of $\R$, and in particular contains a neighborhood of zero. Furthermore,
	when $\phi'(\infty)<\infty$, we have conversely that if $0$ is in the
	interior of the domain of $\cgfxi g\nu\Xi$, then
	$g\in L^\infty(\Xi)$, in which case $\cgf
	g\nu(t)=\cgfxi g\nu\Xi(t)$ whenever $\abs t\cdot\paren[\big]{\esssup_\Xi
	g-\essinf_\Xi g} \leq \phi'(\infty)$.
\end{lemma}

\begin{remark}
	As already discussed, \cref{lem:lb-finitedagger-bounded} implies that when
	$\phi'(\infty)<\infty$, boundedness of $g$ is necessary for the existence
	of a non-trivial lower bound on $\di\mu\nu$ in terms of the $\abs{\ex\mu
	g -\ex\nu g}$. Moreover, we can deduce from
	\cref{lem:lb-finitedagger-bounded} that in this case, any non-trivial lower
	bound must depend on $\esssup_\Xi \abs g$ and cannot depend only on
	properties of $g$ such as its $\nu$-variance. In particular, any
	non-trivial lower bound must converge to 0 as $\esssup_\Xi \abs g$ converges
	to $+\infty$, for if it were not the case, one could obtain a non-trivial
	lower bound for an unbounded function $g$ by approximating it with bounded
	functions $g\cdot\ind\set{\abs g\leq n}$.
\end{remark}

\begin{proof}
	Recall that $(-\infty, 0]\subseteq \dom\con\psi$ and that $\con\psi(x)\leq
	-x$ for $x\leq 0$ by \cref{lem:con-psi-properties}. For $g\in L^\infty(\Xi)$,
	write $B$ for $\esssup_\Xi\abs g$, and for $t\in\R$, write $\lambda
	\eqdef - \abs t\cdot B$.  Then we have that $-2\abs t B\leq t\cdot
	g+\lambda\leq 0\leq\phi'(\infty)$ holds $\Xi$-a.s.,
	and thus also $\con\psi(tg+\lambda)\leq 2\abs
	tB$ holds $\Xi$-a.s. Thus $\cgfxi g\nu\Xi(t)$
	is at most $2\abs t\cdot B<\infty$ by definition, and since $t$ is
	arbitrary, we get $\dom\cgfxi g\nu\Xi=\R$.

	We now assume $\phi'(\infty)<\infty$ and prove the converse claim. If $\cgfxi
	g\nu\Xi(t)$ is finite for some $t\in\R$, then $tg+\lambda\leq\phi'(\infty)$
	holds $\Xi$-a.s.\ for some $\lambda\in\R$.  In particular, if it holds for
	some $t>0$, then $\esssup_\Xi g$ is finite, and if it holds for some $t<0$,
	then $\essinf_\Xi g$ is finite.

	For the remaining claim, since $\con\psi$ is non-decreasing on the
	non-negative reals we have that $\cgf g\nu (t)
	=\inf\set[\big]{\intf{\con\psi\!}\nu(tg+\lambda)\given\lambda\in\R}
	=\inf\set[\big]{\intf{\con\psi\!}\nu(tg+\lambda)\given \essinf_\Xi
	tg+\lambda\leq 0}$. But if $\esssup_\Xi\paren[\big]{t\cdot g}-
	\essinf_\Xi\paren[\big]{t\cdot g}\leq \phi'(\infty)$, then $\essinf_\Xi t\cdot
	g+\lambda\leq 0$ implies $\esssup_\Xi tg+\lambda\leq \phi'(\infty)$ and $\cgf
	g\nu (t)\geq \cgfxi g\nu\Xi(t)\geq \cgf g\nu (t)$.
\end{proof}

Since \cref{lem:lb-finitedagger-bounded} completely characterizes the existence
of a non-trivial lower bound when $\phi'(\infty)<\infty$, we focus on the case
$\phi'(\infty)=\infty$ in the remainder of this section. Recall that $\cgf g\nu
= \cgfxi g\nu\Xi$ in this case, so we only need to consider $\cgf g\nu$ in the
following definition.

\begin{definition}[$(\phi,\nu)$-subexponential functions]
\label{defn:subexponential}
Let $\nu\in\M^1$ be a probability measure. We say that the function
$g\in L^0(\nu)$ is \emph{$(\phi,\nu)$-subexponential} if
$0\in\inter(\dom\cgf g \nu)$ and we denote by $\subexp(\nu)$ the space of all such functions.
We further say that $g\in L^0(\nu)$ is \emph{strongly
$(\phi,\nu)$-subexponential} if $\dom\cgf g\nu=\R$ and denote by
$\ssubexp(\nu)$ the space of all such functions.
\end{definition}
\begin{example}
	For the case of the KL-divergence, if the pushforward $\pushforward\nu g$ of
	$\nu$ induced by $g$ on $\R$ is the Gaussian distribution (respectively the
	gamma distribution), then $g$ is strongly subexponential (respectively
	subexponential). Furthermore, it follows from \cref{ex:kl-cum} that
	$g\in\subexp(\nu)$ iff the moment-generating function of $g$ is finite on a
	neighborhood of $0$,  which is the standard definition of subexponential
	functions (see e.g.\ \citet[\S2.7]{V18}) and thus justifies our terminology.
\end{example}
\begin{example}
	\label{ex:bounded-implies-ssubexp}
	\cref{lem:lb-finitedagger-bounded} shows that $L^\infty(\nu)\subseteq
	\ssubexp(\nu)$ and that furthermore, if $\phi'(\infty)<\infty$, then
	$L^\infty(\nu)=\ssubexp(\nu)=\subexp(\nu)$.
\end{example}

We start with the following key lemma allowing us to relate the finiteness of
$\cgf g\nu$ to the finiteness of the function
$t\mapsto\intf{\con\psi\!}\nu(tg)$.

\begin{lemma}
	\label{lem:shrink}
	For $\nu\in\M^1$, $g\in L^0(\nu)$, and $t\in\dom\cgf g\nu$,
	we have that if $\phi'(\infty)=\infty$ (resp.\ $\phi'(\infty)>0$)
	then $\alpha tg\in\dom\intf{\con\psi\!}\nu$ for all $\alpha\in(0,1)$
	(resp.\ for sufficiently small $\alpha>0$).
\end{lemma}

\begin{proof}
	Let $\lambda\in\R$ be such that
	$\ex*\nu{\con\psi(tg+\lambda)}<\infty$ (such a $\lambda$ exists
	since $t\in\dom\cgf g\nu$). Using the convexity of $\con\psi$, we get
	for any $\alpha\in(0,1)$
	\begin{align*}
		\ex*\nu{\con\psi(\alpha tg)}
		&=\ex*\nu{\con\psi\paren*{\alpha
		(tg+\lambda)+(1-\alpha)\frac{-\alpha\lambda}{1-\alpha}}}\\
		&\leq \alpha\ex*\nu{\con\psi(tg+\lambda)}
		+ (1-\alpha)\con\psi\paren*{\frac{-\alpha\lambda}{1-\alpha}}
		\,.
	\end{align*}
	The first summand is finite by definition, and if
	$-\alpha\lambda/(1-\alpha)\in \dom\con\psi\supseteq (-\infty,\phi'(\infty))$
	then so is the second summand. If $\phi'(\infty)=\infty$ this holds
	for all $\alpha\in(0,1)$, and if $\phi'(\infty)>0$ it holds for
	sufficiently small $\alpha>0$.
\end{proof}

\begin{remark}
	When $\phi'(\infty)<\infty$, it is not necessarily true that 
	any $\alpha\in(0,1)$ can be used in \cref{lem:shrink}. For example,
	\cref{lem:linfty-strong-subexponential} implies that
	$\dom\cgf g\nu=\R$ for all $g\in L^\infty(\nu)$, but since
	$\dom\con\psi\subseteq\bigl(-\infty,\phi'(\infty)\bigr]$ we have
	$\intf{\con\psi\!}\nu(tg)=\infty$ for sufficiently large 
	(possibly only positive or negative) $t$, unless $g$ is
	zero $\nu$-a.s.
\end{remark}

The following proposition gives useful characterizations of subexponential
functions in terms of the finiteness of different integral functionals of
$g$.

\begin{proposition}
\label{prop:interval-cgf-implies-absolute}
	Suppose that $\phi'(\infty)=\infty$ and fix $\nu\in\M^1$ and $g\in L^0(\nu)$.
	Then the following are equivalent:
	\begin{enum}
		\item $g$ is $(\phi,\nu)$-subexponential
		\item $\cgf{\abs g}\nu(t)<\infty$ for some $t>0$
		\item $g\in L^\theta(\nu)$ for $\theta:x\mapsto\max\set{\con\psi(x),\con\psi(-x)}$
			(here $L^\theta(\nu)$ is the Orlicz space defined in \cref{sec:orlicz-prelims})
	\end{enum}
\end{proposition}

\begin{proof}
	$(i)\implies(ii)$\quad If $\dom\cgf g\nu$ contains an open interval around
	$0$, \cref{lem:shrink} and the convexity of $\dom\intf{\con\psi}\nu$ imply
	that there exists $s>0$ such that $\ex*\nu{\con\psi(tg)}<\infty$ for all
	$\abs t<s$.  By non-negativity of $\con\psi$, $\ex*\nu{\con\psi(t\abs g)}
	\leq\ex*\nu{\con\psi(tg)+\con\psi(-tg)}<\infty$ for all $t\in(-s,s)$, which
	in turns implies $(-s,s)\subseteq \dom\cgf{\abs g}\nu$.

	$(ii)\implies(iii)$\quad Define $\eta(x)\eqdef\con\psi(\abs x)$. Since
	$\con\psi(x)\leq -x$ for $x\leq 0$ by \cref{lem:con-psi-properties}, we
	have that $\eta(x)\leq\theta(x)\leq \eta(x)+\abs x$ for all $x\in\R$. Since
	we also have $L^\eta(\nu)\subseteq L^1(\nu)$, this implies that
	$g\in L^\eta(\nu)$ if and only if $L^\theta(\nu)$. We conclude after
	observing that $\cgf{\abs g}\nu(t)<\infty$ for some $t>0$ implies that
	$g\in L^\eta(\nu)$ by \cref{lem:shrink}.

	$(iii)\implies(i)$\quad Observe that for all $t\in\R$,
	\begin{equation}\label{eq:cgf-theta}
	\max\set*{\cgf g\nu(t),\allowbreak\cgf g\nu(-t)}
	\leq\max\set*{\ex*\nu{\con\psi(tg)}, \ex*\nu{\con\psi(-tg)}}
	\leq \ex*\nu{\theta(tg)}
	\,,
	\end{equation}
	where the first inequality is by definition of $\cgf g\nu$ and the second
	inequality is by monotonicity of the integral and the definition of
	$\theta$. Since $\dom \cgf g\nu$ is convex, if there exists $t>0$ such that
	$\intf\theta\nu(tg)<\infty$, then \eqref{eq:cgf-theta} implies that
	$[-t,t]\subseteq \dom\cgf g\nu$ and $g$ is $(\phi,\nu)$-subexponential.
\end{proof}

\begin{remark}
\label{rmrk:cgf-modular-orlicz-norm}
	Though \cref{prop:interval-cgf-implies-absolute} implies that the set of
	$(\phi,\nu)$-subexponential functions is the same as the set $L^\theta(\nu)$
	for $\theta(x)=\max\set{\con\psi(x),\con\psi(-x)}$, we emphasize that the
	Luxemburg norm $\norm{\cdotarg}_\theta$ does \emph{not} capture the
	relationship between $\di\mu\nu$ and the absolute mean deviation
	$\abs{\ex\mu g - \ex\nu g}$. First, the
	function $\theta$, being a symmetrization of $\con\psi$, induces integral
	functionals which are potentially much larger than those defined by
	$\con\psi$, in particular it is possible to have $\max\set{\cgf g\nu(t),\cgf
	g\nu(-t)} < \inf_{\lambda\in\R}\intf\theta\nu(tg+\lambda) <
	\intf\theta\nu(tg)$.
	Furthermore, the Luxemburg norm summarizes the growth
	of $t\mapsto \intf\theta\nu(tg)$ with a single number (specifically its
	inverse at $1$), whereas \Cref{thm:equiv} shows that the relationship
	with the mean deviation is controlled by $\con{\cgf g\nu}$,
	which depends on the growth of $\cgf g\nu(t)$ with $t$.
\end{remark}

We are now ready to prove the main result of this section, which is that the
space $\subexp(\nu)$ of $(\phi,\nu)$-subexponential functions is the largest
space of functions which can be put in dual pairing with (the span of) all
measures $\mu\in \M_c(\nu)$ such that $\di\mu\nu<\infty$, i.e.\
$\dom\intf\phi\nu$.

\begin{theorem}
	\label{prop:subexp-pairing}
	For $\nu\in\M^1$ and $g\in L^0(\nu)$, the following are equivalent:
	\begin{enum}
		\item $g$ is $(\phi,\nu)$-subexponential, i.e.\ $g\in\subexp(\nu)$.
		\item $g$ is $\mu$-integrable for every $\mu\in \M_c(\nu)$
	with $\di\mu\nu<\infty$.
		\item $g$ is $\mu$-integrable for every $\mu\in \M^1_c(\nu)$
	with $\di\mu\nu<\infty$.
	\end{enum}
\end{theorem}

\begin{proof}
	$(i)\implies(ii)$\quad
	If $\phi'(\infty)<\infty$ this follows since $L^\infty(\nu)=\subexp(\nu)$,
	so assume that $\phi'(\infty)=\infty$.
	If $g\in\subexp(\nu)$ then $g\in L^\theta(\nu)$ for
	$\theta(x)=\max\set{\con\psi(x), \con\psi(-x)}$ by
	\cref{prop:interval-cgf-implies-absolute}. Since $\theta\geq\con\psi$ we have
	$\con{\theta}\leq\psi$, and thus for $\mu\in\M_c(\nu)$ with
	$\di\mu\nu<\infty$,
	\begin{displaymath}
		\intf{\con\theta\!}\nu\paren*{\frac{d\mu}{d\nu}-1}
		\leq\intf\psi\nu\paren*{\frac{d\mu}{d\nu}-1}
	=\di\mu\nu<\infty\,,
	\end{displaymath}
	implying that $\frac{d\mu}{d\nu}-1 \in L^{\con{\theta}}(\nu)$. Furthermore,
	since $1\in L^\infty(\nu)\subseteq L^{\con\theta}(\nu)$ we get that
	$\frac{d\mu}{d\nu}\in L^{\con\theta}(\nu)$. Property 2.\ then follows from the fact that
	$(L^{\con\theta},L^{\theta})$ form a dual pair.

	$(ii)\implies(iii)$\quad Immediate.

	$(iii)\implies(i)$\quad Define $C\eqdef\set*{\mu\in\M^1_c(\nu)\given
	\di\mu\nu\leq 1}$, which is closed and convex as a sublevel
	set of the convex lower semicontinuous functional $\dif*\phi\nu$ on the Banach
	space $\M_c(\nu)$ with the total variation norm (recall that this
	space is isomorphic to $L^1(\nu)$ by the Radon--Nikodym theorem).
	Since furthermore $C\subseteq \M^1$, it is bounded in $\M_c(\nu)$ and so is
	cs-compact \citep[Proposition 2]{J72}. Then by assumption, the linear function
	$\mu\mapsto \ex\mu{\abs g}$ is well-defined and bounded below by $0$ on
	$C$, so \cref{lem:convex-bounded-cs-compact} implies that there exists
	$B\in\R$ such that $\ex\mu{\abs g}\leq B$ for all $\mu\in C$.
	Thus, we get that for all $\mu\in C$,
	$\abs{\ex\mu g-\ex\nu g} \leq\ex\mu{\abs g}+\ex\nu{\abs g}\leq B+\ex\nu{\abs g}$.
	In particular, if $\abs{\ex\mu g - \ex\nu g}> B + \ex \nu{\abs g}$ then
	$\di\mu\nu > 1$, proving the existence of a non-zero function $L$ such that
	$\di\mu\nu\geq L\paren[\big]{\abs{\ex\mu g - \ex\nu g}}$. This implies that
	$g\in\subexp(\nu)$ by \cref{cor:subexp-iff-lb}.
\end{proof}

We have the following characterization of the space $\ssubexp(\nu)$ of strongly
subexponential functions. In particular $\ssubexp(\nu)$ can be identified as
a set with $L^\infty(\nu)$ or the Orlicz heart $\orhrt\theta(\nu)$ depending on
whether $\phi'(\infty)$ is finite or infinite (with the finite case
from \cref{lem:lb-finitedagger-bounded}).

\begin{proposition}
\label{prop:strong-subexponential}
	Suppose that $\phi'(\infty)=\infty$ and fix $\nu\in\M^1$ and $g\in L^0(\nu)$.
	Then the following are equivalent:
	\begin{enum}
		\item $g$ is strongly $(\phi,\nu)$-subexponential, i.e.\ $g\in\ssubexp(\nu)$.
		\item $\cgf{\abs g}\nu(t)<\infty$ for all $t>0$.
		\item $g\in \orhrt\theta(\nu)$ for $\theta:x\mapsto\max\set{\con\psi(x),\con\psi(-x)}$.
	\end{enum}
\end{proposition}
\begin{proof}
	$(i)\implies(ii)$\quad Since $\phi'(\infty)=\infty$, \cref{lem:shrink}
	implies that $tg\in\dom\intf{\con\psi\!}\nu$ for all $t\in\R$, and since
	$\con\psi$ is non-negative we have for each $t>0$ that $\cgf{\abs
	g}\nu(t)\leq \ex*\nu{\con\psi(t\abs g)}\leq
	\ex*\nu{\con\psi(tg)+\con\psi(-tg)}<\infty$.

	$(ii)\implies(iii)$\quad Define $\eta:x\mapsto\con\psi(\abs x)$, so that
	by \cref{lem:shrink} we have $\ex*\nu{\eta(tg)}=\ex*\nu{\con\psi(t\abs g)}
	<\infty$ for all $t>0$, and hence Property 2.\ implies
	$g\in \orhrt\eta(\nu)$. As in the proof of
	\cref{prop:interval-cgf-implies-absolute}, 
	$\eta(x)\leq\theta(x)\leq\eta(x)+\abs x$ for all $x\in\R$ and since
	$\orhrt\eta(\nu)\subseteq L^1(\nu)$, we have that $g\in
	\orhrt\eta(\nu)$ iff $g\in \orhrt\theta(\nu)$.

	$(iii)\implies(i)$\quad Immediate since for $t\in\R$, $\cgf
	g\nu(t)\leq\ex*\nu{\con\psi(tg)} \leq\ex*\nu{\theta(tg)}<\infty$.
\end{proof}

Finally, we collect several statements from this section and express them in
a form which will be convenient for subsequent sections.

\begin{corollary}
	\label{cor:topological-summary}
	Define
	$\theta(x)\eqdef\max\set{\con\psi(x),\con\psi(-x)}$. Then we have
	$\ssubexp(\nu)\subseteq\subexp(\nu)\subseteq L^1(\nu)$ and
	$\dom\intf\phi\nu\subseteq L^{\con\theta}(\nu)\subseteq L^1(\nu)$.
	Furthermore, $L^{\con\theta}(\nu)$ is in dual pairing with both
	$\subexp(\nu)$ and $\ssubexp(\nu)$, and when $\phi'(\infty)=\infty$ the
	topology induced by $\norm{\cdotarg}_{\theta}$ on $\ssubexp(\nu)$ is
	complete and compatible with the pairing.
\end{corollary}

\begin{proof}
	The containment $\subexp(\nu)\subseteq L^1(\nu)$ is because $\subexp(\nu)$
	is equal as a set to the Orlicz space $L^\theta(\nu)$ by
	\cref{prop:interval-cgf-implies-absolute}, and the containment
	$\dom\intf\phi\nu\subseteq L^{\con\theta}(\nu)$ can be found in the proof of
	$(i)\implies(ii)$ of \cref{prop:subexp-pairing}. The fact that
	$\paren[\big]{L^{\con\theta}(\nu), \subexp(\nu)}$ form a dual pair is also
	immediate from the identification of $\subexp(\nu)$ with $L^\theta(\nu)$
	as a set. Finally, the last claim follows from the identification of
	$\ssubexp(\nu)$ with $\orhrt{\theta}(\nu)$ as a set and the fact that when
	$\phi'(\infty)=\infty$, then $\dom\theta=\R$ implying that the topological dual of
	the Banach space $(\orhrt\theta(\nu),\norm{\cdotarg}_\theta)$ is isomorphic to
	$(L^{\con\theta}(\nu),\norm{\cdotarg}_{\con\theta})$.
\end{proof}

\subsection{Inf-compactness of divergences and connections to strong
duality}
\label{sec:inf-compact}

In this section, we study the question of inf-compactness of the functional
$\dif\phi\nu$ and that of its restriction $\dif*\phi\nu$ to probability
measures. Specifically, we wish to understand under which topology the
information ``ball'' $\B_{\phi,\nu}(\tau)\eqdef
\set{\mu\in\M\given\di\mu\nu\leq\tau}$ is compact.  Beyond being a natural
topological question, it also has implications for strong duality in
\cref{thm:equiv}, since the following lemma shows that compactness of the ball
under suitable topologies implies strong duality.

\begin{lemma}
\label{lem:lb-lsc-if-compact}
	For every $g$, $\nu$, and $M$ as in \cref{defn:lb-function-single-g},
	if $\mu\mapsto \di\mu\nu$ is inf-compact (or even countably
	inf-compact) with respect to a topology on $M$ such that
	$\mu\mapsto\ex\mu g$ is continuous, then $\lb[M]g\nu$ is inf-compact
	(and in particular lower semicontinuous), so that strong duality holds
	in \cref{thm:equiv}.
\end{lemma}
\begin{proof}
	Recall from \cref{eqn:lb-optimization} that
	\[
		\lb[M]g\nu(\eps)
			=\inf_{\mu\in M}\di\mu\nu+\delta_{\set 0}(\ex\mu g-\ex\nu g-\eps)
	\]
	where $f(\eps,\mu)= \di\mu\nu + \delta_{\set 0}(\ex\mu g-\ex\nu g-\eps)$
	is convex. Furthermore, under the stated assumption, we have that
	$f$ is also inf-compact so that \cref{lem:optimal-value-convex}
	gives the claim.
\end{proof}

Throughout this section, we assume that $\phi'(\infty)=\infty$,\footnote{When
$\phi'(\infty)<\infty$, compactness of information balls is very dependent on
the specific measure space $(\Omega,\A,\nu)$, and in this work we avoid such
conditions.} which implies that $\dom\con\psi=\R$ by
\cref{lem:con-psi-properties}, and furthermore that $\mu\in\M_c(\nu)$ whenever
$\di\mu\nu<\infty$ and hence $\dif\phi\nu=\intf\phi\nu$ and
$\B_{\phi,\nu}(\tau)\subset \M_c(\nu)$ for all $\tau\geq 0$.  It is well known
that in this case, $\B_{\phi,\nu}(\tau)$ is compact in the weak topology
$\sigma\paren[\big]{L^1(\nu),L^\infty(\nu)}$ (e.g.\ \citet[Corollary 2B]{R71} or
\citet{TV93}). This fact can be derived as a simple consequence of the
Dunford--Pettis theorem since $\B_{\phi,\nu}(\tau)$ is uniformly integrable by
the de la Vallée-Poussin theorem (see e.g.\ \citet[pages 67--68]{Valadier70}). In
light of \cref{lem:lb-lsc-if-compact}, it is however useful to understand
whether $\B_{\phi,\nu}(\tau)$ is compact under topologies for which
$\mu\mapsto\ex\mu g$ is continuous, where $g$ could be unbounded.
Léonard \citep[Theorem 3.4]{L01_Inverse} showed, in the context of
convex integral functionals on Orlicz spaces, that
strong duality holds when $g\in\ssubexp(\nu)$, and in this section we
reprove this result in the language of $\phi$-divergences
by noting (as is implicit in \citet[Lemma 3.1]{L01_Inverse}) that $\B_{\phi,\nu}(\tau)$ is compact for
the initial topology induced by the maps of the form $\mu\mapsto\ex\mu
g$ for all strongly subexponential function $g\in\ssubexp(\nu)$.

\begin{proposition}
	\label{prop:inf-compact-small-orlicz}
	Fix $\nu\in \M^1$ and define
	$\theta:x\mapsto\max\set{\con\psi(x),\con\psi(-x)}$ as in
	\cref{prop:strong-subexponential}. If $\phi'(\infty) = \infty$, then the
functional $\intf\phi\nu$ is $\sigma\paren[\big]{L^{\con\theta}(\nu), \ssubexp(\nu)}$
	inf-compact.
\end{proposition}

\begin{proof}
By \cref{cor:topological-summary}, we know that $\paren[\big]{\ssubexp(\nu),
\norm{\cdotarg}_\theta}$ is a Banach space in dual pairing with
$L^{\con\theta}(\nu)$.  Thus, from \cref{prop:i-dual}, the integral functional
$\intf{\con\phi}\nu$ defined on $\ssubexp(\nu)$ is convex, lower
semicontinuous, and has conjugate $\intf{\bicon\phi}\nu=\intf\phi\nu$ on
$L^{\con\theta}(\nu)$. Furthermore, from \cref{lem:shrink} we know for every
$g\in\ssubexp(\nu)$ that $\intf{\con\phi}\nu(g)<\infty$, so
$\intf{\con\phi}\nu$ is convex, lsc, and finite everywhere on a Banach space,
and thus continuous everywhere by \citet[2.10]{B64}.  Finally, \citet[Proposition
1]{M64} implies that its conjugate $\intf\phi\nu$ is inf-compact on
$L^{\con\theta}(\nu)$ with respect to the weak topology
$\sigma\paren[\big]{L^{\con\theta}(\nu),\ssubexp(\nu)}$.
\end{proof}

\begin{remark}
	This result generalizes \citet[Corollary 2B]{R71} since 
	$L^\infty(\nu)\subseteq \ssubexp(\nu)$ whenever $\phi'(\infty)=\infty$ (see
	\cref{ex:bounded-implies-ssubexp}).
\end{remark}

\begin{corollary}
	\label{cor:inf-compact-tilde}
	Under the same assumptions and notations as
	\cref{prop:inf-compact-small-orlicz}, the functional $\dif*\phi\nu$ is
$\sigma\paren[\big]{L^{\con\theta}(\nu), \ssubexp(\nu)}$ inf-compact.
\end{corollary}

\begin{proof}
	Observe that since $\phi(x)=\infty$ for $x<0$, we have for every
	$\tau\in\R$ that $\set{\mu\in L^{\con\theta}(\nu)\given
	\dif*\phi\nu(\mu)\leq\tau} =\set{\mu\in\M^1\cap L^{\con\theta}(\nu)\given
\intf\phi\nu(\mu)\leq\tau} =\set{\mu\in L^{\con\theta}(\nu)\given
\intf\phi\nu(\mu)\leq\tau}\cap f^{-1}(1)$ where $f:\mu\to\ex\mu {\ind_\Omega}$
is continuous in the weak topology
$\sigma\paren[\big]{L^{\con\theta}(\nu),\ssubexp(\nu)}$ since
$L^\infty(\nu)\subseteq \ssubexp(\nu)$ by
\cref{lem:linfty-strong-subexponential}.  Hence, $\M^1\cap\B_{\phi,\nu}(\tau)$
is compact as a closed subset of a compact set.
\end{proof}

\begin{corollary}
\label{cor:ball-compact-initial}
	If $\phi'(\infty)=\infty$, then for every $\tau\in\R$ the
	sets $\B_{\phi,\nu}(\tau)$ and $\M^1\cap\B_{\phi,\nu}(\tau)$ are
	compact in the initial topology induced by $\set{\mu\mapsto\ex\mu g\given
	g\in\ssubexp(\nu)}$.
\end{corollary}
\begin{proof}
	Immediate from \cref{prop:inf-compact-small-orlicz,cor:inf-compact-tilde}.
\end{proof}

\begin{corollary}
\label{cor:L-lsc}
	Let $\nu\in\M^1$ be a probability measure and assume that $\phi'(\infty)=\infty$.
	If $g\in L^0(\nu)$ is strongly $(\phi,\nu)$-subexponential and
	$M\subseteq\M^1_c(\nu)$ is a convex set of probability measures containing
	every $\mu\in\M^1_c(\nu)$ with $\di\mu\nu<\infty$, then the function
	$\lb[M]g\nu$ is lower semicontinuous.
\end{corollary}
\begin{proof}
	Follows from \cref{lem:lb-lsc-if-compact} and
	\cref{cor:ball-compact-initial}.
\end{proof}
\begin{remark}
	\label{rem:lsc-ssubexp}
	\cref{cor:L-lsc} does not apply when $\phi'(\infty)<\infty$ or
	$g\in\subexp(\nu)\setminus\ssubexp(\nu)$ (e.g.\ when the pushforward
	measure $\pushforward\nu g$ is gamma-distributed in the case of the KL
	divergence), and it would be interesting to identify conditions other than
	inf-compactness of $\dif\phi\nu$ under which $\lb g\nu$ is lower
	semicontinuous.
\end{remark}

\subsection{Convergence in \titlephi/-divergence and weak convergence}
\label{sec:continuity}

Our goal in this section is to relate two notions of convergence for a sequence
of probability measures $(\nu_n)_{n\in\mathbb{N}}$ and $\nu\in\M^1$: (i)
$\di{\nu_n}\nu\to 0$,\footnote{Throughout this section, we restrict our
attention to $\phi$ which are not the constant $0$ on a neighborhood of $1$,
i.e.\ such that $1\not\in\inter\set{x\in\R\given \phi(x)=0}$, as
otherwise it is easy to construct probability measures $\mu\neq \nu$ such that
$\di\mu\nu = 0$, hence $\di{\nu_n}\nu\to 0$ does not define a meaningful
convergence notion.} and (ii) $\abs{\ex{\nu_n} g - \ex\nu g}\to 0$ for
$g\in\L^0(\Omega)$. Specifically, we would like to identify the largest class
of functions $g\in\L^0(\Omega)$ such that \emph{convergence in
$\phi$-divergence} (i) implies (ii). In other words, we would like to identify
the finest initial topology induced by linear forms $\mu\mapsto \ex\mu g$ for
which (sequential) convergence is implied by (sequential) convergence in
$\phi$-divergence.\footnote{The natural notion of convergence in
$\phi$-divergence defines a topology on the space of probability measures for
which continuity and sequential continuity coincide (see
e.g.~\citet{K60,D64,H07}), so it is without loss of generality that we consider
only sequences rather than nets in the rest of this section. Note that the
information balls $\set{\mu\in\M^1\given \di\mu\nu<\tau}$ for $\tau>0$ need
not be neighborhoods of $\nu$ in this topology, and the information balls do
not in general define a basis of neighborhoods for a topology on the space of
probability measures \citep{C62,C64,C67_Top,D98}.}
This question is less quantitative than computing the best lower bound of the
$\phi$-divergence in terms of the absolute mean deviation, since it only
characterizes when $\abs{\ex{\nu_n} g-\ex\nu g}$ converges to $0$, whereas the
optimal lower bound quantifies the \emph{rate of convergence} to $0$ when it
occurs.

This has been studied in the specific case of the Kullback--Leibler divergence
by Harremoës, who showed \citep[Theorem
25]{H07} that $\kl{\nu_n}\nu\to 0$ implies $\abs{\ex{\nu_n}g-\ex\nu
   g}\to 0$ for every non-negative function $g$ whose moment
   generating function is finite at some positive real (in fact, the converse
   was also shown in the same paper under a so-called \emph{power-dominance}
   condition on $\nu$). In this section, we generalize this to an arbitrary
   $\phi$-divergence and show that convergence in $\phi$-divergence
	 implies $\ex{\nu_n}g\to\ex\nu g$ if and only if $g$ is
   $(\phi,\nu)$-subexponential.

This question is also closely related to the one of understanding the
relationship between weak convergence and \emph{modular convergence} in Orlicz
spaces (e.g.\ \citet{N50} or \citet{M83}). Although convergence in
$\phi$-divergence as defined above only formally coincides with the notion of
modular convergence when $\phi$ is symmetric about $1$ (though this can
sometimes be relaxed \citep{H67}) and satisfies the so-called $\Delta_2$ growth
condition, it is possible that this line of work could be adapted to the
question studied in this section.

We start with the following proposition, showing that this question is
equivalent to the differentiability of $\con{\lb g\nu}$ at $0$.

\begin{proposition}
\label{prop:cont-iff-diff}
	Let $\nu\in\M^1$, $g\in\L^1(\nu)$, and $M\subseteq\M^1$ be a
	convex set of measures integrating $g$ and containing $\nu$.
	Then the following are equivalent:
	\begin{enum}
		\item $\lim_{n\to\infty}\ex{\nu_n}g=\ex\nu g$ for all
			$(\nu_n)_{n\in\mathbb N}\in M^{\mathbb N}$ such that
			$\lim_{n\to\infty}\di{\nu_n}\nu=0$.
		\item $\lb[M]g\nu$ is strictly convex at $0$, that is
			$\lb[M]g\nu(\eps)=0$ if and only if $\eps=0$.
		\item $\partial \con{\lb[M]g\nu}(0)=\set 0$, that is $\con{\lb[M]g\nu}$ is
			differentiable at $0$ and $\conder{\lb[M]g\nu}(0) = 0$.
	\end{enum}
\end{proposition}
\begin{proof}
	$(i)\implies(ii)$\quad Assume for the sake of contradiction that $\lb[M]
	g\nu(\eps)=0$ for some $\eps\neq 0$. Then by definition of $\lb[M]g\nu$,
	there exists a sequence $(\nu_n)_{n\in\mathbb N}\in M^{\mathbb N}$ such
	that for all $n\in\mathbb N$, $\di{\nu_n}\nu \leq 1/n$ and
	$\ex{\nu_n}g-\ex\nu g=\eps$, thus contradicting $(i)$. Hence, $\lb[M]
	g\nu(\eps) = 0$ if and only if $\eps=0$, which is equivalent to strict
	convexity at $0$ since $\lb[M]g\nu$ is convex with global minimum
	$\lb[M]g\nu(0)=0$ by \cref{lem:lb-properties}.

	$(ii)\implies(i)$\quad Let $(\nu_n)_{n\in\mathbb N}\in M^{\mathbb N}$ be
	a sequence such that $\lim_{n\to\infty}\di{\nu_n}\nu=0$.  By definition of
	$\lb[M]g\nu$, we have that $\di{\nu_n}\nu\geq \lb[M]g\nu
	\paren[\big]{\ex{\nu_n}g-\ex\nu g}\geq0$ for all $n\in\mathbb N$, and in
	particular $\lim_{n\to\infty}\lb[M]g\nu\paren[\big]{\ex{\nu_n}g-\ex\nu g}=0$.
Assume for the sake of contradiction that $\ex{\nu_n} g$ does
not converge to $\ex\nu g$. This implies the existence of $\eps> 0$ such that
$\abs{\ex{\nu_n}g-\ex\nu g}\geq \eps$ for infinitely many $n\in\mathbb N$.  But
	then $\lb[M]g\nu\paren[\big]{\ex{\nu_n}g-\ex\nu g}\geq\min\set{\lb[M]
	g\nu(\eps),\lb[M]g\nu(-\eps)}>0$ for infinitely many $n\in\mathbb N$,
a contradiction.

$(ii)\iff(iii)$\quad By a standard characterization of the subdifferential (see
	e.g.\ \citet[Theorem 2.4.2(iii)]{Z02}), we have that $\partial\con{\lb[M]
	g\nu}(0) =\set{x\in\R \given \con{\lb[M]g\nu}(0)+\bicon{\lb[M]g\nu}(x)=0\cdot
	x} =\set{x\in\R\given \bicon{\lb[M]g\nu}(x)=0}$. Since $\lb[M]g\nu$ is convex,
non-negative, and $0$ at $0$, this subdifferential contains $\eps\neq 0$ if and
	only if there exists $\eps\neq 0$ with $\lb[M]g\nu(\eps)= 0$.
\end{proof}

The above proposition characterizes continuity in terms of the
differentiability at $0$ of the conjugate of the optimal lower bound function,
or equivalently by \cref{prop:lb-conjugate}, differentiability of the function
$\cgfxi g\nu\Xi$. In the previous section we investigated in detail the
finiteness (or equivalently by convexity, the continuity) of these functions
around $0$; in this section we show that continuity at $0$ is equivalent to
differentiability at $0$ assuming that $\phi$ is not the constant $0$ on
a neighborhood of $1$.
\begin{proposition}
	\label{prop:cgf-diff-iff-cont}
	Assume that $1\not\in\inter \set{x\in\R\given \phi(x)=0}$. Then for
	every $\sigma$-ideal $\Xi$, probability measure $\nu\in\M^1_c(\Xi)$,
	and $g\in L^0(\Xi)$, we have that $0\in\inter\dom\cgfxi g\nu\Xi$ if
	and only if $\cgfxi g\nu\Xi'(0)=0$.
\end{proposition}
\begin{proof}
	The if direction is immediate, since differentiability at $0$ implies
	continuity at $0$. Thus, for the remainder of the proof we assume
	that $\cgfxi g\nu\Xi$  is finite on a neighborhood of $0$.

	We first consider the case $\phi'(\infty)<\infty$, where
	\cref{lem:lb-finitedagger-bounded} implies $g\in L^\infty(\Xi)$.
	Define $B\eqdef\esssup_\Xi \abs g$, and let $\sigma\in\set{-1,1}$ be such that
	$\phi(1+\sigma x)>0$ for all $x>0$ as exists by assumption on $\phi$. Since
	$\psi$ is non-negative and $0$ at $0$, a standard characterization of the
	subdifferential (e.g.\ \citet[Theorem 2.4.2(iii)]{Z02}) implies that the
	function $t\mapsto \con\psi(\sigma \abs t)$ has derivative $0$ at $0$.
	Then for all $t\in\R$, by considering $\lambda=\sigma tB$ in
	\eqref{eq:cumul}, we obtain $\cgfxi g\nu\Xi(t)$
	is at most $\ex\nu{ \con\psi(tg+\sigma
	tB)}+\delta_{[-\infty,\phi'(\infty)]}(2\sigma\abs tB)\leq \con\psi(2\sigma
	\abs tB)+\delta_{[-\infty,\phi'(\infty)]}(2\sigma\abs tB)$. Now, if
	$\sigma=-1$ then $2\sigma\abs tB\leq 0\leq \phi'(\infty)$ for all $t$, and
	if $\sigma=1$ then necessarily $\phi'(\infty)>0$ and so $2\sigma\abs tB\leq
	\phi'(\infty)$ for sufficiently small $\abs t$.  Thus, we have for
	sufficiently small $\abs t$ that $\cgfxi g\nu\Xi(t)$
	is between $0$ and $\con\psi(2\sigma\abs tB)$, both of which are $0$ with
	derivative $0$ at $0$, completing the proof in this case.

	Now, assume that $\phi'(\infty)=\infty$, so that we have $\cgfxi g\nu\Xi
	=\cgf g\nu=\inf_{\lambda\in\R} f(\cdot ,\lambda)$ for $f(t,\lambda) \eqdef
	\ex\nu{\con\psi(tg+\lambda)}$. Note that $\psi\geq 0$ implies $f\geq 0$, so
	since $\cgf g\nu(0)=f(0,0)=0$ we have by standard results in convex analysis
	(e.g.\ \citet[Theorem 2.6.1(ii)]{Z02}) that $\partial \cgf g\nu(0)=\set{\con
	t\given (\con t, 0)\in \partial f(0,0)}$. Furthermore, by assumption $\cgf
	g\nu$ is finite on a neighborhood of $0$, so since $\cgf g\nu=\cgf{g+c}\nu$
	for all $c\in\R$, \cref{lem:shrink} implies $\inter\paren*{\dom\cgf
	g\nu}\times \R\subseteq\dom f$ and in particular $(0,0) \in\inter\dom f$.
	Thus, defining for each $\omega\in\Omega$ the function
	$f_\omega(t,\lambda)\eqdef \con\psi(t\cdot g(\omega)+\lambda)$ (where here
	and in the rest of the proof we fix some representative $g\in\L^0(\Omega)$),
	standard results on convex integral functionals (e.g.\ \citet[Theorem 1]{L68} or
	\citet[Formula (7)]{IT69}) imply that $(\con t,\con\lambda)\in \partial
	f(0,0)$ if and only $(\con t, \con\lambda)=\paren[\big]{\ex\nu{\con
	t_\omega}, \ex\nu{\con \lambda_\omega}}$ for measurable functions $\con
	t_\omega, \con \lambda_\omega:\Omega\to\R$ such that $(\con t_\omega,
	\con\lambda_\omega)\in \partial f_\omega(0,0)$ holds $\nu$-a.s.

	Now, for each $\omega\in\Omega$, we have that $(\con
	t_\omega,\con\lambda_\omega)\in \partial f_\omega(0,0)$ if and only if
	$\con\psi(t\cdot g(\omega)+\lambda)\geq \con t_\omega\cdot
	t+\con\lambda_\omega \cdot \lambda$ for all $(t,\lambda)\in\R^2$. By
	considering $t=0$, this implies that
	$\con\lambda_\omega\in\partial\con\psi(0)=\set{x\in\R\given \psi(x)=0}$,
	which is contained in either $\R_{\geq 0}$ or $\R_{\leq 0}$ since $\psi$ is
	not $0$ on a neighborhood of $0$. Then since the integral of a function of
constant sign is zero if and only if it is zero almost surely, we have that
$(\con t,0)=\paren[\big]{\ex\nu{\con t_\omega},\ex\nu{\con\lambda_\omega}}$ if
and only if $\con\lambda_\omega=0$ holds $\nu$-a.s. But $(\con t_\omega,0)
\in\partial f_\omega(0,0)$ if and only if for all $t\in\R$ we have $\con
t_\omega\cdot t\leq\inf_\lambda \con\psi(t\cdot g(\omega)+\lambda)
=\con\psi(0)=0$, i.e.~if and only if $\con t_\omega=0$.

	Putting this together, we get that $\partial \cgf g\nu(0)=\set{\con t\given
	(\con t, 0)\in \partial f(0,0)} = \set{\ex\nu{\con t_\omega} \given (\con
	t_\omega, 0)\in \partial f_\omega(0,0)\text{ $\nu$-a.s.}}=\set{\ex\nu{\con
	t_\omega} \given \con t_\omega=0\text{ $\nu$-a.s.}}=\set 0$ and $\cgf
	g\nu'(0)=0$ as desired.
\end{proof}
\begin{remark}
	If $\phi$ is $0$ on a neighborhood of $1$, then it is easy to show that
	$\cgf g\nu$ is not differentiable at $0$ unless $g$ is $\nu$-essentially
	constant. Thus, the above proposition shows that the following are
	equivalent: (i) $1\not\in\inter\set{x\in\R\given\phi(x)=0}$, (ii) for
	every $g$, continuity of $\cgf g\nu$ at $0$ implies
	differentiability at $0$, (iii) $\di\mu\nu=0$ for probability measures $\mu$
	and $\nu$ if and only if $\mu=\nu$.

	A similar (but simpler) proof shows that the following are equivalent: (i)
	$\phi$ strictly convex at $1$, (ii) for every $g$, continuity of $t\mapsto
	\intf{\con\psi\!}\nu(tg)$ at $0$ implies differentiability at $0$, and (iii)
	$\di\mu\nu=0$ for finite measures $\mu$ and $\nu$ if and only if $\mu=\nu$.
	The similarity of the statements in both cases suggest there may be a common
	proof of the equivalences using more general techniques in convex analysis.
\end{remark}

Thus, combining the previous two propositions and \cref{prop:lb-conjugate}
computing the convex conjugate of the optimal lower bound function, we obtain
the following theorem.

\begin{theorem}
\label{thm:continuity-without-ac}
Assume that $1\notin\inter(\set{x\in\R\given
\phi(x)=0})$. Then for a $\sigma$-ideal $\Sigma$, $g\in L^0(\Sigma)$ and
$\nu\in X^1_g$, writing
$\Xi = \set{A\in\A\given \forall\mu\in X_g, \abs\mu(A)=0}$,
the following are equivalent:
	\begin{enum}
	\item for all $(\nu_n)_{n\in\mathbb N}\in \M^1_c(\Xi)^{\mathbb N}$,
		$\lim_{n\to\infty}\di{\nu_n}\nu= 0$ implies
		that $g$ is $\nu_n$-integrable for sufficiently large
		$n$ and $\lim_{n\to\infty}\ex{\nu_n}g=\ex\nu g$.
	\item for all $(\nu_n)_{n\in\mathbb N}\in (X^1_g)^{\mathbb N}$,
		$\lim_{n\to\infty}\di{\nu_n}\nu= 0$ implies
		$\lim_{n\to\infty}\ex{\nu_n}g=\ex\nu g$.
	\item $\partial \cgfxi g\nu\Xi(0)=\set 0$, i.e.\ $\cgfxi g\nu\Xi$ is differentiable
			at $0$ with $\cgfxi g\nu\Xi'(0)=0$.
		\item $0\in\inter(\dom \cgfxi g\nu\Xi)$, that is, $g\in L^\infty(\Xi)$
			when $\phi'(\infty)<\infty$ and $g\in\subexp(\nu)$ 
			when $\phi'(\infty)=\infty$.
	\end{enum}
\end{theorem}
\begin{proof}
	The equivalence of (ii)-(iv) is immediate from
	\cref{prop:cont-iff-diff,prop:cgf-diff-iff-cont} since \cref{prop:lb-conjugate}
	implies $\con{\lb[X^1_g]g\nu}=\cgfxi g\nu\Xi$.  That (i) implies (ii) is
	immediate by definition of $X^1_g$. The reformulation of
	$0\in\inter(\dom\cgfxi g\nu\Xi)$ depending on the finiteness of
	$\phi'(\infty)$ uses \cref{lem:lb-finitedagger-bounded} and
	\cref{defn:subexponential}. Finally that (ii) and (iv) implies (i) is
	immediate when $\phi'(\infty)<\infty$—
	since every $\mu\in\M^1_c(\Xi)$ integrates every $g\in L^\infty(\Xi)$—and
	follows from \cref{prop:subexp-pairing} otherwise.
\end{proof}
 \section{Optimal bounds relating \titlephi/-divergences and IPMs}
\label{sec:ipm}

In this section we generalize \cref{thm:equiv} on the optimal lower bound
function for a single measure and function to the case of sets of measures and
measurable functions.

\subsection{On the choice of definitions}
\label{sec:optimal-lb-def}

When considering a class of functions $\G$, there are several ways to define
a lower bound of the divergence in terms of the mean deviation of functions in
$\G$. The first one is to consider the IPM $\ipmsymb\G$ induced by $\G$ and to ask
for a function $L$ such that $\di\mu\nu\geq L\paren[\big]{\ipm\mu\nu}$ for all
probability measures $\mu$ and $\nu$, leading to the following definition of
the optimal bound.

\begin{definition}
\label{defn:lb-function}
	Let $\G\subseteq\L^0(\Omega)$ be a non-empty set of measurable functions
	and let $N,M\subseteq\M^1$ be two sets of probability measures such that
	$\G\subseteq L^1(\nu)$ for every $\nu\in N\cup M$. The optimal lower bound
	function $\lb[M]\G N:\R_{\geq 0}\to\eR_{\geq 0}$ is defined by
	\begin{displaymath}
		\lb[M]\G N\paren*{\eps}
		\eqdef \inf\set[\Big]{\di\mu\nu \given (\nu,\mu)\in N\times M\land
		\sup_{g\in\G}\paren[\big]{\ex\mu g-\ex\nu g}=\eps}\,.
	\end{displaymath}
  We also for convenience extend the definition to the negative reals by
	\begin{displaymath}
		\lb[M]\G N\paren*{\eps}\eqdef \lb[M]{-\G} N\paren*{-\eps}
		= \inf\set[\Big]{\di\mu\nu \given (\nu,\mu)\in N\times M\land
      \inf_{g\in\G}\paren[\big]{\ex\mu g-\ex\nu g}=\eps}
	\end{displaymath}
  for $\eps<0$ where $-\G\eqdef\set{-g\given g\in\G}$.
\end{definition}

\begin{remark}
  To motivate the definition of $\lb[M]\G N$ on the negative reals, note that
  the equality $\sup_{g\in \G}\ex\mu g-\ex\nu g=\eps$ for $\eps\geq 0$ constrains by
  how ``much above $0$'' an element of $\G$ can distinguish $\mu$ and $\nu$,
  whereas the constraint $\inf_{g\in\G}\ex\mu g-\ex\nu g=-\eps$ analogously constrains
  how much \emph{below} $0$ an element of $\G$ can distinguish them.
	When $\G$ is closed under negation, then $\sup_{g\in \G}
  \ex\mu g-\ex\nu g=\ipm[\G]\mu\nu=-\inf_{g\in\G}\ex\mu g-\ex\nu g$
  and $\lb[M]\G N$ is even and exactly quantifies the smallest value taken by the $\phi$-divergence
  given a constraint on the IPM defined by $\G$. 
\end{remark}

An alternative definition, using the notations of \cref{defn:lb-function}, is
to consider the largest function $L$ such that $\di\mu\nu\geq L(\ex\mu g-\ex\nu
g)$ for all $(\nu,\mu)\in N\times M$ and $g\in\G$. It is easy to see that this
function can simply be expressed as
\begin{displaymath}
	\inf_{g\in\G} \lb[M]gN\paren*{\eps}
	=\inf_{\substack{g\in\G\\\nu\in N}} \lb[M] g\nu(\eps)
	= \inf\set[\Big]{\di\mu\nu\given (\nu,\mu,g)\in N\times M\times\G\land
		\ex\mu g-\ex\nu g=\eps}\,.
\end{displaymath}
Observe that $\inf_{g\in\G} \lb[M]gN = \lb[M]\G N$
when $\G=\set{g}$ or $\G=\set{-g, g}$. More generally, the goal of this section
is to explore the relationship between $\lb[M] \G N$ and $\inf_{g\in\G}
\lb[M]gN$. In particular, we will show that assuming a basic convexity condition
on the set of measures $M$, these functions can differ only on their (at most
countably many) discontinuity points.

\begin{lemma}
\label{lem:ipm-lb-increasing}
Let $N,M\subseteq \M^1$ be two sets of probability measures with $N\subseteq M$
and $M$ convex. Then the functions $\lb[M]\G N$ and $\inf_{g\in\G} \lb[M] gN$
are non-negative, $0$ at $0$, and are non-decreasing on the non-negative reals. 
\end{lemma}

\begin{proof}
	It is sufficient to prove the result for $\lb[M]\G N$, since the result for
	$\inf_{g\in \G}\lb[M] gN$ follows from the fact that taking infima preserves
	sign and monotonicity. Non-negativity and being $0$ at $0$ follow from
	non-negativity of $\di\mu\nu$ with $\di\nu\nu=0$.

	Fix $0\leq x\leq y$ and consider $\alpha>\lb[M]\G N(y)$, so that by definition
	there exist $\mu\in M$ and $\nu\in N$ with $\di\mu\nu<\alpha$ and
	$\sup_{g\in\G}\paren[\big]{\ex\mu g-\ex\nu g}=y$. Define $\mu'=x/y\cdot \mu
	+ (1-x/y)\cdot\nu$, which is a probability measure in $M$ since $\nu\in
	N\subseteq M$ and $M$ is convex. Then we have for every $g\in\G$ that
	$\ex{\mu'}g-\ex\nu g =x/y\cdot \paren[\big]{\ex\mu g-\ex\nu g}$, and thus
	$\sup_{g\in\G}\paren[\big]{\ex{\mu'}g-\ex\nu g}=x$. Furthermore, by convexity
	of $\dif\phi\nu$ we have $\di{\mu'}\nu\leq x/y\cdot \di\mu\nu
	+ (1-x/y)\cdot\di\nu\nu <x/y\cdot \alpha\leq \alpha$ since $x/y \leq 1$. This
	implies that $\lb[M]\G N(x)<\alpha$ and since $\alpha$ can be made arbitrarily
	close to $\lb[M]\G N(y)$ that $\lb[M]\G N(x)\leq \lb[M]\G N(y)$.
\end{proof}

\begin{remark}
	\label{rem:lb-convex}
	For convex sets of measures $M$ and $N$ and a single function $g\in
	L^1(\nu)$, a simple adaptation of \cref{lem:lb-properties} shows that
	$\lb[M]gN$ is convex, non-decreasing, and non-negative on the non-negative
	reals. \Cref{lem:ipm-lb-increasing} extends the latter two properties to the
	case of $\lb[M]\G N$ for a set of functions $\G$, and in fact its proof
	shows that $\lb[M]\G N(y)/y$ is non-decreasing, which is necessary for
	convexity. It would be interesting to characterize the set of $\G$, $N$,
	and $M$ for which $\lb[M]\G N$ and $\inf_{g\in\G}\lb[M]gN$ are in fact
	convex.
\end{remark}

\begin{proposition}
	\label{prop:optimal-lb-equivalence}
	Under the assumptions of \cref{lem:ipm-lb-increasing}, we have for
	every $\eps>0$ that
	\begin{displaymath}
		\lim_{\eps'\to\eps^-} \inf_{g\in\G} \lb[M]gN(\eps')
		\leq
		\lb[M]\G N(\eps)
		\leq
		\inf_{g\in\G}\lb[M]gN(\eps)
		\,,
	\end{displaymath}
	with equality if $\inf_{g\in\G}\lb[M]gN$ is lower semicontinuous
	(equivalently left-continuous) at $\eps$ or if $\G$ is compact in the
	initial topology on $\L^0$ induced by the maps $\ip{\mu-\nu}\cdotarg$ for
	$\mu\in M$ and $\nu\in N$.
\end{proposition}

\begingroup
\allowdisplaybreaks
\begin{proof}
	Under the assumptions of \cref{lem:ipm-lb-increasing} we have
	$\inf_{g\in\G}\lb[M]gN$ and $\lb[M]\G N$ are non-decreasing on
	the positive reals. Thus, we have
	\begin{align}
		\inf_{g\in\G}\lb[M]gN(\eps) &= \inf\set[\Big]{\di\mu\nu\given (\nu,\mu)\in N\times M\land
		\exists g\in\G, \ex\mu g-\ex\nu g=\eps}
		\nonumber\\
		&\geq
		\inf\set[\Big]{\di\mu\nu\given (\nu,\mu)\in N\times M\land
		\sup_{g\in\G}\paren[\big]{\ex\mu g-\ex\nu g}\geq\eps}
		\label{eqn:inf-geq-class}\\
		&=\inf_{\eps'\geq\eps} \lb[M]\G N(\eps') = \lb[M]\G N(\eps)
		\label{eqn:inf-nondec} \\
		&=
		\inf\set[\Big]{\di\mu\nu \given (\nu,\mu)\in N\times M\land
		\forall\eps'<\eps\;\exists g\in\G,\; \ex\mu g-\ex\nu g\geq\eps'}
		\nonumber\\
		&\geq\sup_{\eps'<\eps}\inf\set[\Big]{\di\mu\nu \given (\nu,\mu)\in N\times M\land
		\exists g\in\G, \ex\mu g-\ex\nu g\geq\eps'}
		\nonumber\\
		&=\sup_{\eps'<\eps}\inf\set[\Big]{\di\mu\nu \given (\nu,\mu,g)\in N\times M\times\G\land
		\ex\mu g-\ex\nu g\geq\eps'}\nonumber\\
		&=\sup_{\eps'<\eps}\inf_{g\in\G}\lb[M]gN(\eps')=
		\lim_{\eps'\to\eps^-} \inf_{g\in\G} \lb[M]gN(\eps')
		\label{eqn:inf-lb-nondecreasing}
	\end{align}
	where \cref{eqn:inf-geq-class} is since if there is $g\in\G$ with $\ex\mu
	g-\ex\nu g=\eps$ then $\sup_{g\in\G}\ex\mu g-\ex\nu g\geq\eps$,
	\cref{eqn:inf-nondec} is because $\lb[M]\G N$ is non-decreasing, and
	\cref{eqn:inf-lb-nondecreasing} is because $\inf_{g\in G}\lb[M]gN$ is
	non-decreasing.

		For the equality claims, since $\inf_{g\in\G} \lb[M]gN$ is
		non-decreasing, lower semicontinuity at $\eps$ is equivalent to
		left-continuity, and $\lim_{\eps'\to\eps^-}\inf_{g\in\G}\lb[M]gN(\eps')
		= \inf_{g\in G}\lb[M]gN(\eps)$ in this case. If $\G$ is compact in the
claimed topology, then $\sup_{g\in\G}\paren[\big]{\ex\mu g-\ex\nu g}$ is the
supremum of the continuous function $\ip{\mu-\nu}\cdotarg$ on the compact set
$\G$, so that $\sup_{g\in\G}\paren[\big]{\ex\mu g-\ex\nu g}
=\max_{g\in\G}\paren[\big]{\ex\mu g-\ex\nu g}$ and thus
\cref{eqn:inf-geq-class} is an equality.
\end{proof}
\endgroup
\begin{corollary}
	\label{cor:optimal-lb-equivalence}
  Under the assumptions of \cref{lem:ipm-lb-increasing} we have that
	$\inf_{g\in\G}\lb[M]gN$ and $\lb[M]\G N$ are non-increasing on the
  non-positive reals, non-decreasing on the non-negative reals, $0$ at
  $0$, and differ only on their (at most countably many) discontinuity points,
  at which $\inf_{g\in\G}\lb[M]gN\geq\lb[M]\G N$. In particular, they
  have the same convex conjugate and biconjugate.
\end{corollary}
\begin{proof}
  Applying \cref{prop:optimal-lb-equivalence} to $\lb[M]{-\G}N(-\eps)$ for
  $\eps < 0$ gives the claim for the negative reals. Since the functions share the same lsc regularization (the
  largest lsc function lower bounding them pointwise), they also share their
  convex conjugate and biconjugate.
\end{proof}
\begin{remark}
  \Cref{cor:optimal-lb-equivalence} is key because, as we will see in
  \cref{sec:ipm-lb-comp}, it lets us reduce the problem of computing
  the optimal lower bound on an IPM to the case of a single function $g$
  considered in \cref{sec:singleg}.
\end{remark}
\begin{remark}
\Cref{cor:optimal-lb-equivalence} also implies that $\inf_{g\in
\G}\lb[M] gN$ and $\lb[M]\G N$ have the same \emph{(generalized) inverse}. This
inverse consists simply of the best bounds on the mean deviation, that is
the largest non-positive function $V$ and smallest non-negative function $U$ such that $V\paren[\big]{\di\mu\nu}
\leq\ex\mu g-\ex\nu g\leq U\paren[\big]{\di\mu\nu}$ for all $(\mu,\nu,g) \in M\times N\times\G$, or
equivalently such that $\ipm\mu\nu\leq U\paren[\big]{\di\mu\nu}$ for all
$(\mu,\nu)\in M\times N$ when $\G$ is closed under negation. In this language, any discontinuity of
$\inf_{g\in\G}\lb[M]gN$ corresponds to an interval on which $U$ is constant,
i.e.~in which changing the value of the divergence does not change the largest
possible value of $\ipm\mu\nu$.
\end{remark}

We conclude this section with two lemmas showing how the lower bound is
preserved under natural transformations of the sets of functions $\G$ or
measures $M,N$.

\begin{lemma}
\label{lem:ipm-preserved-closed-convex}
	For every set $\G\subseteq\L^0(\Omega)$ and pair of measures
	$\mu,\nu\in X_\G$, we have that
	\[
	\sup_{g\in\G}\paren[\big]{\ex\mu g-\ex\nu g}
		=\sup_{g\in\overline{\convhull\G}}\paren[\big]{\ex\mu g-\ex\nu g}
	\]
	where $\overline{\convhull\G}$ is the $\sigma(Y_\G,X_\G)$-closed
	convex hull of $\G$.
\end{lemma}
\begin{proof}
	We have $\G\subseteq\overline{\convhull\G}$, and furthermore since
	$\ip{\mu-\nu}\cdotarg$ is a $\sigma(Y_\G,X_\G)$-continuous linear
	function we have that the set $\set[\big]{h\in Y_\G\given
	\ip{\mu-\nu}h \leq \sup_{g\in\G}\paren*{\ex\mu g-\ex\nu g}}$ is convex,
	$\sigma(Y_\G,X_\G)$-closed, and contains $\G$, and so also
	contains $\overline{\convhull\G}$.
\end{proof}

\begin{lemma}
	\label{lem:lb-pushforward}
	For every $g\in\L^0(\Omega)$, we have $\lbsing[X^1_g]g{X^1_g}
	=\lb[\pushforward{X^1_g}g]{\ident_{\R}}{\pushforward{X^1_g}g}$ where
	$\pushforward{X^1_g}g=\set[\big]{\pushforward\nu g\given\nu\in X^1_g}$ is
	the set of probability measures on $\R$ obtained by pushing forward through
	$g$ the probability measures $\nu\in X^1_g$.
	Furthermore, for every $\nu\in\M^1$ and $g\in L^1(\nu)$ we have that
	$\lbsing[X^1_g] g\nu=\lb[\pushforward{X^1_g}g]{\ident_{\R}}{\pushforward\nu g}$.
\end{lemma}
\begin{proof}
	We first prove the main claim.
	As in \cref{ex:ipm-rv}, we have for every $\mu,\nu\in X^1_g$ that
	$\ex\mu g-\ex\nu g
		=\ex*{\pushforward\mu g}{\ident_{\R}}
			-\ex*{\pushforward\nu g}{\ident_{\R}}$,
	so it suffices to show for every $\mu_0,\nu_0\in X^1_g$ the existence of
	$\mu,\nu\in X^1_g$ with $\pushforward\mu g=\pushforward{\mu_0}g$,
	$\pushforward\nu g=\pushforward{\nu_0}g$, and
	$\di{\pushforward{\mu_0}g}{\pushforward{\nu_0}g}=\di\mu\nu
	\leq\di{\mu_0}{\nu_0}$.

	For this, write $\xi=\frac12(\mu_0+\nu_0)$ so that $\mu_0,\nu_0\ll\xi$ and
	$\xi\in X^1_g$, and define the measures $\mu,\nu\in \M^1_c(\xi)$ by
	$\frac{d\mu}{d\xi}=\frac{d\pushforward{\mu_0}g}{d\pushforward\xi g}\circ g$
	and
	$\frac{d\nu}{d\xi}=\frac{d\pushforward{\nu_0}g}{d\pushforward\xi g}\circ g$
	(note that these are just the conditional expectations of
	$\frac{d\mu_0}{d\xi}$ and $\frac{d\nu_0}{d\xi}$ with respect to $g$).
	It remains to show that $\mu$ and $\nu$ have the desired properties,
	for which we first note that for every (Borel) measurable function
	$h:\R^3\to\R\cup\set{+\infty}$ we have
	\begin{align*}
		\ex*\xi{h\paren*{\frac{d\mu}{d\xi},\frac{d\nu}{d\xi},g}}
		&=
		\ex*\xi{h\paren*{\frac{d\pushforward{\mu_0}g}{d\pushforward\xi g}\circ g,\frac{d\pushforward{\nu_0}g}{d\pushforward\xi g}\circ g,g}}
		\\
		&=
		\ex*{\pushforward\xi g}{h\paren*{\frac{d\pushforward{\mu_0}g}{d\pushforward\xi g},\frac{d\pushforward{\nu_0}g}{d\pushforward\xi g},\ident_{\R}}}\,.
	\end{align*}
	Then taking $h(x,y,z)=x$ we get $\ex\mu\Omega=\ex\mu{\ind_\Omega}
		=\ex{\pushforward{\mu_0}g}{\ind_{\R}}=\ex{\mu_0}{\ind_\Omega}=1$,
		and similarly by taking $h(x,y,z)=y$ we get $\ex\nu\Omega=1$.
	Taking $h(x,y,z)=x\cdot \abs{z}$ we get $\ex\mu{\abs g}
		=\ex{\mu_0}{\abs g}<\infty$ so that $\mu\in X^1_g$,
		and similarly by taking $h(x,y,z)=y\cdot\abs z$ we get
		$\ex\nu{\abs g}=\ex{\nu_0}{\abs g}<\infty$ and $\nu\in X^1_g$.
	Finally, as in \cref{rmrk:divergence-symmetric-defn}, taking
	$h(x,y,z)=y\cdot \phi(x/y)$ if $y\neq 0$ and $h(x,y,z)=x\cdot \phi'(\infty)$
	if $y=0$ gives $\di\mu\nu=\di{\pushforward{\mu_0}g}{\pushforward{\nu_0}g}$,
	and furthermore Jensen's inequality implies that
	$\di\mu\nu\leq\di{\mu_0}{\nu_0}$ since $h$ is convex.

	The furthermore claim is analogous after noting that since when
	$\mu\ll\nu$ and $g\in L^1(\nu)$ we can take $\xi=\nu_0=\nu$.
\end{proof}

\subsection{Derivation of the bound}
\label{sec:ipm-lb-comp}

In this section we give our main results computing optimal lower bounds on
a $\phi$-divergence given an integral probability metric. Note that from
\cref{sec:optimal-lb-def}, the optimal lower bound is simply the infimum of the
optimal lower bound $\lb g\nu$ for each $g\in\G$ and $\nu\in N$. Since
$\con{\lb g\nu}=\cgf g\nu$ by \cref{prop:lb-conjugate}, and given the
order-reversing property of convex conjugacy, it is natural to consider the
best \emph{upper bound} on $\cgf g\nu$ which holds \emph{uniformly} over all
$g\in \G$ and $\nu\in N$. Formally, we have the following definition.

\begin{definition}
	\label{defn:cgfset}
	Let $\Xi$ be a $\sigma$-ideal, $\G\subseteq L^0(\Xi)$ be a set of
	measurable functions, and $N\subseteq \M^1_c(\Xi)$ be a set of measures. We
	write $\cgfxi\G N\Xi(t)\eqdef \sup\set{\cgfxi g\nu\Xi(t)\given
		(g,\nu)\in\G\times N}$ and $\cgf\G N(t)\eqdef \sup\set{\cgf 
		g\nu(t)\given (g,\nu)\in\G\times N}$.
\end{definition}

Note that $\cgfxi\G N\Xi$ is convex and lower semicontinuous as a supremum of
convex and lower semicontinuous functions. Furthermore, as alluded to before
\cref{defn:cgfset}, we expect $\cgfxi \G N\Xi$ to be equal to the conjugate of
the optimal lower bound functions. This is stated formally in the following
\lcnamecref{thm:ipm-lb-arbitrary-N} which also gives a sufficient condition
under which the optimal lower bound functions are convex and lower
semicontinuous (see also \cref{rem:optimal-lb-lsc} below).

\begin{theorem}
	\label{thm:ipm-lb-arbitrary-N}
	Let $(X,Y)$ be a dual pair with $X\subseteq \M$ and $Y\subseteq L^0(\Xi)$
	where $\Xi= \set{A\in\A\given \forall \mu\in X, \abs\mu(A) = 0}$.  Consider
	$\G\subseteq Y$ and $N\subseteq X^1\eqdef X\cap\M^1$ and assume that
	$(X,Y)$ is decomposable with respect to all measures in $N$. Then we have
	\begin{equation}
		\label{eq:lb-conjugate}
		\con{\lb[\,X^1] \G N}
	=\con{\paren[\bigg]{\inf_{g\in\G}\lb[\,X^1] gN}}
	= \cgfxi\G N\Xi\,.
	\end{equation}
\end{theorem}

\begin{proof}
	The first equality in \eqref{eq:lb-conjugate} follows from
  \cref{cor:optimal-lb-equivalence}. For the second
	equality, we have
	\begin{displaymath}
	\con{\paren[\bigg]{\inf_{g\in\G}\lb[\,X^1] gN}}
=\con{\paren[\Bigg]{\inf_{\substack{g\in\G\\\nu\in N}}\lb[\,X^1]g\nu}}
=\sup_{\substack{g\in \G\\\nu\in N}}\con{\lb[\,X^1]g\nu}
=\sup_{\substack{g\in \G\\\nu\in N}}\cgfxi g\nu\Xi
=\cgfxi \G N\Xi\,,
	\end{displaymath}
	where we used successively the definition of $\lb[\,X^1]gN$, 
	the fact that $\con{\paren{\inf_{\alpha\in A} f_\alpha}}
	= \sup_{\alpha\in A} \con f_\alpha$ for any collection
	$(f_\alpha)_{\alpha\in A}$ of functions, \cref{prop:lb-conjugate}, and
	\cref{defn:cgfset}.
\end{proof}

\begin{remark}
	When starting from a set of function $\G\subseteq L^0(\Xi)$ for some
	$\sigma$-ideal $\Xi$, a natural pair to which \cref{thm:ipm-lb-arbitrary-N}
	can be applied is the pair $(X_\G, Y_\G)$ from \cref{defn:G-dual-pair}.
\end{remark}

\begin{remark}
\label{rem:optimal-lb-lsc}
	\Cref{thm:ipm-lb-arbitrary-N} computes the conjugate
	of the optimal lower bound functions, but if this function is not convex or lsc
	and so does not coincide with its biconjugate,
	it is also useful to discuss what we can be said about
	$\lb[\,X^1]\G N$ or $\inf_{g\in\G}\lb[\,X^1] gN$ themselves.

	First note that for all $g\in\G$ and $\nu\in N$, $\lb[\,X^1] g\nu$ is
	convex and non-decreasing over the non-negative reals by
	\cref{lem:lb-properties} and under the assumptions of
	\cref{thm:ipm-lb-arbitrary-N}, $\bicon{\lb[\,X^1] g\nu} = \con{\cgfxi
	g\nu\Xi}$ by \cref{prop:lb-conjugate}. Thus, we can apply
	\cref{lem:bicon-inf} and obtain for all $\eps$ that
	\begin{displaymath}
		\liminf_{\eps'\to\eps}\,\inf_{g\in\G} \lb[\,X^1]gN(\eps')
		\leq
	\inf_{\substack{g\in \G\\\nu\in N}}\con{\cgfxi g\nu\Xi}(\eps)
	\leq \inf_{g\in\G}\lb[\,X^1]gN(\eps)\,.
\end{displaymath}
Thus the function $\inf_{(g,\nu)\in \G\times N}\con{\cgfxi g\nu\Xi}$ allows us
to recover the function $\inf_{g\in\G}\lb[\,X^1] gN$ up to its points of
discontinuity which are countable by monotonicity. Similarly, by
\cref{cor:optimal-lb-equivalence} we also recover $\lb[\,X^1]\G N$ up to
its countably many points of discontinuity.

	More can be said under additional assumptions. If $\lb[\,X^1] g\nu$ is
	lower semicontinuous for each $g\in\G$ and $\nu\in N$ (e.g. when
	$\phi'(\infty)=\infty$, $X\subseteq \M_c(\nu)$, and
	$\G\subseteq\ssubexp(\nu)$ for all $\nu\in N$ by \cref{cor:L-lsc}), then
	\begin{displaymath}
		\inf_{g\in\G}\lb[\,X^1] gN(\eps)
		=\inf_{\substack{g\in \G\\\nu\in N}}\lb[\,X^1] g\nu(\eps)
		=\inf_{\substack{g\in \G\\\nu\in N}}\con{\cgfxi g\nu\Xi}(\eps)\,,
	\end{displaymath}
	Furthermore, if we also know that the function $\inf_{(g,\nu)\in\G\times
	N}\con{\cgfxi g\nu\Xi}$ is itself convex and lsc, then
	\begin{displaymath}
		\inf_{g\in\G}\lb[\,X^1]gN(\eps) = \lb[\,X^1]\G N(\eps)
		= \con{\cgfxi\G N\Xi}(\eps)\,.
	\end{displaymath}
\end{remark}

Similarly to \cref{cor:equiv-operational}, we give in the following
\lcnamecref{cor:optimal-lb-operational} an ``operational'' restatement of
\cref{thm:ipm-lb-arbitrary-N} emphasizing the duality between upper bounds on
$\cgf\G N$ and lower bounds on $\di\mu\nu$ in terms of $\ipm\mu\nu$.

\begin{corollary}
	\label{cor:optimal-lb-operational}
	Under the assumptions of \cref{thm:ipm-lb-arbitrary-N}, for every convex
	and lower semicontinuous function $L:\R_{\geq
	0}\to\eR$, the following are equivalent:
	\begin{enum}
		\item $\di\mu\nu\geq L(\ipm\mu\nu)$ for all $\nu\in N$ and
			$\mu\in X^1$.
		\item $\di\mu\nu\geq L(\abs{\ex\mu g -\ex\nu g})$ for all $g\in\G$,
			$\nu\in N$, and $\mu\in X^1$.
		\item $\cgfxi g\nu\Xi(t)\leq \con L(\abs t)$ for all $t\in\R$,
			$g\in\G$, and $\nu\in N$.
	\end{enum}
\end{corollary}

\begin{proof}
  The equivalence of (i) and (iii) follows from applying
  \cref{thm:ipm-lb-arbitrary-N} to $\G'=\G\cup-\G$, since
  $\ipm[\G]\mu\nu=\sup_{g\in\G'}\ex\mu g-\ex\nu g\geq 0$,
  $\lb[X^1]{\G'}N$ is even, and
  $\cgfxi{\G'}N\Xi(t)=\max\set{\cgfxi\G N\Xi(t),\cgfxi\G N\Xi(-t)}$.
  The equivalence of (i) and (iii) for $\set{g,-g}$ for each $g\in \G$ gives
  the equivalence of (ii) and (iii).
\end{proof}

\begin{example}[Subgaussian functions]
	\label{exmp:phi-subgaussian}
	For the Kullback--Leibler divergence, \citet[Lemma 4.18]{BLM13} shows that
	$\kl\mu\nu \geq \frac12\ipm\mu\nu^2$ for all $\mu\in\M^1$ if and only if
	$\log\ex*\nu{e^{t\paren{g-\ex\nu g}}}\leq t^2/2$ for all $g\in\G$ and
	$t\in\R$. Such a quadratic upper bound on the log moment-generating
	function is one of the characterizations of the so-called
	\emph{subgaussian} functions, which contain as a special case the class of
	bounded functions by Hoeffding's lemma \citep{H63} (see also
	\cref{exmp:hoeffding}). \Cref{cor:optimal-lb-operational} recovers this
	result by considering the (self-conjugate) function $L:t\mapsto t^2/2$, thus
	showing that Pinsker's inequality generalize to all subgaussian functions.

	\Cref{thm:ipm-lb-arbitrary-N} generalizes this further to an arbitrary
	$\phi$-divergence, showing that a subset $\G\subseteq\L^0(\Omega)$ of
	measurable functions satisfies $\di\mu\nu \geq\frac12\ipm\mu\nu^2$ for all
	$\mu\in\M^1$ if and only if $\cgf g\nu(t)\leq t^2/2$ for all $g\in\G$ and
	$t\in\R$. By analogy, we refer to functions whose cumulant generating
	function admits such a quadratic upper bound as
	\emph{$\phi$-subgaussian} functions.
\end{example}

\begin{example}
	\label{exmp:subgaussian-chisquared-experimental}
	Recall from \cref{exmp:hcr} that the $\chi^2$-divergence given by
	$\phi(x)=(x-1)^2 +\delta_{\R_\geq 0}(x)$ satisfies
	\[\con\psi(x)=\begin{cases}x^2/4&x\geq -2\\-1-x&x<-2\end{cases}\]
	and $\cgf g\nu(t)\leq\inf_\lambda \ex*\nu{(tg+\lambda)^2/4}
	=t^2\operatorname{Var}_\nu(g)/4$, showing that the class of
	$\chi^2$-subgaussian functions (see \cref{exmp:phi-subgaussian}) includes
	all those with bounded variance.
\end{example}

\begin{example}
	\label{exmp:bolley-villani}
	As a step towards understanding the Wasserstein distance,
	\citea{Bolley and Villani}{BV05} define a ``weighted total variation distance''
	between probability measures $\mu$ and $\nu$ as $\ex*{\abs{\mu-\nu}}g$ for some
	non-negative measurable function $g\in\L^0(\Omega)$, and their main result
	\citep[Theorem 2.1]{BV05} bounds this weighted total variation in terms of
	the KL divergence.

	We rederive their result by noting that the $g$-weighted total variation is
	$\ipm[g\B]\mu\nu$ for $g\B=\set{g\cdot b\given b\in\B}$ where $\B$ is the
	set of measurable functions taking values in $[-1,1]$, so that it suffices
	by \cref{thm:ipm-lb-arbitrary-N} to upper bound $\cgf{g\cdot b}\nu(t)$ for
	each $b\in\B$ in terms of $\log\ex*\nu{e^g}$ or $\log\ex*\nu{e^{g^2}}$.
	But since $g\geq 0$, we have $g\cdot b\leq\abs g=g$ and we conclude by using
	the fact that finiteness of $\log\ex*\nu{e^h}$ (resp.\
	$\log\ex*\nu{e^{h^2}}$) implies a quadratic upper bound on the
	centered log-moment generating function $\cgf h\nu(t)$ for $\abs t\leq 1/4$
	(resp.\ all $t\in\R$) for any non-negative function $h$ (see
	e.g.~\citet[Propositions 2.5.2 and 2.7.1]{V18}).
\end{example}

Finally, we show that when we take $N=X^1$ in \cref{thm:ipm-lb-arbitrary-N},
that is, we want a lower bound $L$ such that $\di\mu\nu\geq L(\ipm\mu\nu)$ for
all probability measures $\mu$ and $\nu$ in $X$, we no longer need to consider
pairs of measures such that $\mu\not\ll \nu$, and in particular we can ignore
the $\sigma$-ideal $\Xi$ in the derivation of the bound. Intuitively, this is
because when $\mu\not\ll\nu$, we now have sufficiently many measures in $N$ to
approximate $\nu$ with a measure $\nu'$ such that $\mu\ll\nu'$.

\begin{theorem}
  \label{thm:ipm-lb-smoothed}
  Let $(X, Y)$ be a dual pair with $X\subseteq \M$ and assume that $(X,Y)$
  is decomposable with respect to all probability measures in $X$.
  Then for all subsets of functions $\G\subseteq Y$,
	\[
		\con{\lb[\,X^1]\G{X^1}}=\con{\paren*{\inf_{g\in\G}\lb[\,X^1]g{X^1}}}=
		\cgf\G{X^1}\,.
	\]
	In particular, for any $\sigma$-ideal $\Sigma$ and $\G\subseteq L^0(\Sigma)$,
  the following are equivalent for every convex lsc $L:\R_{\geq 0}\to\eR$:
	\begin{enum}
  \item $\di\mu\nu\geq L(\ipm\mu\nu)$ for all $\mu,\nu\in \M_c(\Sigma)$ integrating
    all of $\G$.
		\item $\di\mu\nu\geq L(\abs{\ex\mu g-\ex\nu g})$ for all $g\in\G$ and
			$\mu,\nu\in \M_c(\Sigma)$ integrating all of $\G$.
		\item $\cgf g\nu(t)\leq \con L(\abs t)$ for all $t\in\R$, $g\in\G$, and
			$\nu\in \M_c(\Sigma)$ integrating all of $\G$.
	\end{enum}
\end{theorem}
\begin{proof}
  The in particular claim follows from the main claim applied to $(X_\G, Y_\G)$ by an argument analogous
  to that of \cref{cor:optimal-lb-operational}. For the main claim, by
  \cref{thm:ipm-lb-arbitrary-N}, it suffices to show that
  $\inf_{g\in\G}\lb[\,X^1]g{X^1}=\inf_{g\in\G,\nu\in X^1}\lb[\,X^1]g\nu$
and $\inf_{g\in\G,\nu\in X^1}\lb[\,X^1\cap \M_c(\nu)]g\nu$
have the same conjugate, or simply the same lsc regularization. Since the
former is definitionally no larger than the latter, it suffices to show that
any lsc lower bound $L$ for the latter also lower bounds the former, equivalently,
that if $\di\mu\nu\geq L(\ex\mu g-\ex\nu g)$ for all $\mu\ll\nu \in X^1$ and $g\in\G$,
then this also holds for all $\mu,\nu\in X^1$.

	Given any $\mu,\nu\in X^1$ and $\delta\in[0,1]$, let $\nu_\delta=(1-\delta)\cdot \nu+\delta\cdot \mu$
	so that $\nu_\delta\in X^1$. Then for each $\delta\in[0,1]$ we have that
  $\ex\mu g-\ex{\nu_\delta}g =(1-\delta)(\ex\mu g-\ex\nu g)$ for all $g\in\G$,
	and furthermore, by convexity of $\di\mu\cdotarg$ we have
  for $\delta\in(0,1]$ that
  \[
    (1-\delta)\di\mu\nu=
    (1-\delta)\di\mu\nu +\delta\di\mu\mu
    \geq \di\mu{\nu_\delta}
    \geq L\paren[\big]{(1-\delta)\paren{\ex\mu g - \ex\nu g}}
   \]
  where the last inequality is because $\mu\ll\nu_\delta$.
  But since $L$ is lower semicontinuous, we have that
  $L(\ex\mu g-\ex\nu g)\leq\liminf_{\delta\to 0}L\paren[\big]{(1-\delta)(\ex\mu g-\ex\nu g)}
  \leq \lim_{\delta\to 0^{-}}L\paren[\big]{(1-\delta)(\ex\mu g-\ex\nu g)}$,
  and so we get that $\di\mu\nu\geq L(\ex\mu g-\ex\nu g)$ as desired.
\end{proof}

\subsection{Application to bounded functions and the total variation}
\label{sec:bounded}

In this section, we consider the problem of lower bounding the
$\phi$-divergence by a function of the total variation distance. Though it is
a well-studied problem and most of the results we derive are already known, we
consider this case to demonstrate the applicability of the results obtained in
\cref{sec:ipm-lb-comp}. In \cref{sec:vajda}, we study Vajda's problem
\citep{V72}: obtaining the best lower bound of the $\phi$-divergence by
a function of the total variation distance, and in \cref{sec:pinsker} we show
how to obtain quadratic relaxations of the best lower bound as in Pinsker's
inequality and Hoeffding's lemma.

\subsubsection{Vajda's problem}
\label{sec:vajda}

The Vajda problem \citep{V72} is to quantify the optimal relationship between
the $\phi$-divergence and the total variation, that is to compute the function
\begin{align*}
	\lbipm\B{\M^1}(\eps)
	&=\inf\set*{\di\mu\nu\given(\mu,\nu)\in\M^1\times\M^1\land\tv\mu\nu=\eps}\\
	&=\inf\set*{\di\mu\nu\given (\mu,\nu)\in\M^1\times\M^1\land\ipm[\B]\mu\nu=\eps}
\end{align*}
where $\B$ is the set of measurable functions $\Omega\to[-1,1]$. In this
section, we use \cref{thm:ipm-lb-smoothed} to give for an arbitrary $\phi$ an
expression for the Vajda function as the convex conjugate of a natural
geometric quantity associated with the function $\con\psi$, the inverse of its
\emph{sublevel set volume function}, which we call the \emph{height-for-width}
function.

\begin{definition}
The \emph{sublevel set volume} function $\sls{\con\psi}:\R_{\geq 0} \to\R_{\geq
0}$ maps $h\in\R$ to the Lebesgue measure of the sublevel set
$\set*{x\in\R\given \con\psi(x)\leq h}$. Since $\con\psi$ is convex and
inf-compact, the sublevel sets are compact intervals and their Lebesgue measure
is simply their length.

The \emph{height-for-width} function $\height{\con\psi}:\R_{\geq 0}
\to\overline\R$ is the (right) inverse of the sublevel set volume
function given by $\height{\con\psi}(w)
=\inf\set*{h\in\overline\R\given \sls{\con\psi}(h)\geq w}$.
\end{definition}

To understand this definition, note that since $\con\psi$ is defined
on $\R$, the sublevel set volume function can be interpreted as
giving for each height $h$ the length of longest horizontal line segment
that can be placed in the epigraph of $\con\psi$ but no higher
than $h$. The inverse, the height-for-width function, asks for the
minimal height at which one can place a horizontal line segment of
length $w$ in the epigraph of $\con\psi$. See \cref{fig:height} for
an illustration of this in the case of $\con\psi(x)=e^x-x-1$,
corresponding to the Kullback--Leibler divergence.

The following lemma shows that the height-for-width function can be
equivalently formulated as the optimal value of a simple convex optimization
problem.

\begin{lemma}
\label{lem:height-equivalence}
For all $w\in\R_{\geq 0}$, $\height{\con\psi}(w)=\inf_{\lambda\in\R}\max\set*{
\con\psi(\lambda+w/2),\con\psi(\lambda-w/2)}$. Furthermore, if for $w>0$ there
exists $\lambda_w$ such that $\con\psi(\lambda_w-w/2)=\con\psi(\lambda_w+w/2)$,
then $\height{\con\psi}(w)= \con\psi(\lambda_w-w/2)=\con\psi(\lambda_w+w/2)$.
\end{lemma}
\begin{proof}
For every $w\geq 0$, define the function
$h_w:\lambda\mapsto\max\set{\con\psi(\lambda-w/2),\con\psi(\lambda+w/2)}$
which is the supremum of two convex inf-compact functions with overlapping
domain, and so is itself proper, convex, and inf-compact. In particular,
$h_w$ achieves its global minimum $y_w\in\R$, where by definition and
convexity of $\con\psi$ we have $y_w$ is the smallest number such that there
exists an interval $[\lambda-w/2,\lambda+w/2]$ of length $w$ such that
$\con\psi([\lambda-w/2,\lambda+w/2])\subseteq (-\infty,y_w]$, and thus
$y_w=\inf\set*{x\in\overline\R\given \sls{\con\psi}(x)\geq
w}=\height{\con\psi}(w)$ as desired.

For the remaining claim, consider $w>0$ for which there is $\lambda_w\in\R$ such that
$\con\psi(\lambda_w-w/2)=\con\psi(\lambda_w+w/2)$. By convexity of $\con\psi$
we have for every $\lambda<\lambda_w$ that $\con\psi(\lambda-w/2)\geq
\con\psi(\lambda_w-w/2)$, and analogously for every $\lambda>\lambda_w$ that
$\con\psi(\lambda+w/2)\geq\con\psi(\lambda_w+w/2)$. Thus for every $\lambda$
we have $\max\set{\con\psi(\lambda-w/2),\con\psi(\lambda+w/2)}\geq \min\set{
	\con\psi(\lambda_w-w/2),\con\psi(\lambda_w+w/2)}=\con\psi(\lambda_w-w/2)
=\con\psi(\lambda_w+w/2)$, so the result follows from the main claim.
\end{proof}

\begin{figure*}[htbp]
	\centering
	\BeginAccSupp{method=escape,Alt={A graph of the function
		f(x)=exp(x)-x-1, with a horizontal line segment at approximately 0.12
		which is placed in the epigraph of f and has length 1 and
		has both of its endpoints on the curve, along with a similar
		line at height approximately 1.01 of length 3, and a final
		line at a height(t) between lambda(t)-t/2 and lambda(t)+t/2}}
	\begin{tikzpicture}[scale=1.6]
		\draw[domain=-3.5:1.64361,smooth,variable=\x,thick] plot ({\x},{exp(\x)-\x-1});
		\draw[dashed] (-4,0) -- (2,0);
		\draw[dashed] (-4,0) -- (-4,2.55);

		\node[anchor=north] at (-0.541,0) {$\approx -0.54$};
		\node[anchor=north] at (0.458,0) {$\approx 0.46$};
		\node[anchor=north] at (1.52,0) {$\lambda(w)+w/2$};
		\node[circle,fill,inner sep=0.25ex] at (1.52,0) {};
		\node[circle,fill,inner sep=0.25ex] at (-3,0) {};
		\node[circle,fill,inner sep=0.25ex] at (0,0) {};
		\node[circle,fill,inner sep=0.25ex] at (0.46,0) {};
		\node[circle,fill,inner sep=0.25ex] at (-0.54,0) {};
		\node[circle,fill,inner sep=0.25ex] at (-4,0.123) {};
		\node[circle,fill,inner sep=0.25ex] at (-4,1.01) {};
		\node[circle,fill,inner sep=0.25ex] at (-4,2.05) {};
		\node[anchor=north] at (-3,0) {$\lambda(w)-w/2$};
		\node[anchor=north] at (0,0) {$0$};
		\node[anchor=east] at (-4,0.123) {$\height{\con\psi}(1)\approx 0.12$};
		\node[anchor=east] at (-4,2.05) {$\height{\con\psi}(w)$};
		\node[anchor=east] at (-4,1.01) {$\height{\con\psi}(3)\approx 1.01$};
		\draw[<->] (-0.541,0.123) -- node[yshift=2.2ex] {width $1$} (0.458,0.123);
		\draw[<->] (-3,2.05) -- node[yshift=2.2ex] {width $w$} (1.52,2.05);
		\draw[<->] (-1.85,1.01) -- node[yshift=2.2ex] {width $3$} (1.15,1.01);
	\end{tikzpicture}
	\EndAccSupp{}
	\caption{Illustration of height-for-width function for $\con\psi(x)=e^x-x-1$}
	\label{fig:height}
\end{figure*}
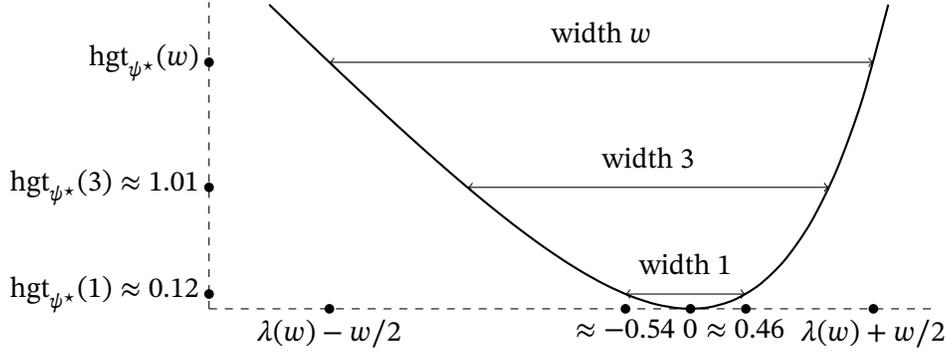

\begin{example}
\label{example:height-kl}
	For the case of the KL divergence for which $\con\psi(w)=e^w-w-1$, one can
	compute that $\con\psi(\lambda(w)+w/2)=\con\psi(\lambda(w)-w/2)$ for
	$\lambda(w)=-\log\frac{e^{w/2}-e^{-w/2}}{w}=-\log\frac{2\sinh(w/2)}w$, so
	that $\height{\con\psi}(w)=-1+\frac w2\coth\frac w2+\log\frac{2\sinh(w/2)}w$.
\end{example}

The duality result of \cref{thm:ipm-lb-smoothed} computes the
biconjugate of the optimal bound $\lbipm\B{\M^1}$, so we first prove
that this function is convex and lsc.
\begin{lemma}
\label{lem:lb-is-k-b-star}
	Let $M$ the set of probability measures supported on $\set{-1,1}$.
	Then $\lbipm\B{\M^1}=\lbipm{\ident_{\set{-1,1}}}M$ is convex and lower
	semicontinuous. In particular
	$\lbipm\B{\M^1}(\eps)=\con{\cgf\B{\M^1}}(\eps)$ for $\eps\geq 0$.
\end{lemma}
\begin{proof}
	By \cref{thm:ipm-lb-smoothed} we have that $\con{\lbipm\B{\M^1}}
	=\cgf\B{\M^1}$, so the in particular statement follows immediately from
	the main claim.
	The main claim, that $\lbipm\B{\M^1}=\lbipm{\ident_{\set{-1,1}}}M$ is
	convex and lower semicontinuous, is well-known and can easily be
	derived using the methods of e.g.~\citet{V72}, but we include a proof
	here in our language for completeness and to illustrate how it could
	be generalized beyond the total variation.

	Note that the set $\B=[-1,1]^\Omega\cap\L^0(\Omega)$ is convex, and
	furthermore is $\sigma(\L^b(\Omega),\M)$-compact by the Banach--Alaoglu
	theorem, and so by the Krein--Milman theorem $\B$ is the
	$\sigma(\L^b(\Omega),\M)$-closed convex hull of its extreme points
	$\extreme(\B)=\set{-1,1}^\Omega \cap\L^0(\Omega)$ the set of
	measurable $\set{-1,1}$-valued functions. Thus,
	\cref{lem:ipm-preserved-closed-convex} implies
	$\ipmsymb\B=\ipmsymb{\extreme(\B)}$, and so
	$\lbipm\B{\M^1}=\lbipm{\extreme(\B)}{\M^1}$.

	We now prove that $\inf_{g\in\extreme(\B)}\lb g{\M^1}$ is convex
	and lsc, which by \cref{cor:optimal-lb-equivalence} also implies
	$\lbipm{\extreme(\B)}{\M^1}=\inf_{g\in\extreme(\B)}\lb g{\M^1}$
	is convex and lsc.
	By \cref{lem:lb-pushforward}, for each $g\in\extreme(\B)$
	we have $\lb g{\M^1}= \lb[M_g]{\ident_{\set{-1,1}}}{M_g}$ for $M_g=
	\set{\pushforward\mu g\given \mu\in \M^1}$. In particular,
	if $g$ is constant this set is the singleton $M_g=\set{\delta_{g(\Omega)}}$,
	and if $g$ is non-constant then it is exactly the set $M$ of
	probability measures supported on $\set{-1,1}$.
	Thus, $\inf_{g\in\extreme(\B)}\lb g{\M^1}=
		\lb[M]{\ident_{\set{-1,1}}}M$.

	Note that the set $M$ with the total variation norm is homeomorphic
	to the unit interval $[0,1]$ via the linear map
	$p\mapsto p\cdot \delta_{\set 1}+(1-p)\cdot\delta_{\set{-1}}$.
	Then the function
	$f:\R\times M^2\to \eR$ given by $f\paren[\big]{\eps,(\mu,\nu)}
	=\di\mu\nu+\delta_{\set 0}\paren*{\ex\mu{\ident_{\set{-1,1}}}-\ex\nu{\ident_{\set{-1,1}}}-\eps}$
	is jointly convex and lower semicontinuous, and hence since
	$M$ is compact also inf-compact.
	Thus, by \cref{lem:optimal-value-convex}, the function
	$\lb[M]{\ident_{\set{-1,1}}}M=\inf_{(\mu,\nu)\in M^2} f\paren[\big]{\cdotarg,(\mu,\nu)}$
	is convex and inf-compact as desired.
\end{proof}

\Cref{lem:lb-is-k-b-star} implies that it suffices to compute $\cgf\B{\M^1}$.
\begin{lemma}
\label{lem:k-b-value}
	$\cgf\B{\M^1}(t)=\height{\con\psi}(2t)$ for every $t\geq0$.
\end{lemma}
\begin{proof}
	For $M=\set{p\cdot \delta_{\set 1}+(1-p)\cdot\delta_{\set{-1}}
		\given p\in[0,1]}$,
	we have by \cref{lem:lb-is-k-b-star,thm:ipm-lb-smoothed}
	that $\cgf\B{\M^1}=\con{\lbipm\B{\M^1}}=
	\con{\lb[M]{\ident_{\set{-1,1}}}M}=\sup_{\nu\in M}\cgf{\ident_{\set{-1,1}}}\nu$.
	For $p\in[0,1]$ we have
	$\cgf{\ident_{\set{-1,1}}}{p\cdot \delta_{\set 1}+(1-p)\cdot\delta_{\set{-1}}}
	=\inf_{\lambda\in\R}\paren[\big]{
		p\cdot \con\psi(t+\lambda)+(1-p)\cdot \con\psi(-t+\lambda)}$,
	so that
	\begin{equation}
		\cgf\B{\M^1}(t)=\sup_{p\in[0,1]}\inf_{\lambda\in\R}
		\paren[\big]{p\cdot\con\psi(\lambda+t)+(1-p)\cdot\con\psi(\lambda-t)}
		\,.
		\label{eqn:kb-minmax}
	\end{equation}
	This mixed optimization problem is convex in $\lambda$ for each $p$
	and linear in $p$ for each $\lambda\in\R$, and the interval $[0,1]$
	is compact, so by the Sion minimax theorem \citep{Sion58} we can
	swap the supremum and infimum to get
	\begin{align*}
		\cgf\B{\M^1}(t)
		&=\inf_{\lambda\in\R}\sup_{p\in[0,1]}
		\paren[\big]{p\cdot\con\psi(\lambda+t)+(1-p)\cdot\con\psi(\lambda-t)}
		\\
		&=\inf_{\lambda\in\R}\max\set{\con\psi(\lambda+t),\con\psi(\lambda-t)}
	\end{align*}
	so the claim follows from \cref{lem:height-equivalence}.
\end{proof}

\begin{example}
\label{exmp:optimal-hoeffding}
	For the Kullback--Leibler divergence, since $\cgf g\nu(t)
	=\log\ex\nu{e^{t(g-\ex\nu g)}}$ as in \cref{ex:dv},
	\cref{lem:k-b-value,example:height-kl}
	imply that the optimal bound on the
	cumulant generating function of a random variable $g$ with $\ex\nu g=0$ and
	$m\leq g\leq M$ $\nu$-a.s. is $\log\ex\nu{e^{tg}}\leq
	\height{\con\psi}{[(M-m)t]}
	=-1+\frac{M-m}2\coth\frac{M-m}2+\log\frac{2\sinh((M-m)t/2)}t$. This is
	a refinement of Hoeffding's lemma, which gives the upper bound of $(M-m)^2t^2/8$,
	which we will also derive as consequence of a general quadratic
	relaxation on the height function in \cref{exmp:hoeffding}.
\end{example}

\begin{corollary}
\label{cor:vajda}
	$\lbipm\B{\M^1}(\eps)=\con{\height{\con\psi}}(\eps/2)$ for all $\eps\geq0$.
	In particular, if $\height{\con\psi}$ is differentiable then
	$\lbipm\B{\M^1}(2\height{\con\psi}'(x))=x\height{\con\psi}'(x)
	-\height{\con\psi}(x)$.
\end{corollary}
\begin{proof}
	The main claim is immediate from \cref{lem:lb-is-k-b-star,lem:k-b-value},
	and the supplemental claim follows from the explicit expression for the
	convex conjugate for differentiable functions.
\end{proof}

\begin{example}
	For the Kullback--Leibler divergence, using \cref{example:height-kl},
	the supplemental claim of \cref{cor:vajda} applied to $x=2t$ gives
	$\lbipm\B{\M^1}(V(t))=\log\frac t{\sinh t}+t\coth t-\frac{t^2}{\sinh^2 t}$
	for $V(t)=2\coth t-\frac t{\sinh^2 t} -1/t$, which is exactly the
	formula derived by \citet{FHT03}.
\end{example}

\begin{remark}
	\Cref{cor:vajda} shows that lower bounds on the $\phi$-divergence
	in terms of the total variation are equivalent to upper bounds on
	the height-for-width function $\height{\con\psi}$, equivalently to
	lower bounds on the sublevel set volume function of $\con\psi$.
	The complementary problem of obtaining upper bounds on the
	sublevel set volume function is of interest in harmonic analysis
	due to its connection to studying oscillatory integrals (e.g.\
	\citet[Chapter 8, Proposition 2]{S93} and \citet[\S1-2]{CCW99}), and it
	would be interesting to see if techniques from that literature
	could be applied in this context.
\end{remark}

\begin{remark}
\label{rmrk:height-dagger}
	Since the total variation $\tv\mu\nu$ is symmetric in terms of
	$\mu$ and $\nu$, the optimal lower bound on $\di\mu\nu$
	in terms of $\tv\mu\nu$ is the same as the optimal lower bound
	on $\di\nu\mu=\di[\phi^\dagger]\mu\nu$ for $\phi^\dagger=x\phi(1/x)$.
	By \cref{cor:vajda}, this implies that $\height{\con\psi}
	=\height{\con{(\psi^\dagger)}}$ (note that this can also be
	derived directly from the definition).
\end{remark}

\subsubsection{Application to Pinsker-type inequalities}
\label{sec:pinsker}

\Cref{cor:vajda} implies that to obtain Pinsker-type inequalities, it suffices
to upper bound the height function $\height{\con\psi}(t)$ by a quadratic
function of $t$. In this section, we show such bounds under mild assumptions
on $\con\psi$, both rederiving optimal Pinsker-type inequalities for the
Kullback--Leibler divergence and $\alpha$-divergences for $-1\leq\alpha\leq 2$
due to \citea{Gilardoni}{G10}, and deriving new but not necessarily optimal
Pinsker-type inequalities for all $\alpha\in\R$. We proceed by giving two
arguments approximating the minimizer $\lambda(t)$ in the optimization
problem defining the height (\cref{lem:height-equivalence}), and an argument
that works directly with the optimal $\lambda(t)$.

We begin with the crudest but most widely applicable bound.
\begin{corollary}
\label{cor:psipp-monotone}
If $\phi$ is twice differentiable on its domain and $\phi''$ is monotone,
then $\height{\con\psi}(t)\leq t^2/(2\phi''(1))$ for all $t\geq 0$.
Equivalently, for such $\phi$ we have that $\di\mu\nu\geq \frac{\phi''(1)}8
\cdot \tv\mu\nu^2$ for all $\mu,\nu\in\M^1$.
\end{corollary}
\begin{proof}
	If $\phi''(1)=0$, then the claim is trivial, so we assume that $\phi''(1)>0$.
	If $\phi''$ is non-decreasing, we have by Taylor's theorem that
	$\phi(x)\geq \frac{\phi''(1)}2 (x-1)^2$ for $x\geq 1$, equivalently
	$\psi(x)\geq \frac{\phi''(1)}2x^2$ for $x\geq 0$, so that
	$\con\psi(x)\leq \frac{1}{2\phi''(1)}x^2$ for $x\geq 0$.
	Then $\height{\con\psi}(t)=\inf_{\lambda\in\R}\max\set{\con\psi(\lambda-t/2),
	\con\psi(\lambda+t/2)}\leq\max\set{\con\psi(0),\con\psi(t)}\leq t^2/(2\phi''(1))$.
	On the other hand, if $\phi''$ is non-increasing, then analogously we have
	$\con\psi(x)\leq \frac{1}{2\phi''(1)}x^2$ for $x\leq 0$, so that
	Then $\height{\con\psi}(t)=\inf_{\lambda\in\R}\max\set{\con\psi(\lambda-t/2),
	\con\psi(\lambda+t/2)}\leq\max\set{\con\psi(0),\con\psi(-t)}\leq t^2/(2\phi''(1))$.
\end{proof}

\begin{example}
\label{example:crude-pinsker}
	Most of the standard $\phi$-divergences satisfy the condition of
	\cref{cor:psipp-monotone}, in particular the $\alpha$-divergences given
	by $\phi_\alpha=\frac{x^\alpha-\alpha(x-1)-1}{\alpha(\alpha-1)}$
	have $\phi''_\alpha(x)=x^{\alpha-2}$ which is monotone for all $\alpha$.
	As a result, we get for all $\alpha$ the (possibly suboptimal) Pinsker
	inequality $\di[\phi_\alpha]\mu\nu\geq\frac18\cdot\tv\mu\nu^2$ for
	all $\mu,\nu\in\M^1$. Such a bound appears to be new for
	$\alpha>2$, but for $\alpha\in[-1,2]$ \citea{Gilardoni}{G10} established
	the better bound $\di[\phi_\alpha]\mu\nu\geq\frac12\cdot\tv\mu\nu^2$,
	extending the standard case of the Kullback--Leibler divergence $\alpha=1$.
	We rederive this optimal constant for these divergences below, and also
	give general conditions under which such bounds hold.
\end{example}

\Cref{cor:psipp-monotone} used the crude linear relaxation
$-t/2\leq\lambda(t)\leq t/2$. In the following Corollary, we derive a tighter
Pinsker-type inequality by using a Taylor expansion of $\lambda(t)$.

\begin{corollary}
\label{cor:height-thrice-diff}
	Suppose that $\phi$ strictly convex and twice differentiable on its domain,
	thrice differentiable at $1$ and that
	\[
		\frac{27\phi''(1)}{\paren*{3-z\phi'''(1)/\phi''(1)}^3}\leq \phi''(1+z)
	\]
	for all $z\geq -1$.
	Then $\height{\con\psi}(t)\leq t^2/(8\phi''(1))$ for all
	$t\geq 0$, equivalently, for such $\phi$ we have $\di\mu\nu\geq
	\frac{\phi''(1)}2 \cdot \tv\mu\nu^2$ for all $\mu,\nu\in\M^1$.
\end{corollary}
\begin{remark}
	The Pinsker constant in \cref{cor:height-thrice-diff} is
	best-possible, since if $\phi$ is twice-differentiable at $1$, then
	Taylor's theorem gives the local expansion
	$\phi(x)=\phi''(1)/2\cdot(x-1)^2 + o\paren*{(x-1)^2}$, and thus the
	distributions $\mu_\eps=(1/2+\eps/2,1/2-\eps/2)$ and $\nu=(1/2,1/2)$
	on the set $\set{0,1}$ have $\tv{\mu_\eps}\nu=\eps$ and
	$\di{\mu_\eps}\nu=\phi''(1)/2\cdot\eps^2 +o(\eps^2)$.
\end{remark}
\begin{proof}
	Under suitable regularity assumptions on $\phi$ and
	$\con\psi$, one can easily show that the second order
	expansion of the function $\lambda(t)$ implicitly defined
	by $\con\psi(\lambda(t)+t/2)=\con\psi(\lambda(t)-t/2)$
	is $L(t)=-\frac{ct^2}{24}$ for
	$c=(\con\psi)'''(0)/(\con\psi)''(0)=-\phi'''(1)/\phi''(1)^2$.
	Taking this as given, we show under the stated assumptions
	of the proposition that for $L(t)=-\frac{ct^2}{24}$
	and $c=-\phi'''(1)/\phi''(1)^2$, we have that $\con\psi(L(t)+st/2)\leq
	t^2/(8\phi''(1))$ for $s\in\set{\pm 1}$. Since both sides are
	$0$ at $0$, it thus suffices to show $\paren[\big]{L'(t)+s/2}
	(\con\psi)'(L(t)+st/2)\leq t/(4\phi''(1))$. Now, let $\lesseqgtr$
	indicate $\leq$ if $L'(t)+s/2\geq 0$ and $\geq$ if $L'(t)+s/2\leq
	0$. Since $\phi$ strictly convex implies
	$\psi'=((\con\psi)')^{-1}$ is strictly increasing,
	we thus have that this is equivalent to
	\begin{equation}
		L(t)+st/2 \lesseqgtr \psi'\paren*{\frac{t/(4\phi''(1))}{L'(t)+s/2}}
		\label{eqn:relaxedheight}
	\end{equation}
	Write $z=\frac{t/(4\phi''(1))}{L'(t)+s/2}=\frac{t/(4\phi''(1))}{-ct/12+s/2}$
	so that $z$ has the same sign as $L'(t)+s/2$ and
	$t=\frac{6sz\phi''(1)}{3+cz\phi''(1)}$.
	Plugging this in and using the fact that $s^2=1$, we wish to show that
	\begin{equation}
		\frac{3z\phi''(1)(6+cz\phi''(1))}{2(3+cz\phi''(1))^2}-\psi'(z)
		\lesseqgtr 0
	\label{eqn:relaxedheightz}
	\end{equation}
	for all $z$ such that $t\geq 0$.
	The left hand side of \cref{eqn:relaxedheightz} is $0$ at $0$,
	so since $z>0$ implies $\lesseqgtr$ is $\leq$ and $z<0$ implies
	$\lesseqgtr$ is $\geq$, it suffices to show that the derivative
	of the left-hand side of \cref{eqn:relaxedheightz} with respect
	to $z$ is non-positive for all $z$. This derivative is
	\begin{equation}
		\frac{27\phi''(1)}{\paren*{3+cz\phi''(1)}^3}-\psi''(z)
		=
		\frac{27\phi''(1)}{\paren*{3-z\phi'''(1)/\phi''(1)}^3}-\phi''(1+z)
		\label{eqn:relaxedheightzp}
	\end{equation}
	which since $\dom\psi\subseteq [-1,\infty)$ is non-positive for
	all $z$ if and only if it is non-positive for all $z\geq -1$.
\end{proof}

\begin{example}[\citet{G10}]
\label{example:height-alpha}
	For the $\alpha$-divergences, we have $\phi_\alpha''(x)=x^{\alpha-2}$,
	and $\phi_\alpha'''(x)=(\alpha-2)x^{\alpha-3}$ so that
	\cref{cor:height-thrice-diff} is equivalent to the condition
	$\frac{27}{(3+(2-\alpha)z)^3}\leq (1+z)^{\alpha-2}$ for $z\geq
	-1$. Note that this is true for $z=0$ for all $\alpha$, and the
	derivative of $\frac{27(1+z)^{2-\alpha}}{(3+(2-\alpha)z)^3}$
	with respect to $z$ is
	$\frac{27(\alpha-2)(\alpha+1)z(1+z)^{1-\alpha}}{(3+(2-\alpha)z)^4}$.
	Thus, for $\alpha\in[-1,2]$ the sign of the derivative is
	the opposite of the sign of $z$, and the condition holds for
	all $z\geq -1$, recovering the result of \citea{Gilardoni}{G10}
	as desired.
\end{example}

\begin{example}
\label{exmp:hoeffding}
	For the case of the Kullback--Leibler divergence,
	\cref{example:height-alpha} rederives Pinsker's inequality and Hoeffding's
	lemma.
\end{example}

Finally, we show that one can also obtain optimal Pinsker-type
inequalities while arguing directly about the optimal $\lambda(t)$,
for which we need the following lemma.
\begin{lemma}
\label{lem:height-derivative}
	Suppose that $f:\R\to\overline\R$ is a convex function continuously
	differentiable on $(a,b)$ the interior of its domain with a unique
	global minimum and such that $\lim_{x\to a^+}f(x)=\infty=\lim_{x\to b^-}f(x)$.
	Then there is a continuously differentiable function $\lambda:(a-b,b-a)\to\R$ such
	that $\height f(t)=f(\lambda(t)+t/2)=f(\lambda(t)-t/2)$
	and
	\begin{align}
		\lambda'(t)
		&=
		\frac{f'\left(\lambda(t)+t/2\right)+f'\left(\lambda(t)-t/2\right)}{2
		\left(f'\left(\lambda(t)-t/2\right)-f'\left(\lambda(t)+t/2\right)\right)}
		\label{eqn:lambdader}\\
		\height f'(t)
		&=
		\frac{f'(\lambda(t)+t/2)f'(\lambda(t)-t/2)}{f'(\lambda(t)-t/2)-f'(\lambda(t)+t/2)}
		\,.
		\label{eqn:heightder}
	\end{align}
\end{lemma}
\begin{proof}
	For each $t\in(a-b,b-a)$, the function $\lambda\mapsto f(\lambda+t/2)-f(\lambda-t/2)$
	is continuously differentiable on its domain $(a+\frac{\abs t}2,b-\frac{\abs t}2)$,
	with limits $-\infty$ and $\infty$. Thus, for all such $t$ there exists
	$\lambda$ satisfying the implicit equation $f(\lambda(t)+t/2)=f(\lambda(t)-t/2)$,
	which by \cref{lem:height-equivalence} also defines $\height f(t)$.
	Furthermore, the fact that $f$ has a unique global minimum implies this function
	is strictly increasing in $\lambda$ for each $t$, and thus the implicit function
	theorem guarantees the existence of the claimed continuously differentiable
	$\lambda(t)$.

	Given the existence of $\lambda(t)$, we have by its definition that
	$\frac{d}{dt} f(\lambda(t)+t/2)=\frac{d}{dt} f(\lambda(t)-t/2)$,
	which implies by the chain rule the claimed value for
	$\lambda'(t)$, which since $\height f'(t)=\frac{d}{dt}
	f(\lambda(t)+t/2)$ implies the claimed expressions for the
	derivative of $\height f$.
\end{proof}

Using the previous lemma, we obtain the same optimal Pinsker-type inequality
as in \cref{cor:height-thrice-diff} under related but incomparable
assumptions.

\begin{corollary}
\label{cor:psipp-concave}
If $\phi$ is strictly convex, has a positive second derivative on its domain,
$1/\phi''$ is concave, and $\lim_{x\to\phi'(\infty)^-}
\con\psi(x)=\infty$ (e.g.~if $\phi'(\infty)=\infty$), then
$\height{\con\psi}(t)\leq t^2/(8\phi''(1))$ for
all $t\geq 0$. Equivalently, for such $\phi$ we have $\di\mu\nu\geq
	\frac{\phi''(1)}2 \cdot \tv\mu\nu^2$ for all $\mu,\nu\in\M^1$.
\end{corollary}
\begin{proof}
	By standard results in convex analysis, the
	existence and positivity of $\psi''$ imply that $\con\psi$
	is itself twice differentiable (e.g.\ \citet[Proposition
	6.2.5]{HUL93} or \citet[Proposition 1.1]{G91}). Thus, by
	\cref{lem:height-derivative}, it suffices to show that
	$\height{\con\psi}'(t)\leq t/(4\phi''(1))$, or equivalently
	\begin{equation}
		\frac{(\con\psi)'(\lambda(t)+t/2)(\con\psi)'(\lambda(t)-t/2)}{
		(\con\psi)'(\lambda(t)-t/2)-(\con\psi)'(\lambda(t)+t/2)}
		\leq \frac t{4\phi''(1)}.
		\label{eqn:heightderpsi}
	\end{equation}
	Since $\con\psi(\lambda(t)+t/2)=\con\psi(\lambda(t)-t/2)$ and
	$\con\psi$ has global minimum at $0$, we have $\lambda(t)-t/2\leq 0$
	and $\lambda(t)+t/2\geq 0$, and $(\con\psi)'(\lambda(t)-t/2)\leq 0$
	and $(\con\psi)'(\lambda(t)+t/2)\geq 0$. Thus, we have that
	the left-hand side of \cref{eqn:heightderpsi} is half the harmonic
	mean of $(\con\psi)'(\lambda(t)+t/2)$ and $-(\con\psi)'(\lambda(t)-t/2)$,
	so it suffices by the arithmetic mean--harmonic mean inequality to prove
	\begin{equation}
		(\con\psi)'(\lambda(t)+t/2)-(\con\psi)'(\lambda(t)-t/2)\leq \frac{t}{\phi''(1)}.
		\label{eqn:relaxedheightder}
	\end{equation}
	Since \cref{eqn:relaxedheightder} holds when $t=0$, it suffices
	to prove that
	\begin{equation}
		(1/2+\lambda'(t))\cdot (\con\psi)''(\lambda(t)+t/2)
		+(1/2-\lambda'(t))\cdot (\con\psi)''(\lambda(t)-t/2)\leq \frac1{\phi''(1)}.
	\end{equation}
	By the relationship between the second derivative of a
	function and the one of its conjugate (e.g.\ \citet[Proposition
	6.2.5]{HUL93}), this is equivalent to
	\begin{equation}
		\frac{1/2+\lambda'(t)}{\psi''\paren[\big]{(\con\psi)'(\lambda(t)+t/2)}}
		+
		\frac{1/2-\lambda'(t)}{\psi''\paren[\big]{(\con\psi)'(\lambda(t)-t/2)}}
		\leq
		\frac1{\phi''(1)}.
		\label{eqn:relaxedheightsecondder}
	\end{equation}
	Now, by \cref{eqn:lambdader}, we have that $\lambda'(t)\in[-1/2,1/2]$,
	so that by Jensen's inequality and the concavity of $1/\psi''$,
	the left-hand side of \cref{eqn:relaxedheightsecondder} is at most
	\begin{equation}
	1/\psi''\paren[\Big]{(1/2+\lambda'(t))(\con\psi)'(\lambda(t)+t/2)
	-(\lambda'(t)-1/2)(\con\psi)'(\lambda(t)-t/2)}.
	\label{eqn:jensenedder}
	\end{equation}
	Finally, since by definition
	$\con\psi(\lambda(t)+t/2)=\con\psi(\lambda(t)-t/2)$, the term
	inside $1/\psi''$ in \cref{eqn:jensenedder} is
	$0$, so since $\psi(x)=\phi(1+x)$ we are done.
\end{proof}

\begin{example}
	For the $\alpha$-divergences, we have
	$1/\phi_\alpha''(x)=x^{2-\alpha}$ which is concave for
	$\alpha\in[1,2]$, so \cref{cor:psipp-concave} applies for
	these divergences. Furthermore, by \cref{rmrk:height-dagger},
	we can consider the reverse $\alpha$-divergences with
	$\phi_\alpha^\dagger(x) =x\phi_\alpha(1/x)$ which has
	$1/(\phi_\alpha^\dagger)''(x) =x^{1+\alpha}$, which is concave
	for $\alpha\in[-1,0]$.
\end{example}
 \section{Discussion}
\label{sec:discussion}

Throughout this paper, the $\phi$-cumulant generating function has proved
central in explicitating the relationship between $\phi$-divergences and
integral probability metrics. As a starting point, the identity $\cgf g\nu
= \con{\lb g\nu}$ (\cref{thm:equiv}) expresses the cumulant generating function
as the convex conjugate of the best lower bound of $\di\mu\nu$ in terms of
$\ex\mu g - \ex\nu g$. This establishes a ``correspondence principle'' by which
properties of the relationship between $\phi$-divergences and integral
probability metrics translate by duality into properties of the cumulant
generating function, and vice versa. An advantage of this correspondence is
that the function $\cgf g\nu$, being expressed as the solution of
a single-dimensional convex optimization problem (\cref{def:cumulant}), is
arguably easier to evaluate and analyze than its counterpart $\lb g\nu$,
expressed as the solution to an infinite-dimensional optimization problem.
Following \cref{thm:equiv}, several results from the present paper can be seen
as instantiations of this ``correspondence principle'' and we summarize some of
them in \cref{table:summary}.

\begin{table}[ht]
	\centering
	\renewcommand{\arraystretch}{1.5}
	\begin{tabular}{@{}rm{12.5em}m{15em}@{}} \toprule
			Ref.& Property of the $\phi$-cumulant\newline generating function
					   & Property of the $\phi$-divergence\\

			\midrule

			\S\ref{sec:char}
			& $\cgf g\nu(t)\leq B(t)$ for all $t\in\R$
			&$\di\mu\nu\geq \con B\paren[\big]{\ex\mu g - \ex\nu g}$\newline for all
			$\mu\in X^1_g$\\

			\S\ref{sec:subexponential}
			& $0\in\inter(\dom\cgf g\nu)$ &
			$\di\mu\nu\geq L\paren[\big]{\abs{\ex\mu g - \ex\nu g}}$\newline for
			some $L\not\equiv 0$, all $\mu\in X^1_g$\\

			\S\ref{sec:continuity}
			& $\cgf g\nu$ differentiable at $0$ &
			$\di{\nu_n}\nu\to 0$ implies\newline $\ex{\nu_n}g\to\ex\nu g$ for
			all $(\nu_n)\in \paren*{X^1_g}^{\mathbb{N}}$\\

			\S\ref{sec:ipm-lb-comp}
			& $\cgf g\nu(t)\leq E(t)$ for all\newline $t\in\R$, $g\in\G$,
			$\nu\in X^1_\G$
			&$\di\mu\nu\geq \con E\paren[\big]{\ipm\mu\nu}$\newline for all
			$\mu,\nu\in X^1_\G$\\

			\S\ref{sec:bounded}
			& $\height{\con\psi}(2t)\leq B(t)$ for all $t\in\R$
			&$\di\mu\nu\geq \con B\paren[\big]{\tv\mu\nu}$\newline for all $\mu,\nu\in\M^1$\\
											  \bottomrule
		\end{tabular}
		\caption{Several examples, proved in this paper, of the dual
			correspondence between properties of the $\phi$-cumulant generating
			function and properties of the relationship between the
			$\phi$-divergence and mean deviations. Throughout, $\mu\in\M^1$,
			$g\in L^1(\nu)$, $B:\R\to\R$ is arbitrary, $E:\R\to\R$ is even,
			$\G\subseteq\L^0$, and $X^1_g$ and $X^1_\G$ are as in
      \cref{defn:G-dual-pair}.}
		\label{table:summary}
\end{table}

A limitation of this correspondence is that it only describes the optimal lower
bound function $\lb g \nu$ via its convex conjugate. When $\lb g \nu$ is lower
semicontinuous, this is without any loss of information by the Fenchel--Moreau
theorem, but in general this only provides information about the
\emph{biconjugate} $\bicon{\lb g\nu}$. While $\lb g\nu$ and $\bicon{\lb g\nu}$
differ in at most two points, as discussed in \cref{sec:char}, the difference
between the optimal lower bound and its biconjugate is potentially much more
important when considering a class of functions $\G$ or a class of measures $N$
as in \cref{sec:optimal-lb-def}. Some conditions under which this lower bound
$\lb \G N$ is necessarily convex and lower semicontinuous were derived in
\cref{sec:inf-compact,sec:bounded}, and we gave a characterization of $\lb \G N$
up to countably many points in \cref{rem:optimal-lb-lsc} regardless, but this
does not completely answer the question (cf.~Remarks~\ref{rem:lsc-ssubexp} and
\ref{rem:lb-convex}). We believe that an interesting direction for future work
would be to identify natural necessary or sufficient conditions under which
$\lb\G N$ is convex or lower semicontinuous.

\section*{Acknowledgements}

The authors are grateful to Flavio du Pin Calmon, Benjamin L.~Edelman, Julien
Fageot, Salil Vadhan, and the anonymous reviewers for their helpful comments
and suggestions on an earlier draft of this paper, and to Julien Fageot and
Salil Vadhan for suggesting the problem considered in \cref{sec:continuity}.
 
\let\cite\oldcite
\newcommand{\etalchar}[1]{$^{#1}$}

\bibliographystyle{alphainits}

\begin{thebibliography}{RRGGP12}

\bibitem[AH20]{AH20}
R.~Agrawal and T.~Horel.
\newblock Optimal {{Bounds}} between {{$f$}}-{{Divergences}} and {{Integral
  Probability Metrics}}.
\newblock In {\em Proceedings of the 37th {{International Conference}} on
  {{Machine Learning}} ({{ICML}} 2020)}, volume 119 of {\em Proceedings of
  {{Machine Learning Research}}}. {PMLR}, July 2020.

\bibitem[AS66]{AS66}
S.~M. Ali and S.~D. Silvey.
\newblock A general class of coefficients of divergence of one distribution
  from another.
\newblock {\em Journal of the Royal Statistical Society. Series B},
  28(1):131--142, 1966.

\bibitem[AS06]{AS06}
Y.~Altun and A.~Smola.
\newblock Unifying divergence minimization and statistical inference via convex
  duality.
\newblock In G.~Lugosi and H.~U. Simon, editors, {\em Learning Theory}, pages
  139--153, Berlin, Heidelberg, 2006. Springer.

\bibitem[BBR{\etalchar{+}}18]{BBROBCH18}
M.~I. Belghazi, A.~Baratin, S.~Rajeshwar, S.~Ozair, Y.~Bengio, A.~Courville,
  and D.~Hjelm.
\newblock Mutual information neural estimation.
\newblock In J.~Dy and A.~Krause, editors, {\em Proceedings of the 35th
  International Conference on Machine Learning}, volume~80 of {\em Proceedings
  of Machine Learning Research}, pages 531--540, Stockholmsmässan, Stockholm
  Sweden, 10--15 Jul 2018. PMLR.

\bibitem[BC77]{BC77}
A.~{Ben-Tal} and A.~Charnes.
\newblock A dual optimization framework for some problems of information theory
  and statistics.
\newblock Technical report, Center for Cybernetic Studies, University of Texas,
  Austin, 1977.

\bibitem[BC79]{BC79}
A.~{Ben-Tal} and A.~Charnes.
\newblock A dual optimization framework for some problems of information theory
  and statistics.
\newblock {\em Problems of Control and Information Theory. Problemy Upravlenija
  i Teorii Informacii}, 8(5-6):387--401, 1979.

\bibitem[BCR84]{BCR84}
C.~Berg, J.~P.~R. Christensen, and P.~Ressel.
\newblock {\em Introduction to Locally Convex Topological Vector Spaces and
  Dual Pairs}, pages 1--15.
\newblock Springer, New York, NY, 1984.

\bibitem[BG99]{BG99}
S.~G. Bobkov and F.~Götze.
\newblock Exponential integrability and transportation cost related to
  logarithmic {S}obolev inequalities.
\newblock {\em Journal of Functional Analysis}, 163(1):1--28, April 1999.

\bibitem[BH79]{BH79}
J.~Bretagnolle and C.~Huber.
\newblock Estimation des densités: risque minimax.
\newblock {\em {Z}eitschrift für {W}ahrscheinlichkeitstheorie und Verwandte
  Gebiete}, 47(2):119--137, Jan 1979.

\bibitem[BK06]{BK06}
M.~Broniatowski and A.~Keziou.
\newblock Minimization of $\phi$-divergences on sets of signed measures.
\newblock {\em Studia Scientiarum Mathematicarum Hungarica}, 43(4):403--442,
  2006.

\bibitem[BL91]{BL91}
J.~M. Borwein and A.~S. Lewis.
\newblock Duality relationships for entropy-like minimization problems.
\newblock {\em SIAM Journal on Control and Optimization}, 29(2):325--338, 1991.

\bibitem[BL93]{BL93}
J.~M. Borwein and A.~S. Lewis.
\newblock Partially-finite programming in ${L}_1$ and the existence of maximum
  entropy estimates.
\newblock {\em SIAM Journal on Optimization}, 3(2):248--267, 1993.

\bibitem[BLM13]{BLM13}
S.~Boucheron, G.~Lugosi, and P.~Massart.
\newblock {\em Concentration Inequalities: A Nonasymptotic Theory of
  Independence}.
\newblock Oxford University Press, Oxford, 2013.

\bibitem[Bou87]{B87}
N.~Bourbaki.
\newblock {\em Topological {{Vector Spaces}}}.
\newblock Elements of {{Mathematics}}. {Springer-Verlag}, {Berlin}, 1987.
\newblock Translated by H.G. Eggleston \& S. Madan from \emph{Espaces
  vectoriels topologiques}, Masson, Paris, 1981.

\bibitem[Br{\o}64]{B64}
A.~Br{\o}ndsted.
\newblock Conjugate convex functions in topological vector spaces.
\newblock {\em Matematisk-fysiske Meddelelser udgivet af det Kongelige Danske
  Videnskabernes Selskab}, 34(2):27, 1964.

\bibitem[BV05]{BV05}
F.~Bolley and C.~Villani.
\newblock Weighted {{Csisz\'ar}}-{{Kullback}}-{{Pinsker}} inequalities and
  applications to transportation inequalities.
\newblock {\em Annales de la Facult\'e des sciences de Toulouse :
  Math\'ematiques}, 14(3):331--352, 2005.

\bibitem[CCW99]{CCW99}
A.~Carbery, M.~Christ, and J.~Wright.
\newblock Multidimensional {{van der Corput}} and sublevel set estimates.
\newblock {\em Journal of the American Mathematical Society}, 12(4):981--1015,
  1999.

\bibitem[CGG99]{CGG99}
I.~Csiszár, F.~Gamboa, and E.~Gassiat.
\newblock {MEM} pixel correlated solutions for generalized moment and
  interpolation problems.
\newblock {\em IEEE Transactions on Information Theory}, 45(7):2253--2270,
  November 1999.

\bibitem[CM03]{CM03}
I.~Csisz{\'a}r and F.~Mat{\'u}{\v s}.
\newblock Information projections revisited.
\newblock {\em IEEE Transactions on Information Theory}, 49(6):1474--1490, June
  2003.

\bibitem[CM12]{CM12}
I.~Csisz{\'a}r and F.~Mat{\'u}š.
\newblock Generalized minimizers of convex integral functionals, {B}regman
  distance, {P}ythagorean identities.
\newblock {\em Kybernetika}, 48(4):637--689, 2012.

\bibitem[Csi62]{C62}
I.~Csisz{\'a}r.
\newblock {Informationstheoretische Konvergenzbegriffe im Raum der
  Wahrscheinlichkeitsverteilungen}.
\newblock {\em A Magyar Tudom\'anyos Akad\'emia. Matematikai Kutat\'o
  Int\'ezet\'enek K\"ozlem\'enyei}, 7:137--158, 1962.

\bibitem[Csi63]{C63}
I.~Csiszár.
\newblock Eine informationstheoretische {U}ngleichung und ihre {A}nwendung auf
  den {B}eweis der {E}rgodizität von {M}arkoffschen {K}etten.
\newblock {\em A Magyar Tudományos Akadémia Matematikai Kutató Intézetének
  Közleményei}, 8(1--2):85--108, 1963.

\bibitem[Csi64]{C64}
I.~Csisz{\'a}r.
\newblock {\"Uber topologische und metrische Eigenschaften der relativen
  Information der Ordnung {{$\alpha$}}}.
\newblock In {\em {Transactions of the Third Prague Conference on Information
  Theory, Statistical Decision Functions, Random Processes, 1962}}, pages
  63--73. {Publishing House of the Czechoslovak Academy of Science}, {Prague},
  1964.

\bibitem[Csi67a]{C67_Top}
I.~Csisz{\'a}r.
\newblock On topological properties of {{$f$}}-divergences.
\newblock {\em Studia Scientiarum Mathematicarum Hungarica}, 2:329--339, 1967.

\bibitem[Csi67b]{C67}
I.~Csiszár.
\newblock Information-type measures of difference of probability distributions
  and indirect observations.
\newblock {\em Studia Sci. Math. Hungar}, 2:299–318, 1967.

\bibitem[Csi75]{C75}
I.~Csiszár.
\newblock ${I}$-divergence geometry of probability distributions and
  minimization problems.
\newblock {\em Ann. Probab.}, 3(1):146--158, 02 1975.

\bibitem[DRG15]{DRG15}
G.~K. Dziugaite, D.~M. Roy, and Z.~Ghahramani.
\newblock Training generative neural networks via maximum mean discrepancy
  optimization.
\newblock In {\em Proceedings of the Thirty-First Conference on Uncertainty in
  Artificial Intelligence}, UAI’15, page 258–267, Arlington, Virginia, USA,
  2015. AUAI Press.

\bibitem[Dud64]{D64}
R.~M. Dudley.
\newblock On sequential convergence.
\newblock {\em Transactions of the American Mathematical Society},
  112(3):483--507, 1964.

\bibitem[Dud98]{D98}
R.~M. Dudley.
\newblock Consistency of {{M}}-{{Estimators}} and {{One}}-{{Sided Bracketing}}.
\newblock In E.~Eberlein, M.~Hahn, and M.~Talagrand, editors, {\em High
  {{Dimensional Probability}}}, pages 33--58. {Birkh\"auser Basel}, {Basel},
  1998.

\bibitem[DV76]{DV76}
M.~D. Donsker and S.~R.~S. Varadhan.
\newblock Asymptotic evaluation of certain {M}arkov process expectations for
  large time—{III}.
\newblock {\em Communications on Pure and Applied Mathematics}, 29(4):389--461,
  1976.

\bibitem[ES89]{ES89}
G.~A. Edgar and L.~Sucheston.
\newblock On maximal inequalities in {{Orlicz}} spaces.
\newblock In R.~D. Mauldin, R.~M. Shortt, and C.~E. Silva, editors, {\em
  Contemporary {{Mathematics}}}, volume~94, pages 113--129. {American
  Mathematical Society}, {Providence, Rhode Island}, 1989.

\bibitem[ET99]{ET99}
I.~Ekeland and R.~Témam.
\newblock {\em Convex Analysis and Variational Problems}.
\newblock Society for Industrial and Applied Mathematics, 1999.

\bibitem[FHT03]{FHT03}
A.~A. Fedotov, P.~Harremoës, and F.~Topsøe.
\newblock Refinements of {P}insker's inequality.
\newblock {\em IEEE Transactions on Information Theory}, 49(6):1491--1498, June
  2003.

\bibitem[GBR{\etalchar{+}}12]{GBRSS12}
A.~Gretton, K.~M. Borgwardt, M.~J. Rasch, B.~Sch{\"o}lkopf, and A.~Smola.
\newblock A kernel two-sample test.
\newblock {\em J. Mach. Learn. Res.}, 13(25):723–773, March 2012.

\bibitem[Gil06]{G06}
G.~L. Gilardoni.
\newblock On the minimum $f$-divergence for given total variation.
\newblock {\em Comptes Rendus Mathematique}, 343(11):763 -- 766, 2006.

\bibitem[Gil08]{G08}
G.~L. Gilardoni.
\newblock An improvement on {V}ajda's inequality.
\newblock In V.~Sidoravicius and M.~E. Vares, editors, {\em In and Out of
  Equilibrium 2}, volume~60 of {\em Progress in Probability}, pages 299--304.
  Birkhäuser Basel, Basel, 2008.

\bibitem[Gil10]{G10}
G.~L. Gilardoni.
\newblock On {P}insker's and {V}ajda's type inequalities for {C}siszár's
  $f$-divergences.
\newblock {\em IEEE Transactions on Information Theory}, 56(11):5377--5386, Nov
  2010.

\bibitem[GL10]{GL10}
N.~Gozlan and C.~Léonard.
\newblock Transport inequalities. {A} survey.
\newblock {\em Markov Processes and Related Fields}, 16(4):635--736, 2010.

\bibitem[Gor91]{G91}
G.~Gorni.
\newblock Conjugation and second-order properties of convex functions.
\newblock {\em Journal of Mathematical Analysis and Applications},
  158(2):293--315, July 1991.

\bibitem[GSS14]{GSS14}
A.~Guntuboyina, S.~Saha, and G.~Schiebinger.
\newblock Sharp inequalities for $f$-divergences.
\newblock {\em IEEE Transactions on Information Theory}, 60(1):104--121, Jan
  2014.

\bibitem[Har07]{H07}
P.~Harremo{\"e}s.
\newblock Information {{Topologies}} with {{Applications}}.
\newblock In I.~Csisz{\'a}r, G.~O.~H. Katona, G.~Tardos, and G.~Wiener,
  editors, {\em Entropy, {{Search}}, {{Complexity}}}, volume~16, pages
  113--150. {Springer Berlin Heidelberg}, {Berlin, Heidelberg}, 2007.
\newblock Series Title: Bolyai Society Mathematical Studies.

\bibitem[Her67]{H67}
H.~H. Herda.
\newblock On non-symmetric modular spaces.
\newblock {\em Colloquium Mathematicum}, 17(2):333--346, 1967.

\bibitem[HL93]{HUL93}
J.-B. {Hiriart-Urruty} and C.~Lemar{\'e}chal.
\newblock {\em Convex {{Analysis}} and {{Minimization Algorithms I}}}, volume
  305 of {\em Grundlehren Der Mathematischen {{Wissenschaften}}}.
\newblock {Springer Berlin Heidelberg}, {Berlin, Heidelberg}, 1993.

\bibitem[Hoe63]{H63}
W.~Hoeffding.
\newblock Probability inequalities for sums of bounded random variables.
\newblock {\em Journal of the American Statistical Association},
  58(301):13--30, 1963.

\bibitem[HV11]{HV11}
P.~Harremo{\"e}s and I.~Vajda.
\newblock On {{Pairs}} of {{$f$}}-{{Divergences}} and {{Their Joint Range}}.
\newblock {\em IEEE Transactions on Information Theory}, 57(6):3230--3235, June
  2011.

\bibitem[IT69]{IT69}
A.~D. Ioffe and V.~M. Tikhomirov.
\newblock On minimization of integral functionals.
\newblock {\em Functional Analysis and Its Applications}, 3(3):218--227, Jul
  1969.

\bibitem[Jam72]{J72}
G.~J.~O. Jameson.
\newblock Convex series.
\newblock {\em Mathematical Proceedings of the Cambridge Philosophical
  Society}, 72(1):37--47, July 1972.

\bibitem[JHW17]{JHW17}
J.~Jiao, Y.~Han, and T.~Weissman.
\newblock Dependence measures bounding the exploration bias for general
  measurements.
\newblock In {\em 2017 {{IEEE International Symposium}} on {{Information
  Theory}} ({{ISIT}})}, pages 1475--1479, June 2017.

\bibitem[Kem69]{K69}
J.~H.~B. Kemperman.
\newblock On the optimum rate of transmitting information.
\newblock {\em The Annals of Mathematical Statistics}, 40(6):2156--2177, 12
  1969.

\bibitem[KFG06]{KFG06}
M.~Khosravifard, D.~Fooladivanda, and T.~A. Gulliver.
\newblock Exceptionality of the {{Variational Distance}}.
\newblock In {\em Proceedings of the 2006 {{IEEE Information Theory
  Workshop}}}, pages 274--276, {Chengdu, China}, October 2006. {IEEE}.

\bibitem[KFG07]{KFG07}
M.~Khosravifard, D.~Fooladivanda, and T.~A. Gulliver.
\newblock Confliction of the {{Convexity}} and {{Metric Properties}} in
  $f$-{{Divergences}}.
\newblock {\em IEICE Transactions on Fundamentals of Electronics,
  Communications and Computer Sciences}, E90-A(9):1848--1853, September 2007.

\bibitem[Kis60]{K60}
J.~Kisy{\'n}ski.
\newblock {Convergence du type $\mathcal L$}.
\newblock {\em Colloquium Mathematicum}, 7(2):205--211, 1960.

\bibitem[KL51]{KL51}
S.~Kullback and R.~A. Leibler.
\newblock On {{Information}} and {{Sufficiency}}.
\newblock {\em Annals of Mathematical Statistics}, 22(1):79--86, March 1951.

\bibitem[K{\"o}n86]{K86}
H.~K{\"o}nig.
\newblock Theory and applications of superconvex spaces.
\newblock In {\em Aspects of Positivity in Functional Analysis ({{T\"ubingen}},
  1985)}, volume 122 of {\em North-{{Holland Math}}. {{Stud}}.}, pages 79--118.
  {North-Holland, Amsterdam}, 1986.

\bibitem[Kul59]{K59}
S.~Kullback.
\newblock {\em Information theory and statistics}.
\newblock Wiley, New York, 1959.

\bibitem[{Kul}67]{K67}
S.~{Kullback}.
\newblock A lower bound for discrimination information in terms of variation.
\newblock {\em IEEE Transactions on Information Theory}, 13(1):126--127,
  January 1967.

\bibitem[L{\'e}o01a]{L01_Inverse}
C.~L{\'e}onard.
\newblock Minimization of {{Energy Functionals Applied}} to {{Some Inverse
  Problems}}.
\newblock {\em Applied Mathematics and Optimization}, 44(3):273--297, January
  2001.

\bibitem[L{\'e}o01b]{L01}
C.~L{\'e}onard.
\newblock Minimizers of energy functionals.
\newblock {\em Acta Mathematica Hungarica}, 93(4):281--325, 2001.

\bibitem[L{\'e}o07]{Leonard07}
C.~L{\'e}onard.
\newblock Orlicz {{Spaces}}.
\newblock
  https://leonard.perso.math.cnrs.fr/papers/Leonard-Orlicz\%20spaces.pdf, April
  2007.

\bibitem[Lev68]{L68}
V.~L. Levin.
\newblock Some properties of support functionals.
\newblock {\em Mathematical Notes of the Academy of Sciences of the USSR},
  4(6):900--906, December 1968.

\bibitem[LZ56]{LZ56}
W.~Luxemburg and A.~Zaanen.
\newblock Conjugate spaces of {O}rlicz spaces.
\newblock {\em Indagationes Mathematicae (Proceedings)}, 59:217--228, 1956.

\bibitem[Mar86]{M86}
K.~Marton.
\newblock A simple proof of the blowing-up lemma (corresp.).
\newblock {\em IEEE Transactions on Information Theory}, 32(3):445--446, May
  1986.

\bibitem[Mor63]{M63}
T.~Morimoto.
\newblock {M}arkov processes and the {H}-theorem.
\newblock {\em Journal of the Physical Society of Japan}, 18(3):328--331, March
  1963.

\bibitem[Mor64]{M64}
J.~J. Moreau.
\newblock Sur la fonction polaire d'une fonction semi-continue
  sup{\'e}rieurement.
\newblock {\em Comptes rendus hebdomadaires des s{\'e}ances de l'Acad{\'e}mie
  des sciences}, 258:1128--1130, 1964.

\bibitem[MT50]{MT50}
M.~Morse and W.~Transue.
\newblock Functionals $f$ {{Bilinear Over}} the {{Product}} ${A}\times {B}$ of
  {{Two Pseudo}}-{{Normed Vector Spaces}}: {{II}}. {{Admissible Spaces A}}.
\newblock {\em Annals of Mathematics}, 51(3):576--614, 1950.

\bibitem[M{\"u}l97]{M97}
A.~M{\"u}ller.
\newblock Integral probability metrics and their generating classes of
  functions.
\newblock {\em Advances in Applied Probability}, 29(2):429--443, 1997.

\bibitem[Mus83]{M83}
J.~Musielak.
\newblock {\em Orlicz {{Spaces}} and {{Modular Spaces}}}, volume 1034 of {\em
  Lecture {{Notes}} in {{Mathematics}}}.
\newblock {Springer Berlin Heidelberg}, {Berlin, Heidelberg}, 1983.

\bibitem[Nak50]{N50}
H.~Nakano.
\newblock {\em Modulared Semi-Ordered Linear Spaces}.
\newblock Tokyo Mathematical Book Series,v. 1. {Maruzen Co.}, {Tokyo}, 1950.

\bibitem[NCM{\etalchar{+}}17]{NCMQW17}
R.~Nock, Z.~Cranko, A.~K. Menon, L.~Qu, and R.~C. Williamson.
\newblock $f$-{{GAN}}s in an information geometric nutshell.
\newblock In {\em Proceedings of the 31st International Conference on Neural
  Information Processing Systems}, NIPS’17, page 456–464, Red Hook, NY,
  USA, 2017. Curran Associates Inc.

\bibitem[NCT16]{NCT16}
S.~Nowozin, B.~Cseke, and R.~Tomioka.
\newblock $f$-{{GAN}}: Training generative neural samplers using variational
  divergence minimization.
\newblock In {\em Proceedings of the 30th International Conference on Neural
  Information Processing Systems}, NIPS’16, page 271–279, Red Hook, NY,
  USA, 2016. Curran Associates Inc.

\bibitem[NWJ08]{NWJ08}
X.~Nguyen, M.~J. Wainwright, and M.~I. Jordan.
\newblock Estimating divergence functionals and the likelihood ratio by
  penalized convex risk minimization.
\newblock In J.~C. Platt, D.~Koller, Y.~Singer, and S.~T. Roweis, editors, {\em
  Advances in Neural Information Processing Systems 20}, pages 1089--1096.
  Curran Associates, Inc., 2008.

\bibitem[NWJ10]{NWJ10}
X.~Nguyen, M.~J. Wainwright, and M.~I. Jordan.
\newblock Estimating divergence functionals and the likelihood ratio by convex
  risk minimization.
\newblock {\em IEEE Trans. Inf. Theor.}, 56(11):5847–5861, November 2010.

\bibitem[Pin60]{P60}
M.~S. Pinsker.
\newblock Informatsiya i informatsionnaya ustoichivost’ sluchainykh velichin
  i protsessov.
\newblock {\em Probl. Peredachi Inf.}, 7, 1960.

\bibitem[Pin64]{P63}
M.~S. Pinsker.
\newblock {\em Information and Information Stability of Random Variables and
  Processes}.
\newblock Holden-Day, 1964.
\newblock Translation of \citet{P60} by Amiel Feinstein.

\bibitem[R{\'e}n61]{R61}
A.~R{\'e}nyi.
\newblock On measures of entropy and information.
\newblock In J.~Neyman, editor, {\em Proceedings of the Fourth Berkeley
  Symposium on Mathematical Statistics and Probability}, volume~I, pages
  547--561. University of California Press, 1961.

\bibitem[Roc66]{R66}
R.~T. Rockafellar.
\newblock Level sets and continuity of conjugate convex functions.
\newblock {\em Transactions of the American Mathematical Society},
  123(1):46--63, 1966.

\bibitem[Roc68]{R68}
R.~T. Rockafellar.
\newblock Integrals which are convex functionals.
\newblock {\em Pacific J. Math.}, 24(3):525--539, 1968.

\bibitem[Roc71]{R71}
R.~T. Rockafellar.
\newblock Integrals which are convex functionals. {II}.
\newblock {\em Pacific J. Math.}, 39(2):439--469, 1971.

\bibitem[Roc76]{R76}
R.~T. Rockafellar.
\newblock Integral functionals, normal integrands and measurable selections.
\newblock In J.~P. Gossez, E.~J. Lami~Dozo, J.~Mawhin, and L.~Waelbroeck,
  editors, {\em Nonlinear Operators and the Calculus of Variations}, pages
  157--207, Berlin, Heidelberg, 1976. Springer Berlin Heidelberg.

\bibitem[RR91]{RR91}
M.~M. Rao and Z.~D. Ren.
\newblock {\em Theory of {{Orlicz}} Spaces}.
\newblock Number 146 in Monographs and Textbooks in Pure and Applied
  Mathematics. {M. Dekker}, {New York}, 1991.

\bibitem[RRGGP12]{RRGP12}
A.~Ruderman, M.~Reid, D.~García-García, and J.~Petterson.
\newblock Tighter variational representations of $f$-divergences via
  restriction to probability measures.
\newblock In J.~Langford and J.~Pineau, editors, {\em Proceedings of the 29th
  International Conference on Machine Learning (ICML '12)}, pages 671--678, New
  York, NY, USA, July 2012. Omnipress.

\bibitem[RW98a]{R98_14}
R.~T. Rockafellar and R.~J.-B. Wets.
\newblock Measurability.
\newblock In Berger et~al. \cite{R98}, chapter~14, pages 642--683.

\bibitem[RW98b]{R98}
R.~T. Rockafellar and R.~J.-B. Wets.
\newblock {\em Variational {{Analysis}}}, volume 317 of {\em Grundlehren Der
  Mathematischen {{Wissenschaften}}}.
\newblock {Springer Berlin Heidelberg}, {Berlin, Heidelberg}, 1998.

\bibitem[RW09]{RW09}
M.~D. Reid and R.~C. Williamson.
\newblock Generalised {P}insker inequalities.
\newblock In {\em {COLT} 2009 - The 22nd Conference on Learning Theory,
  Montreal, Quebec, Canada, June 18-21, 2009}, 2009.

\bibitem[RW11]{RW11}
M.~D. Reid and R.~C. Williamson.
\newblock Information, divergence and risk for binary experiments.
\newblock {\em J. Mach. Learn. Res.}, 12:731--817, 2011.

\bibitem[RZ20]{RZ20}
D.~{Russo} and J.~{Zou}.
\newblock How much does your data exploration overfit? {C}ontrolling bias via
  information usage.
\newblock {\em IEEE Transactions on Information Theory}, 66(1):302--323, Jan
  2020.

\bibitem[San57]{S57}
I.~N. Sanov.
\newblock On the probability of large deviations of random variables.
\newblock {\em Mat. Sb. (N.S.)}, 42(84):11–44, 1957.

\bibitem[Sas18]{S18}
I.~Sason.
\newblock On $f$-{{Divergences}}: {{Integral Representations}}, {{Local
  Behavior}}, and {{Inequalities}}.
\newblock {\em Entropy}, 20(5):383, May 2018.

\bibitem[SFG{\etalchar{+}}12]{SFGSL12}
B.~K. Sriperumbudur, K.~Fukumizu, A.~Gretton, B.~Schölkopf, and G.~R.~G.
  Lanckriet.
\newblock On the empirical estimation of integral probability metrics.
\newblock {\em Electronic Journal of Statistics}, 6:1550--1599, 2012.

\bibitem[SGF{\etalchar{+}}09]{SGFLS09}
B.~K. Sriperumbudur, A.~Gretton, K.~Fukumizu, G.~R.~G. Lanckriet, and
  B.~Schölkopf.
\newblock A note on integral probability metrics and $\phi$-divergences.
\newblock {\em CoRR}, abs/0901.2698v1, 2009.

\bibitem[Sim90]{S90}
S.~Simons.
\newblock The occasional distributivity of {$\circ$} over $\overset+e$ and the
  change of variable formula for conjugate functions.
\newblock {\em Nonlinear Analysis: Theory, Methods \& Applications},
  14(12):1111--1120, January 1990.

\bibitem[Sio58]{Sion58}
M.~Sion.
\newblock On general minimax theorems.
\newblock {\em Pacific Journal of Mathematics}, 8(1):171--176, March 1958.

\bibitem[Ste93]{S93}
E.~M. Stein.
\newblock {\em Harmonic Analysis: Real-Variable Methods, Orthogonality, and
  Oscillatory Integrals}.
\newblock Number~43 in Princeton {{Mathematical Series}}. {Princeton University
  Press, Princeton, NJ}, 1993.

\bibitem[SV16]{SV16}
I.~Sason and S.~Verd{\'u}.
\newblock $f$-{{Divergence Inequalities}}.
\newblock {\em IEEE Transactions on Information Theory}, 62(11):5973--6006,
  November 2016.

\bibitem[Szp34]{S34}
E.~Szpilrajn.
\newblock {Remarques sur les fonctions compl\`etement additives d'ensemble et
  sur les ensembles jouissant de la propri\'et\'e de Baire}.
\newblock {\em Fundamenta Mathematicae}, 22(1):303--311, 1934.

\bibitem[TV93]{TV93}
M.~{Teboulle} and I.~{Vajda}.
\newblock Convergence of best $\phi$-entropy estimates.
\newblock {\em IEEE Transactions on Information Theory}, 39(1):297--301, 1993.

\bibitem[Vaj70]{V70}
I.~Vajda.
\newblock Note on discrimination information and variation.
\newblock {\em IEEE Transactions on Information Theory}, 16(6):771--773,
  November 1970.

\bibitem[Vaj72]{V72}
I.~Vajda.
\newblock On the $f$-divergence and singularity of probability measures.
\newblock {\em Periodica Mathematica Hungarica}, 2(1):223--234, Mar 1972.

\bibitem[Vaj73]{V73}
I.~Vajda.
\newblock $\chi^\alpha$-divergence and generalized {{Fisher}}'s information.
\newblock In {\em Transactions of the {{Sixth Prague Conference}} on
  {{Information Theory}}, {{Statistical Decision Functions}}, {{Random
  Processes}}}, pages 873--886. {Academia}, {Prague}, 1973.

\bibitem[Val70]{Valadier70}
M.~Valadier.
\newblock Int{\'e}gration de convexes ferm{\'e}s notamment d'{\'e}pigraphes.
  {{Inf}}-convolution continue.
\newblock {\em Rev. Fran{\c c}aise Informat. Recherche Op{\'e}rationnelle},
  4(S{\'e}r. R-2):57--73, 1970.

\bibitem[Ver18]{V18}
R.~Vershynin.
\newblock {\em High-Dimensional Probability: An Introduction with Applications
  in Data Science}.
\newblock Number~47 in Cambridge Series in Statistical and Probabilistic
  Mathematics. {Cambridge University Press}, {Cambridge}, 2018.

\bibitem[Z{\u a}l02]{Z02}
C.~Z{\u a}linescu.
\newblock {\em Convex Analysis in General Vector Spaces}.
\newblock {World Scientific}, {River Edge, N.J. ; London}, 2002.

\end{thebibliography}

\let\oldcite\cite
\appendix
\section{Deferred proofs}
\label{sec:deferred_proofs}

In this section, for the sake of completeness, we include proofs of
results that follow from standard tools in convex analysis.

\subsection{Proof of \cref{lem:convex-bounded-cs-compact}}
\label{sec:convex-bounded-proof}

\Cref{lem:convex-bounded-cs-compact} follows immediately from \citet[Remark
1.9]{K86} stated in the general context of superconvex structures,
which applies to cs-compact sets by \citet[Example 1.6(0)]{K86}. For
completeness, we include a proof here in the language of topological
vector spaces.

First, a convex function which is upper bounded on a cs-closed set in the sense
of \citea{Jameson}{J72} satisfies an infinite-sum version of convexity called
cs-convexity (convex-series convexity) by \citea{Simons}{S90}.
\begin{lemma}
\label{lem:cs-closed-cs-convex}
	Let $C$ be a cs-closed subset of a real Hausdorff topological vector space
	and let $f:C\to\R$ be a convex function such that  $\sup_{x\in
	C}f(x)<\infty$. Then $f$ is cs-convex.
\end{lemma}
\begin{proof}
	Let $(\lambda_n)_{n\in\mathbb N}\in\R^{\mathbb N}$ be a sequence of real
	numbers such that $\sum_{i=0}^\infty \lambda_i = 1$ and $\lambda_n\geq 0$
	for all $n\in\mathbb N$.  Let $(x_n)_{n\in\mathbb N}\in C^{\mathbb N}$ be
	a sequence of elements in $C$ such that $r_0 \eqdef
	\sum_{i=0}^\infty\lambda_ix_i$ exists, and thus is in $C$ since $C$ is
	cs-closed. We wish to show that
	\begin{equation}
		f(r_0)\leq \liminf_{n\to\infty} \sum_{i=0}^n \lambda_if(x_i)
		\,.
	\label{eqn:cs-convex}
	\end{equation}
	Define for each $n\in\mathbb{N}$ the partial sums
	\[
		\Lambda_n \eqdef \sum_{i=0}^n\lambda_i
		\quad\text{and}\quad
		s_n \eqdef \Lambda_n^{-1}\sum_{i=0}^n \lambda_ix_i
		\,.
	\]
	If $\Lambda_n=1$ for some $n\in\mathbb N$ then \cref{eqn:cs-convex} is
	immediate from convexity. Otherwise, we have that for each
	$n\geq 0$,
	\[
		r_n\eqdef\frac{r_0-\Lambda_n\cdot s_n}{1-\Lambda_n}
		=\sum_{i=n+1}^\infty \frac{\lambda_i}{1-\Lambda_n}\cdot x_i
	\]
	is an element of $C$ since $C$ is cs-closed, so that
	by convexity of $f$
	\begin{align*}
		f(r_0)
		&\leq \Lambda_n\cdot  f(s_n) + (1-\Lambda_n)\cdot f(r_n)\\
		&\leq \sum_{i=0}^n \lambda_i f(x_i) + (1-\Lambda_n)\cdot \sup_{x\in C}f(x)
		\,.
	\end{align*}
	Since $\sup_{x\in C}f(x)<\infty$ and $\lim_{n\to\infty}\Lambda_n=1$
	by assumption, the previous inequality implies \cref{eqn:cs-convex}
	as desired.
\end{proof}

Second, cs-convex functions are necessarily bounded
below on cs-compact sets. Together with \cref{lem:cs-closed-cs-convex}, this
implies \cref{lem:convex-bounded-cs-compact}.
\begin{lemma}
\label{lem:cs-convex-cs-compact}
	Let $f:C\to\R$ be a cs-convex function on a cs-compact subset $C$ of a real
	Hausdorff topological vector space. Then $\inf_{x\in C}f(x)>-\infty$.
\end{lemma}
\begin{proof}
	We prove the contrapositive, that if $\inf_{x\in C}f(x)=-\infty$ then $f$
	is not cs-convex.  Indeed, if $\inf_{x\in C}f(x)=-\infty$, then for each
	$n\in\mathbb{N}$ there exists $x_n\in C$ with $f(x_n)\leq -4^n$. Since $C$
	is cs-compact, the element $\overline x\eqdef\sum_{i=1}^\infty 2^{-i}\cdot
	x_i$ exists and is in $C$. But then
	\[
		\liminf_{n\to\infty} \sum_{i=1}^n 2^{-i}\cdot f(x_i)\leq 
		\liminf_{n\to\infty} \sum_{i=1}^n 2^{-i}\cdot -4^i = -\infty
		< f(\overline x)\,,
	\]
	proving that $f$ is not cs-convex.
\end{proof}

\subsection{\titlephi/-cumulant generating function}
\label{sec:cgf-deferred}

\begin{lemma}[\Cref{lem:con-psi-properties} restated]
The function $\con\psi:x\mapsto\con\phi(x)-x$ is non-negative, convex, and inf-compact.
Furthermore, it satisfies $\con\psi(0)=0$, $\con\psi(x)\leq -x$ when $x\leq 0$, and
$\inter(\dom\con\psi)=\paren[\big]{-\infty,\phi'(\infty)}$.
\end{lemma}
\begin{proof}
	We have that $\con\psi(x)=\sup_{y\in\R}\paren*{y\cdot x-\phi(y+1)}
	=\sup_{y\in\R}\paren*{(y-1)\cdot x-\phi(y)}
	\allowbreak=-x+\sup_{y\in\R}\paren*{y\cdot x-\phi(y)} =\con\phi(x)-x$.
	Non-negativity of $\con\psi$ holds since $\con\psi(x)\geq 0\cdot x-
	\psi(0)=0$, and convexity and lower semicontinuity hold for any convex
	conjugate. For inf-compactness, we have since $0\in\inter\dom\phi$ by
	assumption that there exists $\alpha>0$ with $[-\alpha,\alpha]
	\subseteq\dom\psi$, so that $\con\psi(y) \geq \max\set*{\alpha\cdot y -
	\psi(\alpha), -\alpha\cdot y - \psi(-\alpha)} \geq \alpha\cdot \abs y -
	\max\set{\psi(\alpha), \psi(-\alpha)}$, so that the sublevel sets
	of $\con\psi$ are closed and bounded and thus compact.

	The claim about $\dom\psi$ is immediate from \cref{lem:domconj} since
$\dom\phi\subseteq\R_{\geq 0}$ implies $\phi'(-\infty)=-\infty$.  Finally,
$\dom\phi\subseteq\R_{\geq 0}$ also implies for $x\leq 0$ that
$\con\psi(x)=\sup_{y\geq -1}\paren[\big]{y\cdot x-\psi(y)} \allowbreak\leq
\sup_{y\geq -1}y\cdot x -\inf_{y\geq -1}\psi(y)=-x$ where the last equality is
because $\psi\geq 0$ and $x\leq 0$.
\end{proof}

\begin{proposition}[\Cref{prop:cgf-properties} restated]
For every $\sigma$-ideal $\Xi$, probability measure $\nu\in\M^1_c(\Xi)$, and
$g\in L^0(\Xi)$, $\cgfxi g\nu\Xi:\R\to\overline\R$ is
non-negative, convex, lsc, and satisfies $\cgfxi
g\nu\Xi(0)=0$.

	Furthermore, if $g$ is not $\nu$-essentially constant then $\cgfxi g\nu\Xi$
	is inf-compact. If there exists $c\in\R$ such that $g=c$ $\nu$-almost
	surely, then there exists $t>0$ (resp.\ $t<0$) such that $\cgfxi
	g\nu\Xi(t)>0$ if and only if $\phi'(\infty)<\infty$ and $\esssup_\Xi g>c$
	(resp.\ $\essinf_\Xi g<c$).
\end{proposition}

We prove this in steps, using the following important function:
\begin{definition}
	For every $\sigma$-ideal $\Xi$, probability measure
	$\nu\in\M^1_c(\Xi)$, and $g\in L^0(\Xi)$, define
	\begin{align*}
		\optimcgf g\nu\Xi(t,\lambda)
		&\eqdef
		\begin{cases}
			\ex*\nu{\con\psi(tg+\lambda)}&\text{if }\esssup_\Xi (tg+\lambda)\leq\phi'(\infty)\\
			+\infty&\text{otherwise}
		\end{cases}
		\\
		&=
		\ex*\nu{\con\psi(tg+\lambda)}
		+
		\begin{cases}
			0&\text{if }tg+\lambda\in[-\infty,\phi'(\infty)]\text{ $\Xi$-a.e.}\\
			+\infty&\text{otherwise}
		\end{cases}
		\,,
	\end{align*}
	so that $\cgfxi g\nu\Xi=\inf_{\lambda\in\R}\optimcgf g\nu\Xi(\cdotarg,\lambda)$.
\end{definition}
\begin{lemma}
	\label{lem:optim-cgf-lsc}
	For every $\sigma$-ideal $\Xi$, probability measure
	$\nu\in\M^1_c(\Xi)$, and $g\in L^0(\Xi)$, the function
	$\optimcgf g\nu\Xi$ is non-negative, convex, lsc, and the set
	$\set[\big]{\lambda\in\R\given \optimcgf g\nu\Xi(0,\lambda)=0}$
	is compact and contains $0$.
\end{lemma}
\begin{proof}
	Non-negativity of $\optimcgf g\nu\Xi$ is immediate from non-negativity
	of $\con\psi$. The function
	\[
		(t,\lambda)\mapsto 
		\begin{cases}
			0&\text{if }tg+\lambda\in[-\infty,\phi'(\infty)]\text{ $\Xi$-a.e.}\\
			+\infty&\text{otherwise}
		\end{cases}
	\]
	is convex and lsc since $[-\infty,\phi'(\infty)]$ is a closed interval.

	Similarly, the convexity of $\con\psi$ implies the convexity of
	$(t,\lambda)\mapsto \ex*\nu{\con\psi(tg+\lambda)}$.  Furthermore, by
	Fatou's lemma and since $\con\psi$ is lsc we have for every sequence
	$(t_n,\lambda_n)\to(t,\lambda)$ that
	\begin{displaymath}
		\liminf_{n\to\infty}\ex*\nu{\con\psi(t_ng+\lambda_n)}
		\geq\ex*\nu{\liminf_{n\to\infty}\con\psi(t_ng+\lambda_n)}
		\geq\ex*\nu{\con\psi(tg+\lambda)}\,,
	\end{displaymath}
	so that this function is also lower semicontinuous.

	Finally,
	$\set[\big]{\lambda\in\R\given \optimcgf g\nu\Xi(0,\lambda)=0}$
	is a sublevel set of a non-negative lsc function and so is closed,
	it contains $0$ since $\con\psi(0)=0$ and $\phi'(\infty)\geq0$,
	and is bounded since it is contained in the compact set
	$\set[\big]{\lambda\in\R\given\con\psi(\lambda)=0}$.
\end{proof}

\begin{lemma}
	\label{lem:optim-cgf-nonzero}
	For every $\sigma$-ideal $\Xi$, probability measure
	$\nu\in\M^1_c(\Xi)$, and $g\in L^0(\Xi)$, we have
	$\R_{\geq 0}\subseteq\set{t\in\R\given \exists \lambda\in\R\land \optimcgf g\nu\Xi(t,\lambda)=0}$
	if and only if $g$ is $\nu$-essentially constant and either  $\phi'(\infty)=\infty$
	or $\esssup_\Xi g=\esssup_\nu g$.
\end{lemma}
\begin{proof}
	If $g=c$ holds $\nu$-a.s.\ for some $c\in\R$ and either $\phi'(\infty)=\infty$
	or $\esssup_\Xi g=\esssup_\nu g=c$, then for all $t\geq 0$ we have
	$\optimcgf g\nu\Xi(t,-t\cdot c)=0$ since $tg+\lambda$ is $0$ $\nu$-a.s.\
	and at most $\phi'(\infty)\geq 0$ $\Xi$-a.e.

	Conversely, suppose $\R_{\geq 0}\subseteq\set{t\in\R\given
	\exists \lambda\in\R\land \optimcgf g\nu\Xi(t,\lambda)=0}$. Then for
	every $t\geq 0$ there is $\lambda\in\R$ such that
	$tg+\lambda\in\set{x\in\R
	\given\con\psi(x)=0}\subseteq [-\infty,\con\psi(\infty)]$ holds $\nu$-a.s.\
	and $tg+\lambda\in [-\infty, \phi'(\infty)]$ holds $\Xi$-a.e.
	Since $\con\psi$ is non-negative, convex, and inf-compact, the set $\set{x\in\R
	\given\con\psi(x)=0}$ is a compact interval $[a,b]$, and thus there is
	$\lambda\in\R$ such that $tg+\lambda\in[a,b]$ holds $\nu$-a.s.\ if
	only if $\abs t\cdot\paren[\big]{\esssup_\nu g-\essinf_\nu g}\leq
	b-a<\infty$. Thus, since this holds for all $t\in\R$, we have
	$\esssup_\nu g=\essinf_\nu g$, equivalently that $g=c$ holds
	$\nu$-a.s.\ for some $c\in\R$. Thus, the condition on $t$ reduces
	to the existence of $\lambda\in\R$ such that $tc+\lambda\in[a,b]$
	and $\esssup_\Xi tg+\lambda=tc+\lambda + t\cdot (\esssup_\Xi g-c)
	\leq\phi'(\infty)$. In particular, this implies that
	$a + t\cdot(\esssup_\Xi g-c)\leq \phi'(\infty)$ for all $t\geq 0$,
	which implies either $\phi'(\infty)=\infty$ or $\esssup_\Xi g\leq c
	=\esssup_\nu g$ as desired.
\end{proof}

We can finally prove \cref{prop:cgf-properties}.
\begin{proof}[Proof of \cref{prop:cgf-properties}]
	The main claim is immediate by applying standard results in convex
	analysis (e.g.~\citep[Propositions 1.17 and 3.32]{R98}) to
	\cref{lem:optim-cgf-lsc}. Furthermore, these results imply that
	$\set{t\in\R\given \cgfxi g\nu\Xi(t)=0}
	=\set{t\in\R\given \exists\lambda\in\R\land \optimcgf g\nu\Xi(t,\lambda)=0}$.
	
	For the supplemental claim, we have since $\cgfxi g\nu\Xi$
	is non-negative, convex, lsc, and $0$ at $0$ that it is inf-compact
	if and only if there exist $t_+>0$ and $t_-<0$ such that
	$\cgfxi g\nu\Xi(t_+),\cgfxi g\nu\Xi(t_-)>0$. The claimed
	characterization thus follows from applying \cref{lem:optim-cgf-nonzero}
	to $g$ and $-g$.
\end{proof}

\end{document}